\documentclass[letterpaper,12pt]{article}
\usepackage{diagbox}
\usepackage{array}
\usepackage{latexsym,amssymb,amsmath,bm}
\usepackage{mathtools}
\usepackage{amsthm}
\usepackage{lipsum,graphicx,multicol}
\usepackage{setspace,color}
\usepackage[sort&compress]{natbib}
\usepackage{tcolorbox}
\usepackage{subcaption}
\usepackage{float}
\usepackage[normalem]{ulem}
\usepackage{authblk}
\usepackage{leftidx}
\usepackage{dsfont}
\usepackage{geometry}
\usepackage{bbm}
\usepackage[title,titletoc]{appendix}
\usepackage{actuarialangle}
\usepackage{lscape}
\usepackage{afterpage}
\usepackage{relsize}
\usepackage[colorlinks=true,allcolors=blue, linkcolor=red]{hyperref}
\allowdisplaybreaks

\usepackage{lineno}
\usepackage{accents}

\def \P {\mathbb{P}}

\def \one {\bm{1}}

\def \E {\mathbb{E}}

\def \d {\mathrm{d}}

\def \RR {\mathbb{R}}

\def\Cov{{\mathrm{Cov}}}
\def\Var{{\mathrm{Var}}}
\def\Corr{{\mathrm{Corr}}}

\numberwithin{equation}{section} 
\newtheorem{theorem}{Theorem}[section]
\newtheorem{definition}[theorem]{Definition}

\newtheorem{example}[theorem]{Example}
\newtheorem{remark}[theorem]{Remark}
\newtheorem{proposition}[theorem]{Proposition}
\newtheorem{lemma}[theorem]{Lemma}
\newtheorem{corollary}[theorem]{Corollary}

\newcommand{\N}{{N}}
\newcommand{\M}{{M}}
\newcommand{\J}{{J}}

\newcommand{\I}{{I}}
\newcommand{\F}{\,\overline{\kern-0.175em F\kern-0em}}
\newcommand{\xb}{\,\overline{\kern-0.125em x\kern-0em}}
\newcommand{\Xb}{\,\overline{\kern-0.125em X\kern-0em}}


\definecolor{darkread}{rgb}{0.7, 0, 0}

\begin{document}
\onehalfspacing
\title{\Large\textbf{{Joint Exclusivity}}}

\author[]{\small Nawaf Mohammed \thanks{\url{nawaf.mohammed.ac@gmail.com}}}

\affil[]{\footnotesize }
\date{}
\maketitle
\vspace{-1cm}
\begin{abstract}
We introduce \emph{joint exclusivity} (JE), a relaxation of mutual exclusivity (ME) for non-negative random vectors that forbids only the simultaneous positivity of all $n$ components. Whereas ME requires the probabilities of positive outcomes to sum to at most one, a JE random vector with prescribed marginals $F_1,\ldots,F_n$ exists if and only if
\[
\sum_{i=1}^n \F_i(0)\le n-1.
\]
An inclusion--exclusion identity shows that JE laws are determined by their $(n-1)$-dimensional margins and decomposes aggregate stop-loss premiums into those of sub-portfolios. A single wrap-around construction attains the existence bound, minimises the number of claims in convex order and, in frequency--severity models with exchangeable severities, the aggregate loss, extending the safest-dependence property of ME beyond its admissible range. We develop canonical and distortion-based constructions, characterise when claim indicators and severities can be taken independent, and introduce an $m$-exclusivity hierarchy between ME and JE together with admissibility criteria for general families of claim patterns, explicit for hierarchical caps. Exclusivity extends to exceedances of arbitrary thresholds, where, at quantile levels, it becomes a property of the copula alone and is compatible with continuous losses. Finally, we identify the most dangerous exclusive structure in convex order for lines sharing a severity law, which yields sharp stop-loss and Tail Value-at-Risk (TVaR) bounds, and show that, for arbitrary marginals and thresholds, the worst-case TVaR under a cap on simultaneous exceedances is an explicit concave programme, concave in the cap.

\vspace{0.4cm}
{\sloppy\noindent\textit{Key words and phrases}: joint exclusivity; mutual exclusivity; counter-monotonicity; convex order; stop-loss premium; Tail Value-at-Risk; threshold exceedances; distortion risk measures; joint mixability.\par}
\end{abstract}

\newpage
\section{Introduction}
\label{sec:intro}

Dependence modelling lies at the heart of quantitative risk management in
insurance. For fixed marginal distributions, the joint law of a portfolio
determines the distribution of aggregate claims and, through it, premiums,
reinsurance prices and solvency capital \citep{Denuit2005, McNeil2015}. Since
the dependence structure is rarely known with precision, a large literature
studies the range of a risk functional over the \emph{Fr\'{e}chet class} of all
joint laws with prescribed marginals \citep{Embrechts2013a, Ruschendorf2013,
Puccetti2015}. The endpoints of this range are typically attained by extremal
dependence structures, which thus furnish sharp bounds on risk measures and
serve as benchmarks for sensitivity analysis \citep{Dhaene2002, Dhaene2002a,
Deelstra2011, Cheung2014a, Lauzier2023}. Practitioners, however, often know less about how losses move together than
about which can occur together: seasonal covers never all respond to one event,
a programme pays at most two benefits per event, business units are never all
hit hard in the same year. This paper develops the dependence theory such
information calls for.

In the bivariate setting, the positive extreme is \emph{co-monotonicity}
\citep{Dhaene2002, Nelson2010}: $X_{12}=(X_1, X_2)$ is co-monotonic if
\[
    X_{12} \stackrel{d}{=} \bigl(h_1(Z),\, h_2(Z)\bigr)
\]
for some random variable $Z$ and functions $h_1$, $h_2$ that are both
non-decreasing (or both non-increasing). Equivalently, its joint cumulative
distribution function (CDF) is the Fr\'{e}chet upper bound
$F_{12}(x_1, x_2) = \min\{F_1(x_1),\, F_2(x_2)\}$, where $F_1$ and $F_2$ are the
marginal CDFs \citep{Frechet1951}. At the opposite extreme lies
\emph{counter-monotonicity} \citep{Dhaene1999, Nelson2010}: $X_{12}$ is
counter-monotonic if $X_{12}\stackrel{d}{=}(h_1(Z),h_2(Z))$ with $h_1$
non-decreasing and $h_2$ non-increasing (or vice versa), and its joint CDF is
the Fr\'{e}chet lower bound
\[
    F_{12}(x_1, x_2) = \max\left\{F_1(x_1) + F_2(x_2) - 1,\; 0\right\}.
\]
In two dimensions, both Fr\'{e}chet bounds are valid joint CDFs for arbitrary
marginals, so both extremal structures exist and are unique in law.

The two notions part ways in higher dimensions. Let $\N = \{1,\dots, n\}$.
Co-monotonicity extends verbatim: $X_{\N}= (X_1, \dots, X_n)$ is co-monotonic
if $X_{\N}\stackrel{d}{=}(h_1(Z),\dots,h_n(Z))$ with all $h_i$ non-decreasing,
and the Fr\'{e}chet upper bound $F_{\N}(x_1, \dots, x_n) = \min_{i\in\N}
F_i(x_i)$ is a valid CDF for arbitrary marginals. Counter-monotonicity admits no
such extension. In dimension $n\ge3$, it is understood in the pairwise sense
\citep{DallAglio1972}, and pairwise counter-monotonicity is very restrictive: if
$X_1$ is continuously distributed and both $(X_1,X_2)$ and $(X_1,X_3)$ are
counter-monotonic, then $X_2$ and $X_3$ are non-increasing functions of $X_1$,
hence co-monotonic, which is incompatible with the counter-monotonicity of
$(X_2,X_3)$ unless one of them is degenerate. Correspondingly, the multivariate Fr\'{e}chet lower
bound
\[
    F_{\N}(x_1, \dots, x_n)
        = \max\left\{\sum_{i\in\N} F_i(x_i) - (n-1),\; 0\right\}
\]
is, in general, not a distribution function when $n \ge 3$; the restrictive
conditions under which it is one go back to \citet{DallAglio1972}; see also
\citet{Joe1997}. The bound is nevertheless \emph{pointwise} sharp: for each
fixed $x_\N$, some joint law with the given marginals attains it
\citep{Ruschendorf1981}; see also \citet[Theorem~2.10.13]{Nelson2010}. This
tension between global failure and pointwise attainability has given rise to
several competing notions of extremal negative dependence in higher dimensions.

The oldest and most tractable of these is \emph{mutual exclusivity} (ME)
\citep{Dhaene1999, Cheung2014}. Letting $\F$ denote the decumulative
distribution function (DDF), a non-negative random vector $X_{\N}$ is ME if
\begin{equation}
    \label{eq:MEdef}
    \F_{ij}(0,0) = 0, \qquad \text{for all } i<j,\quad i,j \in \N,
\end{equation}
that is, no two components are simultaneously positive, as for benefits
triggered by competing causes of loss. ME is not merely one extension of
counter-monotonicity among others: when at least three components are
non-degenerate, pairwise counter-monotonicity holds if and only if no two
components exceed their essential infima simultaneously, or no two fall below
their essential suprema simultaneously
\citetext{\citealp[Lemma~2]{DallAglio1972}; \citealp[Theorem~4.1]{Cheung2014}}. For non-negative risks with essential infima
zero, pairwise counter-monotonicity therefore coincides with ME or with its
mirror image formulated with essential suprema; see also \citet{Lauzier2023}
for a stochastic representation. An ME random vector
with prescribed marginals $F_1, \dots, F_n$ exists if and only if
\begin{equation}
    \label{eq:MEcondition}
    \sum_{i\in\N}\F_i(0) \le 1,
\end{equation}
in which case the Fr\'{e}chet lower bound is a valid CDF and is precisely the
CDF of the (unique in law) ME vector \citep{Dhaene1999}. ME is the safest
dependence structure: it minimises $\sum_{i\in\N}X_i$ in convex order, hence all
stop-loss premiums \citep{Dhaene1999}, and \citet{Cheung2014} showed that this
minimal-convex-sum property characterises ME. This tractability is the
principal appeal of ME; its price is condition~\eqref{eq:MEcondition}. As
$\sum_{i\in\N}\F_i(0)$ is the expected number of positive components,
\eqref{eq:MEcondition} allows a portfolio of losses with probability masses at
zero---typical of per-policy or per-line losses over a fixed period
\citep{Klugman2012a, Furman2010}---to produce at most one claim on average: ten
policies with a claim probability of $15\%$ each already violate it.

Notions of extremal negative dependence that accommodate much broader marginals
include joint mixability \citep{Wang2016}, which builds on complete mixability
\citep{Wang2011, Puccetti2012}, $d$-counter-monotonicity \citep{Lee2014,
Lee2017, Ahn2015} and $\Sigma$-counter-monotonicity, all surveyed by
\citet{Puccetti2015}. They require a sum of (possibly transformed) components
to be almost surely constant, or every partial sum to be counter-monotonic with
its complement, and are often difficult to construct explicitly, with
admissible marginals that are hard to identify: for joint mixability, duality
theory \citep{Rachev1998} characterises them only through an infinite family of
test functions, and explicit criteria are known only for special families, such
as uniform, monotone-density, unimodal-symmetric and certain elliptical
distributions \citep{Wang2016}. This
motivates the question: \emph{is there a notion of extremal negative dependence
that is strictly more flexible than ME, yet admits explicit constructions and a
tractable characterisation of its admissible marginals analogous
to~\eqref{eq:MEcondition}?} A natural answer emerges by relaxing the pairwise
requirement in~\eqref{eq:MEdef}: rather than prohibiting any \emph{pair} of
components from being simultaneously positive, we prohibit only the
simultaneous positivity of \emph{all} $n$ components.

\begin{definition}
\label{def:JE}
A non-negative random vector $X_{\N}$ is said to be \emph{jointly exclusive} (JE)
if
\begin{equation}
    \label{eq:JEdef}
    \F_{\N}(0,\dots,0) = 0,
    \text{ or equivalently, }
    \prod_{i\in\N}X_i=0\text{ almost surely},
\end{equation}
that is, its support is contained in $[0,\infty)^n \setminus (0,\infty)^n$.
\end{definition}

When $n = 2$, conditions \eqref{eq:MEdef} and \eqref{eq:JEdef} coincide. For
$n \ge 3$, ME implies JE but not conversely; see Figure~\ref{fig:MEvsJE}. Three
features, developed in Section~\ref{sec:preliminaries}, make JE a natural
object. First, JE is an \emph{$n$-th order} form of negative dependence: ME
constrains every pair of components, whereas JE constrains only the event
$\{\min_{i\in\N}X_i>0\}$ that all components are positive and leaves
lower-order interactions free, so that JE vectors may even have positive
pairwise correlations (Example~\ref{ex:JE_seasonal}). Second, JE is extremal
with respect to the lower Fr\'{e}chet bound on the survival side: while ME
forces the joint CDF to \emph{equal} the Fr\'{e}chet lower bound everywhere,
JE forces the joint DDF to vanish on $[0,\infty)^n$, which is possible only if
the Fr\'{e}chet lower bound $(\sum_{i\in\N}\F_i(x_i)-(n-1))_+$ for joint
survival functions vanishes there; JE vectors thus attain this bound on the
whole orthant. Accordingly, the
expected payout $\E[\min_{i\in\N}X_i]$ of a stylised cover responding only when
every line is loss-affected is zero under JE and maximal under co-monotonicity
(Lemma~\ref{lem:JE_elementary}). Third, JE constrains only the random
\emph{pattern} $\mathcal{A}=\{i\in\N:X_i>0\}$ of lines reporting a claim,
requiring $|\mathcal{A}|\le n-1$, while leaving the severities within each
pattern free apart from the marginals (Proposition~\ref{prop:JE_face}).

\begin{figure}[H]
\centering
\includegraphics[width=0.92\textwidth]{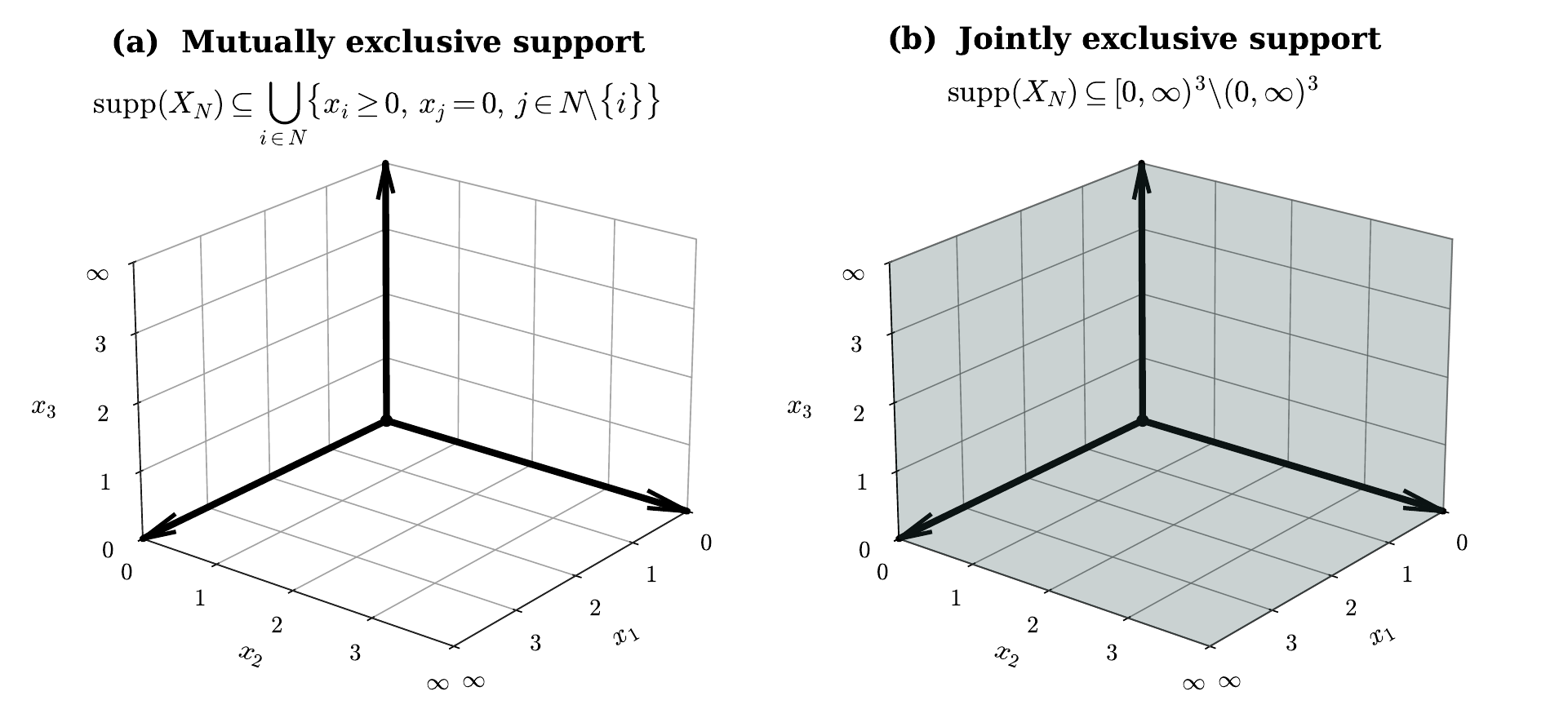}
\caption{Comparison of the supports of ME and JE random vectors for $n=3$.}
\label{fig:MEvsJE}
\end{figure}

Our contributions are fourfold. First, an inclusion--exclusion identity for
expectations (Theorem~\ref{thm:JE_IE}) characterises JE, represents its CDF,
characteristic function and Laplace--Stieltjes transform through the
$(n-1)$-dimensional margins (Proposition~\ref{prop:JECDF}), provides a criterion
for extending such margins to a JE law, and decomposes aggregate tails,
transforms and stop-loss premiums over sub-portfolios
(Corollary~\ref{cor:JE_consequences}). Second, a JE vector with marginals
$F_1,\dots,F_n$ exists if and only if $\sum_{i\in\N}\F_i(0)\le n-1$
(Theorem~\ref{thm:JE_Characterisation_1}); a single wrap-around construction
attains this bound and minimises the number of claims in convex order. With the
lines suitably ordered, the same construction settles the admissibility of
arbitrary hierarchical caps on the number of simultaneous claims
(Corollary~\ref{cor:laminar}). For general families of admissible claim
patterns, described by the groups of lines that may never claim together, we
give a polytope criterion, show that no pairwise check suffices, and extend the
inclusion--exclusion principle (Theorem~\ref{thm:families} and
Proposition~\ref{prop:K_IE}). Third, we describe the JE class
through nested stochastic representations: every JE vector draws a random
admissible pattern and independent face severities; in the canonical family,
each line's severity is independent of the pattern and the face copulas are
free; and the classical frequency--severity form, with claim indicators
independent of a single severity vector, holds exactly when the face laws are
margins of one law (Proposition~\ref{prop:FS_rep}). A generalised family with
distorted face severities is governed by a likelihood-ratio capacity, and an
$m$-exclusivity hierarchy between ME ($m=2$) and JE ($m=n$) has each level
characterised by $\sum_{i\in\N}\F_i(0)\le m-1$ and described by the same nested
representations, the level entering only through the law of the claim pattern
(Proposition~\ref{prop:K_rep}). Fourth, we extend exclusivity
to exceedances of retentions and Value-at-Risk (VaR) levels, where, at quantile
levels, it becomes a property of the copula alone and is compatible with
continuous losses, for which ME, JE and tail mutual exclusivity are impossible
(Proposition~\ref{prop:tail_copula} and Remark~\ref{rm:TME}); the tail mutual exclusivity of \citet{Cheung2017} is a special case whose
structures are TVaR-minimal at high levels, and we identify extremal structures
for
the aggregate loss. In the frequency--severity model with exchangeable
severities, the wrap-around structure is the safest for every value of
$\sum_{i\in\N}\F_i(0)$, which extends the safest-dependence result of
\citet{Dhaene1999} beyond the ME regime, and each exclusivity level caps the
most dangerous structure (Theorem~\ref{thm:safest_FS}). Over the whole
Fr\'{e}chet class, when the lines share a common severity law, each exclusivity
level admits an explicit most dangerous structure in convex order, a
co-monotonic sum of a few `super-risks' (Theorem~\ref{thm:cx_max}); together,
these results yield sharp stop-loss and Tail Value-at-Risk (TVaR) bounds, and
we show by examples that both extremes may fail to exist for heterogeneous
severities. For TVaR, homogeneity is not needed: under a cap on simultaneous
exceedances of arbitrary thresholds, the worst case is an explicit concave
programme, and it is concave in the cap, so that each tightening of the cap is
worth at least as much as the previous one (Theorem~\ref{thm:TVaR_threshold}).
The existence criterion can also be derived from the pointwise sharpness of the
lower Fr\'{e}chet--Hoeffding bound (Remark~\ref{rm:JECDFrm}), but the
constructive approach delivers considerably more than existence.

The remainder of the paper is organised as follows.
Section~\ref{sec:preliminaries} develops the structural properties of JE.
Section~\ref{sec:main_results} establishes the existence criterion and the
extremal construction, the canonical and generalised families, the
$m$-exclusivity hierarchy and general families of admissible claim patterns.
Section~\ref{sec:further_discussion} discusses threshold exceedances, aggregate
losses and their extremal structures, distortion risk measures and the
connection with joint mixability. Section~\ref{sec:conclusions} concludes.
Proofs not given in the text are collected in Appendix~\ref{app:proofs}.

%+++++++++++++++++++++++++++++++++++++++++++++++++++
\section{Preliminaries}
\label{sec:preliminaries}

\subsection{Notation and the geometry of the support}
\label{subsec:notation}

Throughout the paper, $\N=\{1,\dots,n\}$ with $n\ge2$, and subsets of $\N$ are
denoted by $\I$, $\J$ and $K$; sums over $\J\subsetneq\N$ run over all proper
subsets, including the empty set. The family of \emph{intermediate} subsets
$\M=\{\I\subseteq\N:2\le|\I|\le n-1\}$ and, for $i\in\N$, its subfamily
$\M_i=\{\I\in\M:i\in\I\}$ will index the faces of the support that carry joint
losses; thus $\M=\emptyset$ for $n=2$ and $|\M|=2^n-n-2$ in general. We consider
non-negative random vectors $X_{\N}=(X_1,\dots,X_n)$ with sub-vectors
$X_{\I}=(X_i,\,i\in\I)$, and write $x_{\N}=(x_1,\dots,x_n)\in\RR^n$ and
$x_{\I}=(x_i,\,i\in\I)$ for their deterministic counterparts and $0_\N$ for the
origin. For $\I\subseteq\N$, we denote by $F_\I$, $\F_\I$, $\varphi_\I$ and
$L_\I$ the CDF, DDF, characteristic function and Laplace--Stieltjes transform
(LST) of $X_\I$, so that, for instance, $\F_{\I}(x_\I)=\P(X_i>x_i,\,i\in\I)$,
$\varphi_\I(t_\I)=\E[\exp(\mathrm{i}\sum_{i\in\I}t_iX_i)]$ for
$t_\I\in\RR^{|\I|}$ and $L_{\I}(s_\I)=\E[\exp(-\sum_{i\in\I}s_iX_i)]$ for
$s_\I\in[0,\infty)^{|\I|}$, with $F_i=F_{\{i\}}$ and the conventions
$F_\emptyset=\F_\emptyset=\varphi_\emptyset=L_\emptyset\equiv1$. The quantity
\[
\tau=\sum_{i\in\N}\F_i(0)
\]
depends on the marginals only and plays a central role throughout. For
$\J\subseteq\N$, $x^{(\J)}\in\RR^n$ denotes the vector with components
$x^{(\J)}_i=x_i$ for $i\in\J$ and $x^{(\J)}_i=0$ otherwise, so that $X^{(\J)}$
switches off the components of $X_\N$ outside $\J$. We write
$S_\J=\sum_{j\in\J}X_j$, $S=S_\N$ and $S_\emptyset=0$ for aggregate losses,
$(t)_+=\max\{t,0\}$, and $\Pi_Z(d)=\E[(Z-d)_+]$, $d\in\RR$, for the stop-loss
transform of an integrable random variable $Z$. For a CDF $F$,
$F^{-1}(u)=\inf\{x\in\RR:F(x)\ge u\}$, $u\in(0,1)$, is its left-continuous
inverse. Finally, for integrable random variables, $Z_1\le_{\mathrm{cx}}Z_2$
(convex order) means that $\E f(Z_1)\le\E f(Z_2)$ for every convex function $f$
for which the expectations exist; for non-negative $Z_1,Z_2$ with equal means,
this is equivalent to $\Pi_{Z_1}(d)\le\Pi_{Z_2}(d)$ for all $d\ge0$
\citep{Denuit2005, Shaked2007}.

For $\I\subseteq\N$, the (relatively open) \emph{face}
\[
E_\I=\{x_\N\in[0,\infty)^n:\ x_i>0\text{ for }i\in\I,\ x_j=0\text{ for }j\in\N\setminus\I\}
\]
collects the points whose positive coordinates are exactly those in $\I$. The
$2^n$ faces partition $[0,\infty)^n$: $E_\emptyset=\{0_\N\}$ is the origin, the
faces $E_{\{i\}}$ are the open semi-axes, the faces $E_\I$ with $\I\in\M$ are the
\emph{intermediate faces}, and $E_\N=(0,\infty)^n$ is the interior.
Correspondingly, the \emph{active set} $\mathcal{A}=\{i\in\N:X_i>0\}$ satisfies
$\{X_\N\in E_\I\}=\{\mathcal{A}=\I\}$, and $|\mathcal{A}|=\sum_{i\in\N}
\mathds{1}\{X_i>0\}$ counts the positive components; if $X_i$ is the loss of the
$i$-th line of business over a period, $|\mathcal{A}|$ is the number of lines
reporting a claim. By Definition~\ref{def:JE} and Figure~\ref{fig:MEvsJE}, a
random vector is JE if and only if it places no mass on $E_\N$, that is, if its
support lies in the union of the $n$ coordinate hyperplanes $\{x_i=0\}$ within
$[0,\infty)^n$. Its law is thus singular with respect to $n$-dimensional
Lebesgue measure and, as shown below, is determined by its lower-dimensional
margins.

\begin{lemma}
\label{lem:JE_elementary}
Let $X_\N$ be a non-negative random vector with marginals $F_1,\dots,F_n$.
The following are equivalent: (a) $X_\N$ is JE; (b) $\min_{i\in\N}X_i=0$
almost surely; (c) $|\mathcal{A}|\le n-1$ almost surely; (d) $\F_\N(x_\N)=0$
for all $x_\N\in[0,\infty)^n$. Moreover, every non-negative $X_\N$ satisfies
\begin{equation}
\label{eq:count_identity}
\tau=\E\,|\mathcal{A}|
\qquad\text{and}\qquad
\Bigl(\sum_{i\in\N}\F_i(t)-(n-1)\Bigr)_+\le
\P\Bigl(\min_{i\in\N}X_i>t\Bigr)\le\min_{i\in\N}\F_i(t),\quad t\ge0,
\end{equation}
with the upper bound attained for all $t$ by co-monotonic vectors.
Consequently, $X_\N$ is JE if and only if $\tau\le n-1$ and the lower bound
in~\eqref{eq:count_identity} is attained at $t=0$ (equivalently, for all
$t\ge0$).
\end{lemma}

\begin{proof}
Non-negativity gives $\prod_{i\in\N}X_i=0$ if and only if
$\min_{i\in\N}X_i=0$, if and only if $\mathcal{A}\neq\N$, whence
(a)$\Leftrightarrow$(b)$\Leftrightarrow$(c); (a)$\Leftrightarrow$(d) follows
from $\F_\N(x_\N)\le\F_\N(0_\N)$ on $[0,\infty)^n$. The identity follows by
taking expectations in $|\mathcal{A}|=\sum_{i\in\N}\mathds{1}\{X_i>0\}$, and
the bounds are the Fr\'{e}chet bounds for $\P(\bigcap_{i\in\N}\{X_i>t\})$; for
$X_i=F_i^{-1}(U)$ with $U$ uniform on $(0,1)$, $\P(\min_{i\in\N}X_i>t)=
\P(U>\max_{i\in\N}F_i(t))=\min_{i\in\N}\F_i(t)$. If $X_\N$ is JE, (c) and the
identity give $\tau\le n-1$, so the lower bound vanishes for $t\ge0$ and is
attained by (d); the converse is immediate at $t=0$.
\end{proof}

Integrating the bounds in~\eqref{eq:count_identity} over $t\ge0$ gives
$0\le\E[\min_{i\in\N}X_i]\le\int_0^\infty\min_{i\in\N}\F_i(t)\,\d t$, with
equality on the left if and only if $X_\N$ is JE and on the right under
co-monotonicity: JE and co-monotonicity are antipodal for the payout
$\min_{i\in\N}X_i$ of an all-lines cover. The identity $\tau=\E|\mathcal{A}|$
also shows at a glance why $\tau\le n-1$ is necessary for JE, since a JE
portfolio never has all $n$ lines active; likewise, \eqref{eq:MEcondition},
that is, $\tau\le1$, is necessary for ME, under which at most one line is
active.

\subsection{An inclusion--exclusion principle}
\label{subsec:IE}

For a function $g:[0,\infty)^n\to\mathbb{C}$, a set $\J\subseteq\N$ and a point
$x_\N\in[0,\infty)^n$, the \emph{rectangular increment} of $g$ over the box
$[0_\N,x^{(\J)}]$ is
\begin{equation}
\label{eq:rect_increment}
\Delta^\J_{x_\N}g=\sum_{K\subseteq\J}(-1)^{|\J|-|K|}\,g\bigl(x^{(K)}\bigr),
\qquad\text{and we write }\Delta_{x_\N}g=\Delta^\N_{x_\N}g;
\end{equation}
for instance, $\Delta_{x_\N}g=g(x_1,x_2)-g(x_1,0)-g(0,x_2)+g(0,0)$ when $n=2$.
The following elementary facts drive the analysis of JE and, in
Section~\ref{subsec:aggregate}, of the whole $m$-exclusivity hierarchy.

\begin{lemma}
\label{lem:increments}
Let $g:[0,\infty)^n\to\mathbb{C}$ and $x_\N\in[0,\infty)^n$.
\begin{itemize}
\item[(a)] \emph{M\"obius inversion.} $g(x_\N)=\sum_{\J\subseteq\N}\Delta^\J_{x_\N}g$.
\item[(b)] \emph{Degenerate boxes.} $\Delta^\J_{x_\N}g=0$ whenever $x_k=0$ for
some $k\in\J$.
\item[(c)] \emph{Product functions.} If $g(x_\N)=\prod_{i\in\N}g_i(x_i)$ with
$g_i(0)=1$, $i\in\N$, then $\Delta^\J_{x_\N}g=\prod_{j\in\J}\bigl(g_j(x_j)-1\bigr)$.
\end{itemize}
\end{lemma}

\begin{proof}
(a) Exchanging the order of summation, the coefficient of $g(x^{(K)})$ in
$\sum_{\J\subseteq\N}\Delta^\J_{x_\N}g$ is
$\sum_{\J\supseteq K}(-1)^{|\J|-|K|}=(1-1)^{n-|K|}$, which vanishes unless
$K=\N$. (b) If $x_k=0$ with $k\in\J$, then
$x^{(K)}=x^{(K\cup\{k\})}$ for every $K\subseteq\J\setminus\{k\}$, and the terms
indexed by $K$ and $K\cup\{k\}$ cancel. (c) Since $g_i(0)=1$,
$g(x^{(K)})=\prod_{j\in K}g_j(x_j)$, and expanding
$\prod_{j\in\J}(g_j(x_j)-1)$ gives~\eqref{eq:rect_increment}.
\end{proof}

Isolating the term $\J=\N$ in~\eqref{eq:rect_increment} and invoking
Lemma~\ref{lem:increments}(b) yields the following principle: since a JE vector
lives precisely where the boxes $[0_\N,X_\N]$ are degenerate, the full
increment of any test function vanishes under JE.

\begin{theorem}
\label{thm:JE_IE}
Let $X_\N$ be a non-negative random vector and let
$g:[0,\infty)^n\to\mathbb{C}$ be measurable with $\E|g(X^{(\J)})|<\infty$ for
all $\J\subseteq\N$. Then
\begin{equation}
\label{eq:IE_master}
\E\bigl[g(X_\N)\bigr]
=\sum_{\J\subsetneq\N}(-1)^{n-1-|\J|}\,\E\bigl[g\bigl(X^{(\J)}\bigr)\bigr]
+\E\Bigl[\Delta_{X_\N}g\;\mathds{1}\bigl\{\min_{i\in\N}X_i>0\bigr\}\Bigr].
\end{equation}
In particular, if $X_\N$ is JE, then
\begin{equation}
\label{eq:IE_JE}
\E\bigl[g(X_\N)\bigr]=\sum_{\J\subsetneq\N}(-1)^{n-1-|\J|}\,
\E\bigl[g\bigl(X^{(\J)}\bigr)\bigr]
\end{equation}
for every such $g$. Conversely, if~\eqref{eq:IE_JE} holds for every bounded
measurable $g$, then $X_\N$ is JE.
\end{theorem}

\begin{proof}
Isolating the term $\J=\N$ in~\eqref{eq:rect_increment} gives
$g(x_\N)=\sum_{\J\subsetneq\N}(-1)^{n-1-|\J|}g(x^{(\J)})+\Delta_{x_\N}g$, and
the increment vanishes when $\min_{i\in\N}x_i=0$, by
Lemma~\ref{lem:increments}(b). Taking
expectations at $x_\N=X_\N$ gives~\eqref{eq:IE_master}; under JE, the last term
vanishes by Lemma~\ref{lem:JE_elementary}. Conversely, for
$g=\mathds{1}_{E_\N}$, every $X^{(\J)}$ with $\J\subsetneq\N$ has a vanishing
component, so~\eqref{eq:IE_JE} gives $\P(X_\N\in E_\N)=0$.
\end{proof}

For product test functions, Lemma~\ref{lem:increments}(c) evaluates the last
term of~\eqref{eq:IE_master} in closed form. Taking indicators, complex
exponentials and real exponentials yields, by one and the same argument,
representations of the CDF, the characteristic function and the LST of a JE
vector through its lower-dimensional margins.

\begin{proposition}
\label{prop:JECDF}
Let $X_\N$ be a non-negative random vector.
\begin{itemize}
\item[(a)] For all $x_\N\in[0,\infty)^n$,
$F_\N(x_\N)-\sum_{\J\subsetneq\N}(-1)^{n-1-|\J|}F_\J(x_\J)=(-1)^n\F_\N(x_\N)$.
Hence $X_\N$ is JE if and only if
\begin{equation}
\label{eq:JECDF}
F_\N(x_\N)=\sum_{\J\subsetneq\N}(-1)^{n-1-|\J|}F_\J(x_\J)
=\sum_{\J\subsetneq\N}(-1)^{|\J|}\F_\J(x_\J)
\end{equation}
holds at $x_\N=0_\N$, in which case it holds for all $x_\N\in[0,\infty)^n$.
\item[(b)] For all $t_\N\in\RR^n$,
$\varphi_\N(t_\N)-\sum_{\J\subsetneq\N}(-1)^{n-1-|\J|}\varphi_\J(t_\J)
=\E\bigl[\prod_{i\in\N}(\mathrm{e}^{\mathrm{i}t_iX_i}-1)\bigr]$. Hence
$X_\N$ is JE if and only if
\begin{equation}
\label{eq:JECF}
\varphi_\N(t_\N)=\sum_{\J\subsetneq\N}(-1)^{n-1-|\J|}\varphi_\J(t_\J),
\qquad t_\N\in\RR^n.
\end{equation}
\item[(c)] For all $s_\N\in[0,\infty)^n$,
\begin{equation}
\label{eq:LST_defect}
L_\N(s_\N)-\sum_{\J\subsetneq\N}(-1)^{n-1-|\J|}L_\J(s_\J)
=(-1)^n\,\E\Bigl[\prod_{i\in\N}\bigl(1-\mathrm{e}^{-s_iX_i}\bigr)\Bigr].
\end{equation}
Hence $X_\N$ is JE if and only if
\begin{equation}
\label{eq:JELST}
L_\N(s_\N)=\sum_{\J\subsetneq\N}(-1)^{n-1-|\J|}L_\J(s_\J)
\end{equation}
for some, equivalently for every, $s_\N\in(0,\infty)^n$.
\end{itemize}
\end{proposition}

Thus JE can be verified from the CDF or the LST at a single point. By
dominated convergence, the right-hand side of~\eqref{eq:LST_defect} tends to
$(-1)^n\P(\min_{i\in\N}X_i>0)$ as $\min_{i\in\N}s_i\to\infty$, so it also
quantifies the departure of a given vector from JE.

\begin{corollary}
\label{cor:JE_consequences}
\mbox{}
\begin{itemize}
\item[(a)] \emph{Determination.} Let $X_\N$ be JE with law $\mu$, and let
$\mu^{(\J)}$ denote the law of $X^{(\J)}$. Then
\begin{equation}
\label{eq:nu_def}
\mu=\sum_{\J\subsetneq\N}(-1)^{n-1-|\J|}\,\mu^{(\J)};
\end{equation}
hence $\mu$ is determined by the laws of $X_{\N\setminus\{k\}}$, $k\in\N$, and
two JE vectors with the same $(n-1)$-dimensional margins are equal in law.

\item[(b)] \emph{Extension criterion.} For $k\in\N$, let $\mu_{\N\setminus\{k\}}$
be a law on $[0,\infty)^{\N\setminus\{k\}}$, and suppose that
$\mu_{\N\setminus\{k\}}$ and $\mu_{\N\setminus\{l\}}$ induce the same law on
$[0,\infty)^{\N\setminus\{k,l\}}$ for all $k\neq l$. For $\J\subsetneq\N$, let
$\mu_\J$ be the induced law on $[0,\infty)^\J$ and $\mu^{(\J)}$ its image under
the embedding into $[0,\infty)^n$ that sets the coordinates outside $\J$ to
zero. Then a JE vector with $(n-1)$-dimensional margins
$\mu_{\N\setminus\{k\}}$, $k\in\N$, exists if and only if the signed measure
$\nu=\sum_{\J\subsetneq\N}(-1)^{n-1-|\J|}\mu^{(\J)}$ is non-negative, in which
case $\nu$ is its unique law.

\item[(c)] \emph{Aggregate claims.} If $X_\N$ is JE, then, for $z,s,d\ge0$,
\begin{equation}
\label{eq:IE_stoploss}
\begin{aligned}
\P(S>z)&=\sum_{\emptyset\neq\J\subsetneq\N}(-1)^{n-1-|\J|}\,\P(S_\J>z),
\qquad
L_S(s)=\sum_{\J\subsetneq\N}(-1)^{n-1-|\J|}\,L_{S_\J}(s),\\
\Pi_S(d)&=\sum_{\emptyset\neq\J\subsetneq\N}(-1)^{n-1-|\J|}\,\Pi_{S_\J}(d),
\end{aligned}
\end{equation}
where $L_{S_\J}$ is the LST of $S_\J$, and the last identity requires
$\E[X_i]<\infty$, $i\in\N$.
\end{itemize}
\end{corollary}

Corollary~\ref{cor:JE_consequences} has a transparent actuarial reading. By
(a), a JE portfolio carries no genuinely $n$-variate information. By (b), the
notoriously difficult compatibility problem for lower-dimensional margins
\citep{DallAglio1972, Joe1997} reduces, when the target law is JE, to an
explicit positivity check; for $n=2$, it reads $\nu(\{0_\N\})=
F_1(0)+F_2(0)-1\ge0$, which is exactly~\eqref{eq:MEcondition}. By (c), tails,
transforms and stop-loss premiums of the aggregate are alternating sums over
sub-portfolios. For $n=2$, this is the additivity $\Pi_S=\Pi_{X_1}+\Pi_{X_2}$
of stop-loss premiums under ME \citep{Dhaene1999}; for $n=3$,
\[
\Pi_S(d)=\Pi_{X_1+X_2}(d)+\Pi_{X_1+X_3}(d)+\Pi_{X_2+X_3}(d)
-\Pi_{X_1}(d)-\Pi_{X_2}(d)-\Pi_{X_3}(d),\qquad d\ge0,
\]
so that pricing an aggregate stop-loss cover on a trivariate JE portfolio
requires only bivariate computations; similarly, the LST identity permits the
aggregate distribution to be recovered by transform inversion from those of the
sub-portfolios \citep{Klugman2012a}.

\begin{remark}
\label{rm:JECDFrm}
\mbox{}
\begin{itemize}
\item[(1)] Representation~\eqref{eq:JECDF} illuminates the structural
distinction between JE and ME. Under ME, no two components are simultaneously
positive, so $\F_\J(x_\J)=0$ for every $\J\in\M$ and
$x_\J\in[0,\infty)^{|\J|}$. The interaction terms
\begin{equation}
\label{eq:MECDF_Example}
\sum_{\J\in\M}(-1)^{|\J|}\F_\J(x_\J)
\end{equation}
in the second sum of~\eqref{eq:JECDF} therefore vanish, and the CDF reduces to
$F_\N(x_\N)=1-\sum_{i\in\N}\F_i(x_i)=\sum_{i\in\N}F_i(x_i)-(n-1)$ on
$[0,\infty)^n$, which is the Fr\'{e}chet lower bound and is a valid CDF under
condition~\eqref{eq:MEcondition} \citep{Dhaene1999}. Under JE, by contrast, the
interaction terms~\eqref{eq:MECDF_Example} need not vanish, reflecting a richer
dependence structure in which sub-groups of up to $n-1$ components may be
simultaneously positive.

\item[(2)] By the distributional transform \citep{Ruschendorf2013}, every
$X_\N$ admits the representation $X_i=F_i^{-1}(U_i)$ almost surely, $i\in\N$,
where $U_\N$ follows a copula $C$ of $X_\N$. Since $F_i^{-1}(u)>0$ if and only
if $u>F_i(0)$, $X_\N$ is JE if and only if $C$ assigns no mass to the upper box
$\prod_{i\in\N}(F_i(0),1]$. JE is thus not a property of the copula alone, as it
depends on the atoms of the marginals at zero; in particular, the marginals of
a JE vector cannot all be continuous. Moreover, the lower Fr\'{e}chet--Hoeffding bound $W^n$ is
pointwise sharp: for every $u_\N\in[0,1]^n$ there is a copula $D$ with
$D(u_\N)=W^n(u_\N)=(\sum_{i\in\N}u_i-(n-1))_+$
\citetext{\citealp{Ruschendorf1981}; \citealp[Theorem~2.10.13]{Nelson2010}}. With $u_i=\F_i(0)$, $U'_\N\sim D$ and
$X_i=F_i^{-1}(1-U'_i)$, one obtains marginals $F_1,\dots,F_n$ and
$\P(\min_{i\in\N}X_i>0)\le D(u_\N)=(\tau-(n-1))_+$. The sufficiency part of
Theorem~\ref{thm:JE_Characterisation_1} below is thus a probabilistic
counterpart of the pointwise sharpness of $W^n$. This argument, however,
produces a single copula and is silent on the structure of the JE class, which
Section~\ref{sec:main_results} describes explicitly.
\end{itemize}
\end{remark}

\subsection{Face decomposition and structural properties}
\label{subsec:structure}

The partition of $[0,\infty)^n$ into faces yields an exhaustive stochastic
representation that separates the \emph{pattern} of positive components from
their \emph{severities}.

\begin{proposition}
\label{prop:JE_face}
Let $X_\N$ be non-negative with marginals $F_1,\dots,F_n$. For
$\I\subseteq\N$, let $p_\I=\P(\mathcal{A}=\I)$ and, if $p_\I>0$ and
$\I\neq\emptyset$, let $Q_\I$ be the conditional law of $X_\I$ given
$\{\mathcal{A}=\I\}$, a law on $(0,\infty)^{|\I|}$ with marginal survival
functions $\overline{Q}_{\I,i}$, $i\in\I$ (set $p_\I\overline{Q}_{\I,i}\equiv0$
if $p_\I=0$). Then $\sum_{\I\subseteq\N}p_\I=1$,
\begin{equation}
\label{eq:face_balance}
\F_i(x)=\sum_{\I\ni i}p_\I\,\overline{Q}_{\I,i}(x),\quad x\ge0,\ i\in\N,
\qquad\text{in particular}\qquad
\F_i(0)=\sum_{\I\ni i}p_\I,
\end{equation}
and $X_\N$ is JE if and only if $p_\N=0$. Conversely, for probability weights
$(p_\I)_{\I\subseteq\N}$ with $p_\N=0$ and laws $Q_\I$ on $(0,\infty)^{|\I|}$
satisfying~\eqref{eq:face_balance}, the mixture that, with probability $p_\I$,
draws $X_\I\sim Q_\I$ and sets $X_j=0$ for $j\notin\I$, is a JE law with
marginals $F_1,\dots,F_n$. By Sklar's theorem, each $Q_\I$ with $|\I|\ge2$ is
specified by its marginal survival functions and a survival copula
\citep{Sklar1959, Nelson2010}.
\end{proposition}

\begin{proof}
All claims follow from the law of total probability over the partition
$\{\mathcal{A}=\I\}$, $\I\subseteq\N$, since $\{X_i>x\}$ with $x\ge0$ forces
$i\in\mathcal{A}$ and $Q_\I$ is concentrated on $(0,\infty)^{|\I|}$.
\end{proof}

Thus JE constrains only the law of the pattern $\mathcal{A}$, which must never
equal $\N$, while the face severity laws $Q_\I$ are restricted solely by the
marginal balance~\eqref{eq:face_balance}. Equivalently, every non-negative
$X_\N$ has the \emph{random-face representation}
\begin{equation}
\label{eq:random_face}
X_i\stackrel{d}{=}\mathds{1}\{i\in\mathcal{A}'\}\,Y^{(\mathcal{A}')}_i,\qquad i\in\N,
\end{equation}
jointly in $i$, where $\mathcal{A}'$ is a random subset of $\N$ with law
$(p_\I)_{\I\subseteq\N}$ and, independently of $\mathcal{A}'$, the face vectors
$Y^{(\I)}\sim Q_\I$ are independent random vectors on $(0,\infty)^{|\I|}$; JE
requires only that $\mathcal{A}'\neq\N$ almost surely. In frequency--severity
terms \citep{Klugman2012a}, $X_i=B_iY_i$, where the claim indicators
$B_i=\mathds{1}\{X_i>0\}$ form a JE Bernoulli vector and the severities
$Y_i=Y^{(\mathcal{A}')}_i$ may depend on the realised pattern;
Proposition~\ref{prop:FS_rep} below characterises when they can be taken
independent of it. The families of Section~\ref{sec:main_results}
are obtained by specific choices of the face laws: the canonical family uses
$\overline{Q}_{\I,i}=\F_i/\F_i(0)$ on every face containing $i$, and the
generalised family replaces this ratio on the intermediate faces by a
distortion $G_i\circ\F_i$.

For marginals $F_1,\dots,F_n$ on $[0,\infty)$, let $\mathcal{J}(F_1,\dots,F_n)$
denote the set of JE laws with these marginals.

\begin{proposition}
\label{prop:JE_closure}
\mbox{}
\begin{itemize}
\item[(a)] \emph{Coverage modifications.} If $X_\N$ is JE and
$h_i:[0,\infty)\to[0,\infty)$ are measurable with $h_i(0)=0$, $i\in\N$, then
$(h_1(X_1),\dots,h_n(X_n))$ is JE. In particular, JE is preserved by per-risk
deductibles $(x-d_i)_+$, limits $\min\{x,u_i\}$, layers
$\min\{(x-d_i)_+,u_i\}$ and quota shares $\beta_ix$, where
$d_i,u_i,\beta_i\ge0$.
\item[(b)] \emph{Enlargement.} If $X_\N\ge0$ and $X_\I$ is JE for some
$\I\subseteq\N$ with $|\I|\ge2$, then $X_\N$ is JE; in particular, ME implies
JE.
\item[(c)] \emph{Compactness.} $\mathcal{J}(F_1,\dots,F_n)$ is a convex and
weakly compact face of the Fr\'{e}chet class of $F_1,\dots,F_n$: if a mixture
of two laws with these marginals is JE, then so is each of them. Hence, if
$\mathcal{J}(F_1,\dots,F_n)$ is non-empty and the marginals have finite means,
then for each $d\ge0$ the stop-loss premium $\Pi_S(d)$ attains its minimum and
maximum over $\mathcal{J}(F_1,\dots,F_n)$.
\end{itemize}
\end{proposition}

Part (a) ensures that JE of ground-up losses is inherited by insured, retained
and ceded losses under any per-risk coverage modification that pays nothing in
the absence of a loss \citep{Klugman2012a, Albrecher2017}. Part (b) cannot be
reversed: sub-vectors of a JE vector need not be JE, and this freedom is what
distinguishes JE from ME. Part (c) ensures that optimisation over the JE class,
such as the search for JE structures with extremal stop-loss premiums in
Section~\ref{subsec:aggregate}, is well posed. JE is, however, not preserved by
transformations that move the origin, such as adding a positive constant to
every component.

\begin{example}
\label{ex:JE_seasonal}
A single peril strikes once a year on a date $\Theta$ uniformly distributed on
the circular calendar $[0,1)$. Three seasonal covers respond only to events
within their eight-month windows (January--August, September--April and
May--December), namely $R_1=[0,\frac23)$, $R_2=[\frac23,1)\cup[0,\frac13)$
and $R_3=[\frac13,1)$, and pay positive severities $Y_1,Y_2,Y_3$ that are
independent of $\Theta$. The claims $X_i=Y_i\mathds{1}\{\Theta\in R_i\}$
satisfy $\P(X_i>0)=\frac23$ and $\P(X_i>0,X_j>0)=\frac13$ for $i\neq j$, since
any two windows share a four-month arc, whereas no date lies in all three
windows. Hence $X_\N$ is JE but not ME: $\tau=2$ violates
\eqref{eq:MEcondition} while attaining the JE bound $n-1=2$, and every date is
covered by exactly two windows, so $|\mathcal{A}|=2$ almost surely. JE is,
moreover, compatible with positive pairwise dependence: for a common severity
$Y_1=Y_2=Y_3=Y$ with $\E[Y^2]<\infty$,
$\Cov(X_i,X_j)=\frac13\E[Y^2]-\frac49(\E[Y])^2$, $i\neq j$, which is positive
whenever the coefficient of variation of $Y$ exceeds $1/\sqrt3$, e.g., for
exponential severities. All three pairs are then positively correlated, even
though the three covers never respond together. With independent unit-mean
exponential severities, every date is covered twice, so $S$ follows a gamma
law with shape two; Corollary~\ref{cor:JE_consequences}(c) recovers this from
the pairs alone, as each $X_i+X_j$ is such a gamma variable with probability
$\frac13$ and a unit exponential one with probability $\frac23$, whence
$\sum_{i<j}\Pi_{X_i+X_j}-\sum_i\Pi_{X_i}$ is the stop-loss transform of the
gamma law. More generally, windows of
lengths $w_1,\dots,w_n\in[0,1]$ laid end to end around the calendar cover each
date either $\lfloor\sum_{i\in\N}w_i\rfloor$ or $\lceil\sum_{i\in\N}w_i\rceil$
times, so that the resulting covers are JE whenever $\sum_{i\in\N}w_i\le n-1$;
Section~\ref{subsec:existence} builds on this mechanism.
\end{example}

%+++++++++++++++++++++++++++++++++++++++++++++++++++
\section{Main Results}
\label{sec:main_results}

Section~\ref{sec:preliminaries} showed that JE is a constraint on the random
pattern $\mathcal{A}$ of positive components alone
(Proposition~\ref{prop:JE_face}) and that $\tau\le n-1$ is necessary for JE
(Lemma~\ref{lem:JE_elementary}). This section shows that the condition is also
sufficient, by means of an explicit and, in a precise sense, extremal
construction (Section~\ref{subsec:existence}). It then develops a rich
parametric family of JE laws in which mass and dependence are specified face by
face (Section~\ref{subsec:canonical}), a generalisation in which the
within-face severities are reshaped (Section~\ref{subsec:GJE}), a hierarchy
of exclusivity levels interpolating between ME and JE
(Section~\ref{subsec:mexclusive}), and general families of admissible claim
patterns (Section~\ref{subsec:families}).

\subsection{Existence and an extremal construction}
\label{subsec:existence}

The construction rests on a covering property of arcs laid end to end around a
circle, already encountered in Example~\ref{ex:JE_seasonal}.

\begin{lemma}
\label{lem:wraparound}
Let $w_1,\dots,w_n\ge0$ with total length $\|w\|_1=\sum_{i\in\N}w_i$, and let
$\omega\ge\max_{i\in\N}w_i$ with $\omega>0$. Set $\sigma_0=0$,
$\sigma_i=w_1+\dots+w_i$, and
\[
R_i=\bigl\{\theta\in[0,\omega):\ \theta+k\omega\in[\sigma_{i-1},\sigma_i)
\text{ for some } k\in\mathbb{Z}\bigr\},\qquad i\in\N,
\]
so that $R_1,\dots,R_n$ are consecutive arcs of lengths $w_1,\dots,w_n$ on a
circle of circumference $\omega$. Then each $R_i$ has Lebesgue measure $w_i$,
and, for every block $\{a,\dots,b\}$ of consecutive indices,
$1\le a\le b\le n$, the number of arcs of the block covering $\theta$ satisfies
\begin{equation}
\label{eq:wrap_block}
\sum_{i=a}^b\mathds{1}\{\theta\in R_i\}\in
\Bigl\{\Bigl\lfloor\frac{\sigma_b-\sigma_{a-1}}{\omega}\Bigr\rfloor,
\Bigl\lceil\frac{\sigma_b-\sigma_{a-1}}{\omega}\Bigr\rceil\Bigr\},
\qquad \theta\in[0,\omega).
\end{equation}
In particular, the number $\kappa(\theta)=\sum_{i\in\N}\mathds{1}\{\theta\in R_i\}$
of arcs covering $\theta$ satisfies
\begin{equation}
\label{eq:wrap_count}
\kappa(\theta)\in\bigl\{\lfloor \|w\|_1/\omega\rfloor,\ \lceil \|w\|_1/\omega\rceil\bigr\},
\qquad \theta\in[0,\omega),
\end{equation}
and $\kappa\equiv\|w\|_1/\omega$ whenever $\|w\|_1/\omega$ is an integer.
\end{lemma}

\begin{proof}
Each interval $[\sigma_{i-1},\sigma_i)$ has length $w_i\le\omega$, so it
contains at most one point of the lattice $\theta+\omega\mathbb{Z}$. Hence
reduction modulo $\omega$ maps it injectively onto $R_i$, which therefore has
measure $w_i$, and $\theta\in R_i$ if and only if $[\sigma_{i-1},\sigma_i)$
contains a point of $\theta+\omega\mathbb{Z}$. As the intervals of a block tile
$[\sigma_{a-1},\sigma_b)$, the left-hand side of~\eqref{eq:wrap_block} equals
the number of points of $\theta+\omega\mathbb{Z}$ in $[\sigma_{a-1},\sigma_b)$,
that is, the number of integers in a half-open interval of length
$(\sigma_b-\sigma_{a-1})/\omega$, which is the integer part of this length or
its ceiling. The case $a=1$, $b=n$ gives~\eqref{eq:wrap_count}.
\end{proof}

\begin{theorem}
\label{thm:JE_Characterisation_1}
Let $F_1,\dots,F_n$ be marginal distributions on $[0,\infty)$. A JE random
vector $X_{\N}$ with marginals $F_1,\dots,F_n$ exists if and only if
\begin{equation}
\label{eq:JECondition}
\sum_{i\in\N}\F_i(0)\le n-1.
\end{equation}
Moreover, there is a random vector $X^*_\N$ with marginals $F_1,\dots,F_n$ whose
number of positive components $\kappa^*=|\mathcal{A}(X^*_\N)|$ takes only the
values $\lfloor\tau\rfloor$ and $\lceil\tau\rceil$, and which satisfies
$\kappa^*\le_{\mathrm{cx}}|\mathcal{A}(X_\N)|$ for every random vector $X_\N$
with marginals $F_1,\dots,F_n$. Under~\eqref{eq:JECondition}, $X^*_\N$ is JE.
\end{theorem}

\begin{proof}
Necessity is Lemma~\ref{lem:JE_elementary}: under JE,
$\tau=\E|\mathcal{A}|\le n-1$. For the remaining claims, apply
Lemma~\ref{lem:wraparound} with $w_i=\F_i(0)\le1$ and $\omega=1$. Let $\Theta$
be uniform on $[0,1)$ and, independently of $\Theta$, let $Y_\N$ be any random
vector with $\P(Y_i>y)=\F_i(y)/\F_i(0)$, $y\ge0$, for each $i$ with
$\F_i(0)>0$ (and $Y_i\equiv1$ otherwise). Set
$X^*_i=Y_i\,\mathds{1}\{\Theta\in R_i\}$. Then
$\P(X^*_i>x)=\P(\Theta\in R_i)\,\P(Y_i>x)=\F_i(x)$ for $x\ge0$, so $X^*_\N$ has
the prescribed marginals, and $\mathcal{A}(X^*_\N)=\{i\in\N:\Theta\in R_i\}$,
so $\kappa^*=\kappa(\Theta)\in\{\lfloor\tau\rfloor,\lceil\tau\rceil\}$
by~\eqref{eq:wrap_count}. Under~\eqref{eq:JECondition},
$\lceil\tau\rceil\le n-1$, so $X^*_\N$ is JE by Lemma~\ref{lem:JE_elementary}(c).
Finally, let $X_\N$ have marginals $F_1,\dots,F_n$, so that
$\kappa=|\mathcal{A}(X_\N)|$ is integer-valued with $\E\kappa=\tau=\E\kappa^*$.
For convex $f$, let $\tilde f$ be the piecewise linear interpolation of $f$ at
the integers; $\tilde f$ is convex, coincides with $f$ on $\mathbb{Z}$ and is
affine on $[\lfloor\tau\rfloor,\lceil\tau\rceil]$. By Jensen's inequality,
$\E f(\kappa)=\E\tilde f(\kappa)\ge\tilde f(\tau)=\E\tilde f(\kappa^*)
=\E f(\kappa^*)$.
\end{proof}

The vector $X^*_\N$ is the multivariate analogue of Example~\ref{ex:JE_seasonal}:
a single uniform ``clock'' $\Theta$ determines which lines are hit, windows of
lengths $\F_1(0),\dots,\F_n(0)$ being laid end to end around the unit circle.
Among all dependence structures with the given marginals, it produces the least
variable number of claims in the sense of convex order
\citep{Denuit2005, Shaked2007}; in particular,
$\Var|\mathcal{A}(X_\N)|\ge(\tau-\lfloor\tau\rfloor)(\lceil\tau\rceil-\tau)$
for every $X_\N$ with marginals $F_1,\dots,F_n$, with equality for $X^*_\N$.
The same vector also attains the strongest level of exclusivity compatible with
the marginals (Section~\ref{subsec:mexclusive}).

Condition~\eqref{eq:JECondition}, that is, $\tau\le n-1$, is far less
restrictive than the ME condition $\tau\le1$, and the gap widens linearly in
$n$. In terms of the atoms at zero, JE requires $\sum_{i\in\N}F_i(0)\ge1$,
whereas ME requires $\sum_{i\in\N}F_i(0)\ge n-1$; for a portfolio of $n$ lines
with a common claim probability $q$, JE is admissible whenever $q\le(n-1)/n$,
while ME requires $q\le1/n$. Under ME, a component that is positive almost surely forces all other
components to vanish, so ME vectors with at least two non-zero components
require $F_i(0)>0$ for every $i\in\N$. Under JE, by contrast, $n-2$ marginals may be supported on
$(0,\infty)$, in which case~\eqref{eq:JECondition} reduces to
$F_{i_1}(0)+F_{i_2}(0)\ge1$ for the two remaining marginals.

\subsection{Canonical JE: face-by-face specification}
\label{subsec:canonical}

The vector $X^*_\N$ uses a single severity vector $Y_\N$ for all patterns. A far
richer family is obtained by specifying the mass and the dependence on each
face separately. In the notation of Proposition~\ref{prop:JE_face}, we let every
face involving component $i$ carry the zero-truncated survival function
$\F_i/\F_i(0)$, assign mass $p_\I\ge0$ and a survival copula $C_\I$ on
$[0,1]^{|\I|}$ \citep{Nelson2010} to each intermediate face $E_\I$, $\I\in\M$,
and let the semi-axes and the origin absorb the remaining marginal mass. For
face weights $(p_\I)_{\I\in\M}$, we write
\[
\pi_i=\sum_{\I\in\M_i}p_\I,\qquad i\in\N,
\]
for the total mass carried by the intermediate faces involving component $i$.

\begin{definition}
\label{def:JECanonical}
A random vector $X_{\N}$ is called a \emph{canonical JE random vector} with
face weights $(p_\I)_{\I\in\M}$ and face copulas $(C_\I)_{\I\in\M}$ if
$\P(X_\N\in E_\N)=0$ and
\begin{align}
\P\left(X_i>x_i,\;i\in \I;\;X_j=0,\;j\in\N\setminus \I\right)
&= p_{\I}\,C_{\I}\left(\frac{\F_i(x_i)}{\F_i(0)},\;i\in\I\right),
\quad \I\in\M,\ x_\I\in[0,\infty)^{|\I|},
\label{eq:JEProbFaces}\\
\P\left(X_i>x_i;\;X_j=0,\;j\in\N\setminus\{i\}\right)
&=\F_i(x_i)\left(1-\frac{\pi_i}{\F_i(0)}\right),
\quad i\in\N,\ x_i\in[0,\infty),
\label{eq:JEProbAxes}\\
\P\left(X_i=0,\;i\in\N\right)
&=1-\sum_{i\in\N}\F_i(0)+\sum_{\I\in\M}(|\I|-1)\,p_{\I},
\label{eq:JEProbOrigin}
\end{align}
where the face weights $p_\I\ge0$ satisfy
\begin{align}
\sum_{\I\in \M_i}p_{\I}&\le\, \F_i(0), \qquad i\in\N,
\label{eq:JE_parameters_1}\\
\sum_{\I\in\M}(|\I|-1)\,p_{\I}&\ge \sum_{i\in\N}\F_i(0)-1.
\label{eq:JE_parameters_2}
\end{align}
We denote by $\mathcal{P}(F_1,\dots,F_n)$ the set of such face weights. (If
$\F_i(0)=0$, then~\eqref{eq:JE_parameters_1} forces $p_\I=0$ for all
$\I\in\M_i$ and the right-hand side of~\eqref{eq:JEProbAxes} is read as zero,
so the ratios $\F_i/\F_i(0)$ need only be evaluated when $\F_i(0)>0$.)
\end{definition}

\begin{proposition}
\label{prop:canonical_valid}
Let $F_1,\dots,F_n$ be marginal distributions on $[0,\infty)$.
\begin{itemize}
\item[(a)] For every $(p_\I)_{\I\in\M}\in\mathcal{P}(F_1,\dots,F_n)$ and every
family of copulas $(C_\I)_{\I\in\M}$, \eqref{eq:JEProbFaces}--\eqref{eq:JEProbOrigin}
define the law of a JE random vector with marginals $F_1,\dots,F_n$.
\item[(b)] $\mathcal{P}(F_1,\dots,F_n)\neq\emptyset$ if and only
if~\eqref{eq:JECondition} holds.
\item[(c)] A JE random vector $X_\N$ with marginals $F_1,\dots,F_n$ is canonical
if and only if, for every $i\in\N$, the severity $X_i$ is conditionally
independent of the active set $\mathcal{A}$ given $\{X_i>0\}$, that is,
\begin{equation}
\label{eq:canonical_char}
\P(X_i>x,\ \mathcal{A}=\I)=\P(\mathcal{A}=\I)\,\frac{\F_i(x)}{\F_i(0)},
\qquad x\ge0,\ \I\ni i.
\end{equation}
\end{itemize}
\end{proposition}

By Proposition~\ref{prop:canonical_valid}(c), canonical JE is the JE analogue
of the classical independence between claim frequency and claim severity: the
size of a loss on line $i$ does not depend on which other lines are hit.
Canonical JE thus separates two layers of modelling. The face weights
$(p_\I)$, which range over the polytope $\mathcal{P}(F_1,\dots,F_n)$ determined
by the marginal atoms alone, govern the frequency pattern: how often each
sub-group of lines is jointly hit. The copulas $(C_\I)$ govern the severity
dependence within each sub-group; they are unconstrained by the marginals and
may be chosen independently across faces, so that $C_{\I_1}$ need not be a
margin of $C_{\I_2}$ when $\I_1\subset\I_2$. The classical frequency--severity
model, in which the claim indicators are independent of a single severity
vector, is the special case in which all face laws are margins of one law.

\begin{proposition}[Frequency--severity representation]
\label{prop:FS_rep}
Let $X_\N$ be non-negative, with pattern law $(p_\I)$ and face laws $(Q_\I)$ as
in Proposition~\ref{prop:JE_face}. There exist a random vector
$B_\N\in\{0,1\}^n$ and a random vector $Y_\N$ on $(0,\infty)^n$, independent of
$B_\N$, such that $X_\N\stackrel{d}{=}(B_iY_i,\,i\in\N)$ if and only if some
law $\nu$ on $(0,\infty)^n$ has $Q_\I$ as its $\I$-margin for every non-empty
$\I$ with $p_\I>0$. In that case, $\P(B_\N=\mathbf{1}_\I)=p_\I$, one may take
$Y_\N\sim\nu$, and, if $X_\N$ is JE, it is canonical.
\end{proposition}

\begin{proof}
If $X_\N\stackrel{d}{=}(B_iY_i)$ with $B_\N$ independent of $Y_\N>0$, the
active set is $\{i:B_i=1\}$, so $p_\I=\P(B_\N=\mathbf{1}_\I)$, and, given
$B_\N=\mathbf{1}_\I$, the positive components $Y_\I$ have the $\I$-margin of
the law $\nu$ of $Y_\N$. Conversely, with $B_\N=\mathbf{1}_{\mathcal{A}'}$ and
$Y_\N\sim\nu$ independent, $(B_iY_i)$ has pattern law $(p_\I)$ and face laws
$(Q_\I)$, hence the law of $X_\N$ by~\eqref{eq:random_face}. Finally,
\eqref{eq:face_balance} gives $\F_i=\F_i(0)\,\overline{\nu}_i$ for the $i$-th
margin $\nu_i$ of $\nu$, so every face carries the severity law
$\F_i/\F_i(0)$ of line $i$, which is~\eqref{eq:canonical_char}.
\end{proof}

The representations are therefore nested: every JE vector has the random-face
representation~\eqref{eq:random_face}; canonical vectors are those in which each
line's severity does not depend on the pattern; and frequency--severity vectors
are those whose face laws are, in addition, margins of a single law. The last
inclusion is strict: the face laws must be consistent, $Q_\I$ being the
$\I$-margin of $Q_\J$ whenever $\I\subseteq\J$ both carry mass, and even
consistent faces may fail to be margins of one law, as in
Example~\ref{ex:JE_canonical_example} below. The wrap-around vector $X^*_\N$ of
Theorem~\ref{thm:JE_Characterisation_1} is of frequency--severity form, with a
single uniform clock driving the indicators. For canonical vectors,
Proposition~\ref{prop:JECDF} yields the CDF in closed form.

\begin{corollary}
\label{cor:JE_Canonical_CDF_phi}
The CDF of a canonical JE random vector with face weights $(p_\I)_{\I\in\M}$
and face copulas $(C_\I)_{\I\in\M}$ is given, for $x_{\N}\in[0,\infty)^n$, by
\begin{equation}
\label{eq:JE_Canonical_CDF}
F_{\N}(x_{\N})=
1-\sum_{i\in\N}\F_i(x_i)
+\sum_{\I\in\M}(-1)^{|\I|}
\sum_{\substack{\J\in\M\\ \I\subseteq \J}}
p_{\J}\,C_{\J}\left(\frac{\F_i(x_i)}{\F_i(0)},\,i\in\I;\;1,\,j\in \J\setminus \I\right),
\end{equation}
and $F_\N(x_\N)=0$ otherwise. Conversely, the law of a random vector with this
CDF is canonical JE with the same parameters.
\end{corollary}

\begin{remark}
\label{rm:JE_Canonical_Remark}
\mbox{}
\begin{itemize}
\item[(1)] If $p_\I=0$ for all $\I\in\M$, no two components are simultaneously
positive and the canonical vector is ME; \eqref{eq:JE_parameters_2} then
reduces to the ME condition $\tau\le1$. Conversely, whenever $\tau>1$,
\eqref{eq:JE_parameters_2} forces some intermediate face to carry positive mass.

\item[(2)] The $i$-th semi-axis carries no mass if and only if
$\pi_i=\F_i(0)$. Imposing this for all $i\in\N$ gives $n$ linear equations in
$|\M|=2^n-n-2$ unknowns, which is under-determined for $n\ge4$. For $n=3$, the
unique solution is $p_{\N\setminus\{j\}}=\frac12\bigl(\tau-2\F_j(0)\bigr)$,
$j\in\N$, which is admissible if and only if, in addition
to~\eqref{eq:JECondition}, $\F_1(0),\F_2(0),\F_3(0)$ satisfy the triangle
inequality $2\max_{j\in\N}\F_j(0)\le\tau$; the origin then carries mass
$1-\tau/2$.

\item[(3)] The origin carries no mass if and only if equality holds
in~\eqref{eq:JE_parameters_2}.

\item[(4)] At the boundary $\tau=n-1$ of~\eqref{eq:JECondition}, every JE
vector---canonical or not---satisfies $|\mathcal{A}|\le n-1$ and
$\E|\mathcal{A}|=n-1$, hence $|\mathcal{A}|=n-1$ almost surely: exactly one
component vanishes. By~\eqref{eq:face_balance},
$\F_k(0)=\sum_{j\neq k}p_{\N\setminus\{j\}}=1-p_{\N\setminus\{k\}}$, so the
pattern law is uniquely determined, $\P(\mathcal{A}=\N\setminus\{k\})=F_k(0)$,
and the zero events $\{X_k=0\}$, $k\in\N$, partition the sample space. At the
JE boundary, the \emph{zeros} are mutually exclusive and exhaustive, dually to
ME, under which the \emph{positive} values are mutually exclusive. For
canonical vectors with $n\ge3$, all the remaining freedom lies in the copulas
$C_{\N\setminus\{k\}}$.
\end{itemize}
\end{remark}

The admissible face weights are governed by a linear programme, which we solve
in closed form. It is formulated for general capacities $c_i$ and for faces of
bounded size, as required in Sections~\ref{subsec:GJE}
and~\ref{subsec:mexclusive}; for canonical JE, $c_i=\F_i(0)$ and $m=n$. For
$2\le m\le n$, let $\M^{(m)}=\{\I\subseteq\N:2\le|\I|\le m-1\}$, so that
$\M^{(n)}=\M$ and $\M^{(2)}=\emptyset$, and for $c=(c_1,\dots,c_n)$ let
$\|c\|_1=\sum_{i\in\N}c_i$ and $\|c\|_\infty=\max_{i\in\N}c_i$.

\begin{proposition}
\label{prop:face_allocation}
Let $2\le m\le n$ and $c\in[0,\infty)^n$. Then
\begin{equation}
\label{eq:face_allocation_value}
\begin{aligned}
V_m(c)&:=\max\Bigl\{\sum_{\I\in\M^{(m)}}(|\I|-1)p_\I:\ p_\I\ge0,\ 
\sum_{\I\in\M^{(m)},\,\I\ni i}p_\I\le c_i,\ i\in\N\Bigr\}\\
&\phantom{:}=\min\Bigl\{\frac{m-2}{m-1}\,\|c\|_1,\ \|c\|_1-\|c\|_\infty\Bigr\}.
\end{aligned}
\end{equation}
\end{proposition}

For canonical JE, $\mathcal{P}(F_1,\dots,F_n)\neq\emptyset$ if and only if
$V_n(\F_1(0),\dots,\F_n(0))\ge\tau-1$, which reduces to~\eqref{eq:JECondition}
and recovers Proposition~\ref{prop:canonical_valid}(b). The set
$\mathcal{P}(F_1,\dots,F_n)$ is a polytope in $\RR^{|\M|}$ which typically has
positive dimension, leaving substantial freedom in the face weights; it
degenerates to a single point at the boundary $\tau=n-1$, where the face
weights are $p_{\N\setminus\{k\}}=F_k(0)$ by
Remark~\ref{rm:JE_Canonical_Remark}(4). For $n=3$, it is the
set of $(p_{12},p_{13},p_{23})\ge0$ satisfying three pairwise capacity
constraints and one lower bound on the total face mass, and a convenient
one-parameter section through it can be written down explicitly. We state it for
general capacities, as it serves both the canonical and the generalised
constructions.

\begin{proposition}
\label{prop:JE_trivariate}
\label{prop:G-JE_trivariate}
Let $n=3$, and let $c\in[0,1]^3$ satisfy $V_3(c)\ge\tau-1$, that is,
$\tau-1\le\min\{\frac12\|c\|_1,\ \|c\|_1-\|c\|_\infty\}$. For
$\I=\N\setminus\{j\}$, let $\overline{p}_\I=\min_{i\in\I}c_i$,
$\underline{p}_\I=(\tau-1-c_j)_+$ and
\begin{equation}
\label{eq:JE_trivariate_P}
p_{\I}(\lambda)=\lambda\,\overline{p}_\I+(1-\lambda)\,\underline{p}_\I,\qquad\lambda\in[0,1].
\end{equation}
Then $\underline{p}_\I\le\overline{p}_\I$ for every $\I\in\M$, and the set of
$\lambda\in[0,1]$ for which $\sum_{\I\in\M_i}p_\I(\lambda)\le c_i$, $i\in\N$, and
$\sum_{\I\in\M}p_\I(\lambda)\ge\tau-1$ is a non-empty interval
\begin{equation}
\label{eq:JE_trivariate_lambda}
[\lambda_-,\lambda_+]\subseteq[0,1],\qquad
\lambda_-=\Bigl(\frac{\tau-1-\sum_{\I\in\M}\underline{p}_\I}
{\sum_{\I\in\M}(\overline{p}_\I-\underline{p}_\I)}\Bigr)_+,
\end{equation}
with the convention $\lambda_-=0$ if the denominator vanishes. For canonical JE,
$c_i=\F_i(0)$, the hypothesis reduces to~\eqref{eq:JECondition}, and
$\overline{p}_\I$ and $\underline{p}_\I=(\sum_{i\in\I}\F_i(0)-1)_+$ are the
Fr\'{e}chet bounds for the probability $p_\I=\P(X_i>0,\,i\in\I)$ that the pair
$\I$ is jointly positive.
\end{proposition}

Within the family~\eqref{eq:JE_trivariate_P}, $\lambda=\lambda_-$ puts the
least mass on the faces compatible with the marginals, and increasing $\lambda$
transfers mass from the semi-axes and the origin to the faces, with all three
copulas remaining free.

\begin{example}
\label{ex:JE_canonical_example}
Let $n=3$ and let each $X_i$ equal zero with probability $\frac12$ and be
uniform on $(0,1)$ otherwise, so that $\F_i(x)=\frac12(1-x)$ for $x\in[0,1]$.
Then $\tau=\frac32$, so~\eqref{eq:JECondition} holds whereas the ME condition
fails. With $c_i=\F_i(0)=\frac12$, Proposition~\ref{prop:JE_trivariate} gives
$\overline{p}_\I=\frac12$, $\underline{p}_\I=0$ and
\[
p_{\I}(\lambda)=\frac{\lambda}{2},
\qquad
\lambda\in\left[\frac{1}{3},\,\frac{1}{2}\right].
\]
At $\lambda=\frac13$, equality holds in~\eqref{eq:JE_parameters_2} and the
origin carries no mass; at $\lambda=\frac12$, the triangle inequality of
Remark~\ref{rm:JE_Canonical_Remark}(2) holds with $p_\I=\frac12(\tau-2\F_j(0))
=\frac14$, and the semi-axes carry no mass. Choose the copulas
\[
C_{12}(u_1,u_2)=u_1u_2,\quad
C_{13}(u_1,u_3)=\min\{u_1,u_3\},\quad
C_{23}(u_2,u_3)=\max\{u_2+u_3-1,0\},
\]
under which the faces $E_{\{1,2\}}$, $E_{\{1,3\}}$ and $E_{\{2,3\}}$ carry
independence, co-monotonicity and counter-monotonicity, respectively; the
resulting support is shown in Figure~\ref{fig:JE_canonical_ex}.

\begin{figure}[H]
\centering
\includegraphics[width=0.62\textwidth]{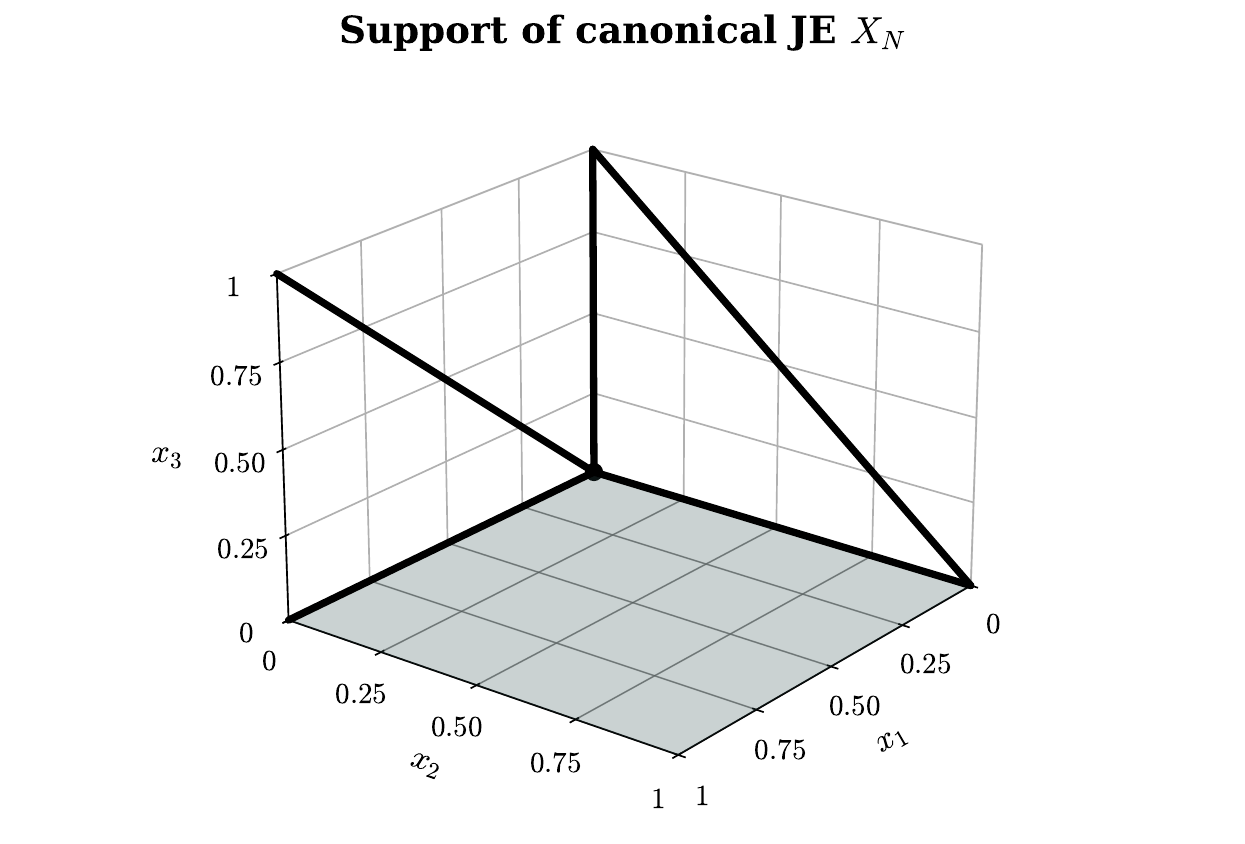}
\caption{Support of the canonical JE random vector $X_{\N}$ with copulas
$C_{12}$, $C_{13}$, and $C_{23}$.}
\label{fig:JE_canonical_ex}
\end{figure}

\noindent Since $\F_i(x_i)/\F_i(0)=1-x_i$,
Corollary~\ref{cor:JE_Canonical_CDF_phi} gives, for $x_\N\in[0,1]^3$ (which
determines $F_\N$, as $X_\N\in[0,1]^3$ almost surely),
\[
F_{\N}(x_{\N})=
1-\frac{1}{2}\sum_{i\in\N}(1-x_i)
+\frac{\lambda}{2}\Bigl[(1-x_1)(1-x_2)
+\min\{1-x_1,1-x_3\}
+\max\{1-x_2-x_3,0\}\Bigr].
\]
On each face, the positive components are uniform on $(0,1)$; moreover,
$X_1=X_3$ on $E_{\{1,3\}}$ and $X_2+X_3=1$ on $E_{\{2,3\}}$. Hence
$\E[X_i]=\frac14$, $\Var(X_i)=\frac{5}{48}$, $\E[X_1X_2]=\frac{\lambda}{8}$,
$\E[X_1X_3]=\frac{\lambda}{6}$ and $\E[X_2X_3]=\frac{\lambda}{12}$, and the
Pearson correlations are
\[
\Corr(X_1,X_2)=\frac{6\lambda-3}{5},\qquad
\Corr(X_1,X_3)=\frac{8\lambda-3}{5},\qquad
\Corr(X_2,X_3)=\frac{4\lambda-3}{5}.
\]
At $\lambda=\frac13$ they equal $-\frac15$, $-\frac1{15}$ and $-\frac13$, and at
$\lambda=\frac12$ they equal $0$, $\frac15$ and $-\frac15$. All three
correlations increase linearly in $\lambda$, as mass moves from the semi-axes,
where the components exclude one another, to the faces, where the copulas take
over. Thus a single JE model spans negatively correlated, uncorrelated and
positively correlated pairs, while the three components are never jointly
positive. This canonical vector has no frequency--severity representation
(Proposition~\ref{prop:FS_rep}): in a trivariate law with uniform margins, a
co-monotonic pair $(1,3)$ and a counter-monotonic pair $(2,3)$ force
$Y_2=1-Y_1$, so the pair $(1,2)$ cannot be independent.
\end{example}

\subsection{Generalised JE}
\label{subsec:GJE}

In canonical JE, every face involving component $i$ uses the same
zero-truncated severity law $\F_i/\F_i(0)$, so the joint losses on the faces
range over the full support of the marginals. In applications, one may wish to
confine joint losses to a sub-range---for instance, to prevent large losses from
occurring together---or to reshape their distribution. We therefore replace
$\F_i/\F_i(0)$ on the intermediate faces by a transformation $G_i\circ\F_i$,
letting the semi-axes absorb the difference.

\begin{definition}
\label{def:g-JE}
For each $i\in\N$, let $G_i:[0,1]\to[0,1]$ be non-decreasing and
left-continuous, with $G_i(u)=0$ for $u\in[0,a_i]$, $G_i(u)=1$ for
$u\in[b_i,1]$ and $\lim_{u\downarrow0}G_i(u)=0$, where
$0\le a_i<b_i\le\F_i(0)$; in particular, $\F_i(0)>0$. Define
\begin{equation}
\label{eq:G-JE_G_star}
G_i^*=\inf_{\substack{0\le x<y\\
G_i(\F_i(x))>G_i(\F_i(y))}}
\frac{\F_i(x)-\F_i(y)}{G_i\left(\F_i(x)\right)-G_i\left(\F_i(y)\right)}.
\end{equation}
A random vector $X_{\N}$ is said to be \emph{generalised JE} (G-JE) if
$\P(X_\N\in E_\N)=0$ and
\begin{align}
\P\left(X_i>x_i,\;i\in \I;\;X_j=0,\;j\in\N\setminus \I\right)
&= p_{\I}\,C_{\I}\left(G_i\left(\F_i(x_i)\right),\;i\in\I\right),
\quad \I\in\M,\ x_{\I}\in[0,\infty)^{|\I|},
\label{eq:G-JEProbFaces}\\
\P\left(X_i>x_i;\;X_j=0,\;j\in\N\setminus\{i\}\right)
&=\F_i(x_i)-\pi_i\,G_i\left(\F_i(x_i)\right),
\quad i\in\N,\ x_i\in[0,\infty),
\label{eq:G-JEProbAxes}
\end{align}
the origin carries mass~\eqref{eq:JEProbOrigin}, and the face weights
$p_\I\ge0$ satisfy~\eqref{eq:JE_parameters_2} and
\begin{equation}
\label{eq:G-JE_parameters_1}
\sum_{\I\in \M_i}p_{\I}\le G_i^*, \qquad i\in\N.
\end{equation}
\end{definition}

Under these assumptions, $G_i\circ\F_i$ is a survival function on $(0,\infty)$:
it equals one at zero because $b_i\le\F_i(0)$, it is right-continuous because
$G_i$ is left-continuous and $\F_i$ is right-continuous and non-increasing, and
it vanishes at infinity because $G_i(0+)=0$. Hence~\eqref{eq:G-JEProbFaces}
specifies, on each face, a law with marginal survival functions
$G_i\circ\F_i$ and survival copula $C_\I$. The constant $G_i^*$ is the smallest
ratio between a decrease of $\F_i$ and the corresponding decrease of
$G_i\circ\F_i$; it measures how much face mass component $i$ can support while
its semi-axis still carries a genuine survival function (a likelihood-ratio
interpretation is given in Section~\ref{subsec:distortion}). Taking $x=0$ and
$y\to\infty$ in~\eqref{eq:G-JE_G_star} shows that $G_i^*\le\F_i(0)$. We write
$\overline{H}_i(x)=\F_i(x)-\pi_iG_i(\F_i(x))$, $x\ge0$, for the right-hand
side of~\eqref{eq:G-JEProbAxes}, the mass that the $i$-th semi-axis carries
above $x$. Canonical
JE is recovered with $G_i(u)=\min\{u/\F_i(0),1\}$, for which $G_i^*=\F_i(0)$
and~\eqref{eq:G-JE_parameters_1} reduces to~\eqref{eq:JE_parameters_1}.

\begin{theorem}
\label{thm:G_JE_characterization}
Let $F_1,\dots,F_n$ be marginal distributions on $[0,\infty)$ and let
$G_1,\dots,G_n$ be as in Definition~\ref{def:g-JE}. A G-JE random vector
$X_{\N}$ with marginals $F_1,\dots,F_n$ exists if and only if
\begin{equation}
\label{eq:G-JECondition}
\sum_{i\in\N}\F_i(0)-1
\le \min\left\{\frac{n-2}{n-1}\sum_{i\in\N}G_i^*,\;
\sum_{i\in\N}G_i^*-\max_{i\in\N}G_i^*\right\}.
\end{equation}
\end{theorem}

\begin{remark}
\label{rm:G-JE_Remark}
\mbox{}
\begin{itemize}
\item[(1)] Since $V_n$ is non-decreasing in each capacity and
$G_i^*\le\F_i(0)$, condition~\eqref{eq:G-JECondition} implies
\eqref{eq:JECondition}: the additional control over the face severities is paid
for by a smaller admissible range of marginals. If $\F_i$ attains both $b_i$
and $a_i$, then~\eqref{eq:G-JE_G_star} gives $G_i^*\le b_i-a_i$: the capacity
of component $i$ is capped by the width of the retained range.

\item[(2)] If $\F_i$ is continuous at some $x^*>0$ while $G_i\circ\F_i$ jumps
there, i.e., $G_i$ is discontinuous at $\F_i(x^*)$ from the right, then letting
$x\uparrow x^*$ and $y=x^*$ in~\eqref{eq:G-JE_G_star} gives $G_i^*=0$, and all
faces involving $i$ must vanish. Useful transformations therefore create no
jumps where $\F_i$ is continuous.

\item[(3)] Replacing each argument $\F_i(x_i)/\F_i(0)$
in~\eqref{eq:JE_Canonical_CDF} by $G_i(\F_i(x_i))$ yields the CDF of a G-JE
vector, and Remark~\ref{rm:JE_Canonical_Remark}(1),(3) carries over verbatim.

\item[(4)] The $i$-th semi-axis carries no mass if and only if
$\pi_i=\F_i(0)$. As $\pi_i\le G_i^*\le\F_i(0)$, this forces
$G_i^*=\F_i(0)$ and $\overline{H}_i\equiv0$, hence
$G_i(\F_i(x))=\F_i(x)/\F_i(0)$ for all $x\ge0$: an axis-free component is
necessarily canonical. In particular, at the boundary $\tau=n-1$, where all
semi-axes are empty by Remark~\ref{rm:JE_Canonical_Remark}(4), every G-JE
vector is canonical, whereas equality in~\eqref{eq:G-JECondition} alone does
not force the semi-axes to be empty.
\end{itemize}
\end{remark}

When $n=3$, Proposition~\ref{prop:G-JE_trivariate} applies verbatim with
capacities $c_i=G_i^*$, whose hypothesis $V_3(c)\ge\tau-1$ is
exactly~\eqref{eq:G-JECondition}. We illustrate the flexibility gained with
linear truncations; power transformations, which tilt the face severities
towards small or large losses, are treated in Example~\ref{ex:G-JE_power} of
Appendix~\ref{app:example}.

\begin{example}[Linear truncation]
\label{ex:G-JE_example}
Let $G_i(u)=(u-a_i)/(b_i-a_i)$ on
$[a_i,b_i]$, extended by $0$ on $[0,a_i]$ and by $1$ on $[b_i,1]$. The ratio
in~\eqref{eq:G-JE_G_star} equals $b_i-a_i$ when both $\F_i(x)$ and $\F_i(y)$
lie in $[a_i,b_i]$, and is at least $b_i-a_i$ otherwise. With $b_i=\F_i(0)$ and
$\F_i$ attaining some value in $[a_i,\F_i(0))$, we obtain $G_i^*=\F_i(0)-a_i$,
and $a_i=0$ recovers canonical JE. Condition~\eqref{eq:G-JECondition} then
becomes the pair of inequalities
\[
\sum_{i\in\N}\F_i(0)+(n-2)\sum_{i\in\N}a_i\le n-1
\qquad\text{and}\qquad
\sum_{i\in\N}a_i+\max_{i\in\N}\bigl(\F_i(0)-a_i\bigr)\le1,
\]
the first strengthening~\eqref{eq:JECondition} by the truncation levels. With
$a_i=1-\alpha_i$ for levels $\alpha_i\in(F_i(0),1]$, so that
$G_i^*=\alpha_i-F_i(0)$, these read
\[
\sum_{i\in\N}(1-\alpha_i)+\frac{1}{n-1}\sum_{i\in\N}\bigl(\alpha_i-F_i(0)\bigr)\le1
\qquad\text{and}\qquad
\sum_{i\in\N}(1-\alpha_i)+\max_{i\in\N}\bigl(\alpha_i-F_i(0)\bigr)\le1.
\]
Here $G_i(\F_i(x))=0$ as soon as $F_i(x)\ge\alpha_i$, so on every intermediate
face $X_i\le F_i^{-1}(\alpha_i)$ almost surely: exceedances of the
$\alpha_i$-quantiles occur only on the semi-axes, where all other components
vanish. Such G-JE vectors are therefore \emph{tail mutually exclusive} at levels
$\alpha_1,\dots,\alpha_n$ in the sense of~\cite{Cheung2017}; when each $F_i$ is
continuous at $F_i^{-1}(\alpha_i)$, so that $\P(X_i>F_i^{-1}(\alpha_i))=1-\alpha_i$,
the second inequality is the admissibility condition for that structure, and
the first is the additional price of JE.

As a concrete case, modify Example~\ref{ex:JE_canonical_example} by setting
$a_i=\frac18$ and $b_i=\F_i(0)=\frac12$. Then $G_i^*=\frac38$,
condition~\eqref{eq:G-JECondition} holds ($\frac12\le\frac{9}{16}$), and
$G_i(\F_i(x))=(1-\frac43x)_+$ for $x\in[0,1]$. Proposition~\ref{prop:G-JE_trivariate}
with $c_i=\frac38$ gives $\overline{p}_\I=\frac38$, $\underline{p}_\I=\frac18$ and
$p_\I(\lambda)=\frac{\lambda}{4}+\frac18$, $\lambda\in[\frac16,\frac14]$. With
the copulas of Example~\ref{ex:JE_canonical_example}, the CDF follows from
Remark~\ref{rm:G-JE_Remark}(3). Joint losses are now confined to $[0,\frac34)$
in each active coordinate, whereas losses in $[\frac34,1]$ occur only on the
semi-axes; see Figure~\ref{fig:G-JE_ex}.

\begin{figure}[H]
\centering
\includegraphics[width=0.62\textwidth]{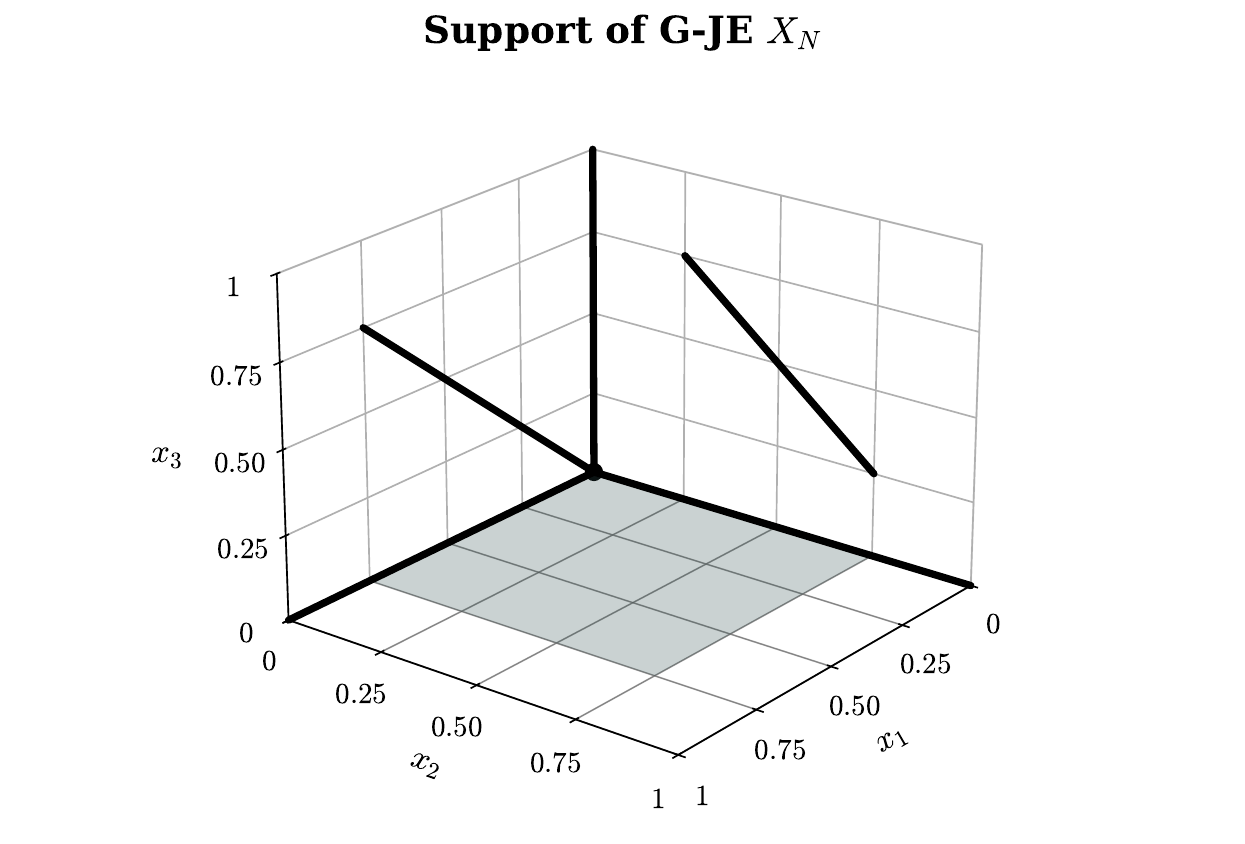}
\caption{Support of the G-JE random vector $X_\N$ with linear truncation
functions $G_i$ at level $a_i=1/8$, restricting the effective face support to
$[0,3/4)$ in each active coordinate.}
\label{fig:G-JE_ex}
\end{figure}

\end{example}

\subsection{The \texorpdfstring{$m$}{m}-exclusivity hierarchy}
\label{subsec:mexclusive}

JE forbids the joint positivity of all $n$ components; forbidding the joint
positivity of every group of $m$ components yields a hierarchy that
interpolates between JE and ME.

\begin{definition}
\label{def:m_exclusive}
Let $2\le m\le n$. A non-negative random vector $X_{\N}$ is said to be
\emph{$m$-exclusive} if
\[
\F_{\I}(0_{\I})=0,\text{ or equivalently, }
\prod_{i\in\I}X_i=0\text{ almost surely},\quad\text{for all }\I\subseteq\N\text{ with }|\I|=m,
\]
that is, no $m$ components are simultaneously positive. The special cases
$m=n$ and $m=2$ recover JE and ME, respectively.
\end{definition}

Equivalently, $X_\N$ is $m$-exclusive if and only if $|\mathcal{A}|\le m-1$
almost surely: at most $m-1$ lines can report a claim in the same period. In
particular, $m$-exclusivity implies $m'$-exclusivity for every $m'\ge m$, and the
levels form the chain
\[
\text{ME} = 2\text{-exclusive}
\;\Rightarrow\;3\text{-exclusive}
\;\Rightarrow\;\cdots
\;\Rightarrow\;n\text{-exclusive} = \text{JE}.
\]
Like JE, $m$-exclusivity restricts the claim pattern alone; its stochastic
representations are those of every family of admissible patterns
(Proposition~\ref{prop:K_rep} below), with the $m$-subsets of $\N$ as forbidden
groups.

\begin{theorem}
\label{thm:m-exclusivity}
Let $F_1,\dots,F_n$ be marginal distributions on $[0,\infty)$ and let
$2\le m\le n$. An $m$-exclusive random vector $X_{\N}$ with these marginals
exists if and only if
\begin{equation}
\label{eq:m-exclusive_Condition}
\sum_{i\in\N}\F_i(0)\le m-1.
\end{equation}
Moreover, the vector $X^*_\N$ of Theorem~\ref{thm:JE_Characterisation_1} is
$m$-exclusive for every $m$ satisfying~\eqref{eq:m-exclusive_Condition}.
\end{theorem}

\begin{proof}
If $X_\N$ is $m$-exclusive, then $\tau=\E|\mathcal{A}|\le m-1$
by~\eqref{eq:count_identity}. Conversely, if $\tau\le m-1$, then
$\kappa^*\le\lceil\tau\rceil\le m-1$ almost surely.
\end{proof}

Theorem~\ref{thm:m-exclusivity} shows that the single construction $X^*_\N$
realises the strongest exclusivity level compatible with the marginals, namely
$m^*=\max\{\lceil\tau\rceil+1,2\}$ whenever $\tau\le n-1$. At the other
extreme, an exclusivity level
also caps how dangerous the claim pattern can be.

\begin{proposition}
\label{prop:count_range}
Let $F_1,\dots,F_n$ be marginal distributions on $[0,\infty)$, and let
$2\le m\le n$ with $\tau\le m-1$. Let $\kappa^\sharp_m$ equal $m-1$ with
probability $\tau/(m-1)$ and $0$ otherwise. Then every $m$-exclusive $X_\N$ with
marginals $F_1,\dots,F_n$ satisfies
\begin{equation}
\label{eq:count_range}
\kappa^*\le_{\mathrm{cx}}|\mathcal{A}(X_\N)|\le_{\mathrm{cx}}\kappa^\sharp_m .
\end{equation}
The lower bound is attained by $X^*_\N$; the upper bound is attained if and
only if $\max_{i\in\N}\F_i(0)\le\tau/(m-1)$, and then by the vector
$X^\sharp_\N$ obtained from the construction of
Theorem~\ref{thm:JE_Characterisation_1} by laying the arcs of
Lemma~\ref{lem:wraparound} on a circle of circumference $\tau/(m-1)$ inside
$[0,1)$. Moreover, $\kappa^\sharp_m\le_{\mathrm{cx}}\kappa^\sharp_{m'}$ for
$m\le m'\le n$.
\end{proposition}

The extremal vector $X^\sharp_\N$ produces the most concentrated claim pattern
that $m$-exclusivity permits: either no line is hit, or exactly $m-1$ lines are
hit together. Applying~\eqref{eq:count_range} to $f(k)=k^2$ shows that the variance
of the number of claims of an $m$-exclusive vector lies in
$[(\tau-\lfloor\tau\rfloor)(\lceil\tau\rceil-\tau),\ \tau(m-1-\tau)]$, and
applying it to $f(k)=k(k-1)/2$ shows that the expected number of jointly hit
pairs, $\sum_{i<j}\P(X_i>0,X_j>0)$, lies between
$\lfloor\tau\rfloor(2\tau-\lfloor\tau\rfloor-1)/2$ and $\tau(m-2)/2$.

Canonical and generalised $m$-exclusive vectors are obtained by restricting the
face weights to $\M^{(m)}$, that is, by setting $p_\I=0$ whenever $|\I|\ge m$ in
Definitions~\ref{def:JECanonical} and~\ref{def:g-JE}. Applying
Proposition~\ref{prop:face_allocation} with capacities $c_i=G_i^*$ in the proof
of Theorem~\ref{thm:G_JE_characterization} yields the following result.

\begin{corollary}
\label{cor:m_exclusive_G_star}
Let $F_1,\dots,F_n$ be marginal distributions on $[0,\infty)$, let $2\le m\le n$,
and let $G_1,\dots,G_n$ be as in Definition~\ref{def:g-JE}. A generalised
$m$-exclusive random vector $X_\N$ with marginals $F_1,\dots,F_n$ exists if and
only if
\begin{equation}
\label{eq:G-m-exclusive_Condition}
\sum_{i\in\N}\F_i(0)-1
\le\min\left\{
\frac{m-2}{m-1}\sum_{i\in\N}G_i^*,\;
\sum_{i\in\N}G_i^*-\max_{i\in\N}G_i^*
\right\}.
\end{equation}
When $\F_i(0)>0$ for all $i\in\N$, the choice $G_i(u)=\min\{u/\F_i(0),1\}$
reduces this to~\eqref{eq:m-exclusive_Condition}.
\end{corollary}

\begin{remark}
\label{rm:m_exclusive_canonical}
\mbox{}
\begin{itemize}
\item[(1)] By Proposition~\ref{prop:JE_face}, a canonical JE vector is
$m$-exclusive if and only if $p_\I=0$ for every $\I\in\M$ with $|\I|\ge m$:
as $m$ decreases from $n$ to $2$, faces of successively lower dimension are
switched off, until only the semi-axes and the origin remain under ME.

\item[(2)] At the boundary $\tau=m-1$ of~\eqref{eq:m-exclusive_Condition}, the
argument of Remark~\ref{rm:JE_Canonical_Remark}(4) gives $|\mathcal{A}|=m-1$
almost surely, so every $m$-exclusive vector lives on the faces of dimension
exactly $m-1$. For $m=2$, exactly one component is positive, with
$\P(\mathcal{A}=\{i\})=\F_i(0)$; for $m=n$, exactly one component vanishes.

\item[(3)] For $n=3$, the hierarchy consists of ME and JE only; a genuinely
intermediate level first appears for $n=4$ and $m=3$, where the admissible
weights of the six pairwise faces form a polytope of dimension up to six.
\end{itemize}
\end{remark}

\subsection{General families of admissible claim patterns}
\label{subsec:families}

The hierarchy of Section~\ref{subsec:mexclusive} restricts the claim pattern
only through its size. In practice, exclusivity constraints often depend on the
identity of the lines: covers that respond to competing causes of loss never
claim together, the benefits of one policy may be mutually exclusive while
those of different policies are not, and a programme may cap the number of
simultaneous claims within each site as well as overall. Such constraints can
be described in two equivalent ways: by the claim patterns that may occur, or
by the groups of lines that may never claim together.

\begin{definition}
\label{def:K_exclusive}
A family $\mathcal{K}\subseteq2^\N$ of \emph{admissible patterns} is closed
under taking subsets and contains every singleton. Its \emph{forbidden groups}
are the minimal sets outside $\mathcal{K}$,
\[
\mathcal{F}(\mathcal{K})=\bigl\{\J\subseteq\N:\ \J\notin\mathcal{K}\ \text{and}\
\J\setminus\{i\}\in\mathcal{K}\text{ for all }i\in\J\bigr\}.
\]
A non-negative random vector $X_\N$ is \emph{$\mathcal{K}$-exclusive} if
$\mathcal{A}\in\mathcal{K}$ almost surely.
\end{definition}

A pattern is admissible if and only if it contains no forbidden group: a set
outside $\mathcal{K}$ contains a minimal such set, and a set containing a
forbidden group lies outside $\mathcal{K}$, which is closed under taking
subsets. The two descriptions thus determine each other, and every forbidden
group has at least two lines; requiring singletons to be admissible only
discards lines that never claim. ME, $m$-exclusivity and JE forbid all pairs, all
$m$-subsets and $\N$ itself, respectively. In combinatorial terms,
$\mathcal{K}$ is a simplicial complex on $\N$, and its forbidden groups are its
minimal non-faces. Table~\ref{tab:structures} collects these structures and
those studied below.

\begin{table}[H]
\centering
\caption{Exclusivity structures as families of admissible patterns: forbidden
groups, admissibility of the marginals (Theorem~\ref{thm:families}) and the
margins that determine the law (Proposition~\ref{prop:K_IE}).}
\label{tab:structures}
\small
\setlength{\tabcolsep}{4pt}
\newcolumntype{R}[1]{>{\raggedright\arraybackslash}p{#1}}
\begin{tabular}{@{}R{2.45cm}R{3.0cm}R{3.2cm}R{4.15cm}R{2.95cm}@{}}
\hline
Structure & Admissible patterns & Forbidden groups & Admissible if and only if &
Law determined by margins on\\
\hline
ME & $|\I|\le1$ & pairs & $\tau\le1$ & single lines\\
$m$-exclusivity & $|\I|\le m-1$ & $m$-subsets & $\tau\le m-1$ &
$(m-1)$-subsets\\
JE & $\I\neq\N$ & $\N$ & $\tau\le n-1$ & $(n-1)$-subsets\\
Hierarchical caps & $|\I\cap L|\le c_L$, $L\in\mathcal{L}$ (laminar) &
minimal $(c_L+1)$-subsets of the $L\in\mathcal{L}$ &
$\sum_{i\in L}\F_i(0)\le c_L$, $L\in\mathcal{L}$ & maximal patterns\\
Matroid & independent sets & circuits &
$\sum_{i\in\J}\F_i(0)\le\varrho(\J)$, $\J\subseteq\N$ & bases\\
Exclusion graph $H$ & stable sets of $H$ & edges of $H$ &
clique inequalities if $H$ is perfect; NP-hard in general & maximal stable
sets\\
\hline
\end{tabular}
\end{table}

As $\mathcal{K}$-exclusivity constrains only the claim pattern, the
representations of Sections~\ref{subsec:structure} and~\ref{subsec:canonical}
carry over.

\begin{proposition}[Representations of exclusive vectors]
\label{prop:K_rep}
For a non-negative random vector $X_\N$ with marginals $F_1,\dots,F_n$,
pattern law $(p_\I)$ and face laws $(Q_\I)$ as in
Proposition~\ref{prop:JE_face}, the following are equivalent:
\begin{itemize}
\item[(a)] $X_\N$ is $\mathcal{K}$-exclusive;
\item[(b)] \emph{forbidden groups:} the sub-vector $X_\J$ is JE for every
$\J\in\mathcal{F}(\mathcal{K})$, that is, $\P(X_i>0,\ i\in\J)=0$;
\item[(c)] \emph{random faces:} $p_\I=0$ for every $\I\notin\mathcal{K}$, that
is, the random pattern in~\eqref{eq:random_face} lies in $\mathcal{K}$ almost
surely;
\item[(d)] \emph{copulas:} some, equivalently every, copula of $X_\N$ assigns
no mass to the boxes $\{u_\N\in[0,1]^n:u_i>F_i(0),\ i\in\J\}$,
$\J\in\mathcal{F}(\mathcal{K})$.
\end{itemize}
Moreover, $X_\N\stackrel{d}{=}(B_iY_i,\,i\in\N)$ for a random vector
$B_\N\in\{0,1\}^n$ and a random vector $Y_\N$ on $(0,\infty)^n$ independent of
$B_\N$ if and only if the face laws of $X_\N$ are margins of a single law
(Proposition~\ref{prop:FS_rep}); then $B_\N$ is a $\mathcal{K}$-exclusive
Bernoulli vector with $\P(B_i=1)=\F_i(0)$, and $\P(Y_i>y)=\F_i(y)/\F_i(0)$,
$y\ge0$, whenever $\F_i(0)>0$. Every $\mathcal{K}$-exclusive $X_\N$ shares its
marginals and its pattern law with such a frequency--severity vector.
\end{proposition}

Part~(b) gives the most transparent reading: a family of admissible patterns
imposes JE on each forbidden sub-portfolio. Each of these constraints is
admissible on its own if and only if $\sum_{i\in\J}\F_i(0)\le|\J|-1$, by
Theorem~\ref{thm:JE_Characterisation_1} applied to $X_\J$, but forbidden groups
overlap, and so do their constraints: under ME, every pair of lines is
admissible on its own as soon as $\F_i(0)+\F_j(0)\le1$, whereas all pairs
together require $\tau\le1$. By~(d), as for JE (Remark~\ref{rm:JECDFrm}(2)),
exclusivity depends on the atoms at zero and is impossible when all marginals
are continuous, an obstruction that disappears above marginal quantiles
(Section~\ref{subsec:threshold}). By the last statement, as in canonical JE, the attainable pattern laws do not
depend on the face severities, so admissibility is a question about pattern
laws alone. The nested representations of Section~\ref{subsec:canonical} persist:
under ME they coincide, as the ME law with given marginals is unique
\citep{Dhaene1999}, whereas the canonical vector of
Example~\ref{ex:JE_canonical_example} is $3$-exclusive but not of
frequency--severity form.

Write
$q_\N=(\F_1(0),\dots,\F_n(0))$ for the vector of claim probabilities,
$\mathbf{1}_\I\in\{0,1\}^n$ for the indicator vector of $\I\subseteq\N$, and
\[
h_{\mathcal{K}}(w)=\max_{\I\in\mathcal{K}}\sum_{i\in\I}w_i,\qquad
w_\N\in[0,\infty)^n,
\]
for the largest total weight of an admissible pattern.

\begin{theorem}
\label{thm:families}
Let $F_1,\dots,F_n$ be marginal distributions on $[0,\infty)$ and let
$\mathcal{K}$ be as in Definition~\ref{def:K_exclusive}. The following are
equivalent:
\begin{itemize}
\item[(a)] a $\mathcal{K}$-exclusive random vector with marginals
$F_1,\dots,F_n$ exists;
\item[(b)] $q_\N$ lies in the polytope
$\mathcal{Q}(\mathcal{K})=\mathrm{conv}\{\mathbf{1}_\I:\I\in\mathcal{K}\}$;
\item[(c)] $\sum_{i\in\N}w_i\F_i(0)\le h_{\mathcal{K}}(w)$ for every
$w_\N\in[0,\infty)^n$.
\end{itemize}
In that case, the pattern laws of the $\mathcal{K}$-exclusive vectors with
marginals $F_1,\dots,F_n$ form the polytope
\begin{equation}
\label{eq:pattern_polytope}
\mathcal{P}_{\mathcal{K}}(q_\N)=\Bigl\{(p_\I)_{\I\in\mathcal{K}}:\ p_\I\ge0,\
\sum_{\I\in\mathcal{K}}p_\I=1,\ \sum_{\I\in\mathcal{K},\,\I\ni i}p_\I=\F_i(0),\
i\in\N\Bigr\};
\end{equation}
each element of $\mathcal{P}_{\mathcal{K}}(q_\N)$ is attained with arbitrary
survival copulas on the faces, as well as in frequency--severity form, and each
vertex of $\mathcal{P}_{\mathcal{K}}(q_\N)$ charges at most $n+1$ patterns.
\end{theorem}

In every period the active set is an admissible pattern, so the vector $q_\N$
of claim probabilities, the average of the indicator vectors
$\mathbf{1}_\mathcal{A}$, is a mixture of admissible patterns; by~(b), this
necessary condition is also sufficient, as a pattern drawn from any such
mixture can be dressed with independent severities. Condition~(c) is the dual
statement: whatever weights are attached to the lines, the expected weight of
the active set cannot exceed the largest weight of an admissible pattern. The
weights $w_\N=\one$ give $\tau\le\max_{\I\in\mathcal{K}}|\I|$, which recovers
Theorems~\ref{thm:JE_Characterisation_1} and~\ref{thm:m-exclusivity}, and
$w_\N=\mathbf{1}_\J$, $\J\in\mathcal{F}(\mathcal{K})$, gives the JE conditions
of the forbidden groups. By the vertex bound, a portfolio subject to any
exclusivity constraints can be generated from a table of at most $n+1$ claim
scenarios. For JE, $\mathcal{P}_{\mathcal{K}}(q_\N)$ is in bijection with the
polytope $\mathcal{P}(F_1,\dots,F_n)$ of Definition~\ref{def:JECanonical}.

For hierarchical constraints, the infinitely many inequalities in~(c) collapse
to one inequality per constraint, and a single wrap-around vector attains all
of them simultaneously. A family $\mathcal{L}$ of subsets of $\N$ is
\emph{laminar} if any two of its members are either disjoint or nested.

\begin{corollary}
\label{cor:laminar}
Let $\mathcal{L}$ be a laminar family of subsets of $\N$ with integer
capacities $c_L\ge1$, $L\in\mathcal{L}$, and let
$\mathcal{K}=\{\I\subseteq\N:|\I\cap L|\le c_L\text{ for all }L\in\mathcal{L}\}$.
A $\mathcal{K}$-exclusive random vector with marginals $F_1,\dots,F_n$ exists
if and only if
\begin{equation}
\label{eq:laminar}
\sum_{i\in L}\F_i(0)\le c_L,\qquad L\in\mathcal{L}.
\end{equation}
In that case, if the lines are labelled so that every $L\in\mathcal{L}$
consists of consecutive indices, the vector $X^*_\N$ of
Theorem~\ref{thm:JE_Characterisation_1} is $\mathcal{K}$-exclusive and, for
every $L\in\mathcal{L}$, minimises the number $|\mathcal{A}\cap L|$ of claims in
$L$ in convex order among all random vectors with marginals $F_1,\dots,F_n$.
\end{corollary}

Corollary~\ref{cor:laminar} contains Theorem~\ref{thm:m-exclusivity}
($\mathcal{L}=\{\N\}$, $c_\N=m-1$) and covers group-wise ME combined with an
overall cap, as in the following example; in particular, if the forbidden
groups are pairwise disjoint, the JE conditions of the groups are also
sufficient. Laminar capacity families are instances of matroids; for the
independent sets of a general matroid with rank function $\varrho$, Edmonds'
description of the matroid polytope \citep{Edmonds1971} shows that~(c) reduces
to $\sum_{i\in\J}\F_i(0)\le\varrho(\J)$, $\J\subseteq\N$, although a single
construction attaining all constraints is then no longer available in general.

\begin{example}
\label{ex:hierarchical}
A parametric multi-peril programme covers three sites, each with two covers
whose triggers are mutually exclusive, and pays at most two benefits per
period. Thus $\mathcal{L}=\{\{1,2\},\{3,4\},\{5,6\},\N\}$ with
$c_{\{1,2\}}=c_{\{3,4\}}=c_{\{5,6\}}=1$ and $c_\N=2$, and the forbidden groups
are the three pairs of covers at the same site and the eight triples with one
cover at each site. For claim probabilities
$q_\N=(0.3,0.4,0.2,0.5,0.35,0.25)$, condition~\eqref{eq:laminar} reads
$0.7,0.7,0.6\le1$ and $\tau=2\le2$. Laying arcs of lengths $q_1,\dots,q_6$ end
to end around the unit circle, the arcs of each site occupy a block of length at
most one and every point is covered exactly twice: in every period exactly two
covers claim, and they belong to different sites. ME is impossible here, and
JE alone would allow both covers of a site, or all three sites, to claim
together.
\end{example}

For exclusivity constraints that are not hierarchical, condition~(c) cannot in
general be reduced to finitely many explicit inequalities. Suppose that the
forbidden groups are the edges of an \emph{exclusion graph} $H$ on $\N$, so
that adjacent lines never claim together and the admissible patterns are the
stable sets of $H$. Then $\mathcal{Q}(\mathcal{K})$ is the stable set polytope
of $H$, and evaluating $h_{\mathcal{K}}$ is the maximum-weight stable set
problem, which is NP-hard \citep{Karp1972}. If $H$ is perfect, for instance a
disjoint union of cliques, an interval graph or a bipartite graph,
condition~(c) reduces to the clique inequalities $\sum_{i\in C}\F_i(0)\le1$ for
every clique $C$ of $H$ \citep{Chvatal1975}. Without perfection, checking the
mutually exclusive pairs separately is not enough.

\begin{example}
\label{ex:ring}
Five covers are arranged in a ring, and adjacent covers $i$ and $i+1$ (indices
modulo $5$) respond to competing causes of loss and never claim together,
whereas non-adjacent covers may. The admissible patterns are the stable sets of
the $5$-cycle, none of which has more than two elements. With
$\F_i(0)=\frac12$ for every $i$, each pair of mutually exclusive covers
satisfies the ME condition $\F_i(0)+\F_{i+1}(0)\le1$, yet no admissible vector
exists: with $w_\N=\one$ in Theorem~\ref{thm:families}(c), the expected
number of claims $\tau=\frac52$ exceeds the largest admissible pattern size
$h_{\mathcal{K}}(\one)=2$. By contrast, with $\F_i(0)=\frac25$, the pattern law
$p_{\{i,i+2\}}=\frac15$, $i\in\N$, is admissible and produces exactly two
non-adjacent claims in every period.
\end{example}

Such constraints may also force a variable number of claims: if line~$4$ never
claims together with lines $1$, $2$ or $3$ and $\F_i(0)=\frac12$ for all $i$,
the only admissible pattern law charges $\{4\}$ and $\{1,2,3\}$ with
probability $\frac12$ each, so $|\mathcal{A}|\in\{1,3\}$ although $\tau=2$,
whereas under laminar caps $X^*_\N$ remains admissible.

Finally, the inclusion--exclusion principle of Section~\ref{subsec:IE} extends
to every family of admissible patterns, so that exclusivity constraints of any
kind reduce the computation of premiums to sub-portfolios.

\begin{proposition}[Inclusion--exclusion for admissible patterns]
\label{prop:K_IE}
For $\I\in\mathcal{K}$, let
$\chi_{\mathcal{K}}(\I)=\sum_{\J\in\mathcal{K},\,\J\supseteq\I}(-1)^{|\J|-|\I|}$.
\begin{itemize}
\item[(a)] A non-negative random vector $X_\N$ is $\mathcal{K}$-exclusive if
and only if
\begin{equation}
\label{eq:K_IE}
\E\bigl[g(X_\N)\bigr]=\sum_{\I\in\mathcal{K}}\chi_{\mathcal{K}}(\I)\,
\E\bigl[g\bigl(X^{(\I)}\bigr)\bigr]
\end{equation}
for every bounded measurable $g:[0,\infty)^n\to\RR$. If $X_\N$ is
$\mathcal{K}$-exclusive, \eqref{eq:K_IE} holds for every measurable $g$ with
$\E|g(X^{(\I)})|<\infty$, $\I\in\mathcal{K}$, and the law of $X_\N$ is
determined by its margins on the maximal admissible patterns.
\item[(b)] Consistent laws on the coordinates of the maximal admissible
patterns are the margins of a $\mathcal{K}$-exclusive vector if and only if
$\nu=\sum_{\I\in\mathcal{K}}\chi_{\mathcal{K}}(\I)\,\mu^{(\I)}\ge0$, where
$\mu^{(\I)}$ is the law they induce on the coordinates in $\I$, embedded into
$[0,\infty)^n$ by setting the others to zero; $\nu$ is then its unique law.
\end{itemize}
\end{proposition}

Both parts are proved in Appendix~\ref{app:proofs}. The coefficient $\chi_{\mathcal{K}}(\I)$ is an
alternating count of the admissible patterns that contain $\I$: up to sign, it is the reduced Euler
characteristic of the link of $\I$ in the simplicial complex $\mathcal{K}$,
and it equals one when $\I$ is maximal. For JE,
$\chi_{\mathcal{K}}(\I)=(-1)^{n-1-|\I|}$, and Proposition~\ref{prop:K_IE}
reduces to Theorem~\ref{thm:JE_IE} and Corollary~\ref{cor:JE_consequences}(a),(b);
for $m$-exclusivity, the coefficients are binomial,
see~\eqref{eq:m_IE} below, and for ME, $\chi_{\mathcal{K}}(\emptyset)=1-n$
and $\chi_{\mathcal{K}}(\{i\})=1$; the coefficients factorise over blocks of
lines that are constrained separately. With $g(x_\N)=(\sum_ix_i-d)_+$ and
marginals with finite means, the stop-loss premium
$\Pi_S(d)=\sum_{\emptyset\neq\I\in\mathcal{K}}\chi_{\mathcal{K}}(\I)\,
\Pi_{S_\I}(d)$, $d\ge0$, is a signed sum over admissible sub-portfolios. In
Example~\ref{ex:hierarchical}, the maximal admissible patterns are the
twelve pairs of covers at different sites, and $\chi_{\mathcal{K}}$ equals $1$
on these pairs, $-3$ on single covers and $7$ on the empty pattern. Hence
$\Pi_S(d)=\sum\Pi_{X_i+X_j}(d)-3\sum_{i\in\N}\Pi_{X_i}(d)$, $d\ge0$, the first
sum running over the twelve cross-site pairs: the stop-loss premiums of the
programme follow from bivariate computations, whatever the dependence between
the covers that may claim together.

%+++++++++++++++++++++++++++++++++++++++++++++++++++
\section{Further Discussion}
\label{sec:further_discussion}

We now connect the theory with actuarial practice and with the wider literature
on extremal dependence: threshold exceedances (Section~\ref{subsec:threshold}),
aggregate losses, stop-loss premiums and their extremal structures
(Section~\ref{subsec:aggregate}), distortion risk measures
(Section~\ref{subsec:distortion}) and joint mixability
(Section~\ref{subsec:JM}).

\subsection{Exclusivity above thresholds}
\label{subsec:threshold}

In risk management, one is often less concerned with whether losses are
positive than with whether they exceed prescribed levels, such as retentions or
VaR thresholds \citep{McNeil2015}. For levels $l_\N=(l_i,\,i\in\N)\in\RR^n$ and
$2\le m\le n$, we say that a random vector $Z_\N$ is \emph{$m$-exclusive above
$l_\N$}, or tail $m$-exclusive, if
\[
\P(Z_i>l_i,\;i\in\I)=0\qquad\text{for all }\I\subseteq\N\text{ with }|\I|=m,
\]
that is, if the vector of excesses $Z^+_\N$ with $Z^+_i=(Z_i-l_i)_+$ is
$m$-exclusive; for $m=n$ we speak of JE above $l_\N$. At most $m-1$ components
then exceed their levels in the same period, whereas the values below the
levels are unconstrained. Exclusivity below levels, $\P(Z_i<l_i,\,i\in\I)=0$,
reduces to this by the reflection $Z_\N\mapsto-Z_\N$; for instance, the
components of a non-positive vector are never all strictly negative if and only
if its negative is JE.

\begin{proposition}
\label{prop:threshold}
Let $F_{Z_1},\dots,F_{Z_n}$ be distribution functions on $\RR$, let
$l_\N\in\RR^n$, and let $2\le m\le n$. A random vector $Z_\N$ with these
marginals that is $m$-exclusive above $l_\N$ exists if and only if
\begin{equation}
\label{eq:m-exclusive_threshold}
\sum_{i\in\N}\P(Z_i>l_i)\le m-1.
\end{equation}
Moreover, let $Z^+_\N$ be $m$-exclusive with $Z^+_i\stackrel{d}{=}(Z_i-l_i)_+$,
and let $Z^-_\N$ be independent of $Z^+_\N$, with $Z^-_i$ distributed as $Z_i$
given $\{Z_i\le l_i\}$ (and $Z^-_i\equiv l_i$ if $\P(Z_i\le l_i)=0$). Then
$\widetilde Z_i=(l_i+Z^+_i)\,\mathds{1}\{Z^+_i>0\}+Z^-_i\,\mathds{1}\{Z^+_i=0\}$,
$i\in\N$, defines such a vector.
\end{proposition}

The constraint thus bears only on the pattern of exceedances. Indeed, every
$Z_\N$ splits pathwise into exceedance indicators, excesses and sub-threshold
parts,
\[
Z_i=B_i\,(l_i+E_i)+(1-B_i)\,D_i,\qquad
B_i=\mathds{1}\{Z_i>l_i\},\quad E_i=(Z_i-l_i)_+,\quad D_i=\min\{Z_i,l_i\},
\]
and $Z_\N$ is $m$-exclusive above $l_\N$ if and only if the Bernoulli vector
$B_\N$, equivalently the excess vector $E_\N=Z^+_\N$, is $m$-exclusive. This is
the tail counterpart of the frequency--severity form, with exceedances in place
of claims: the excesses have the representations of
Proposition~\ref{prop:K_rep}, the dependence below the levels is free, and the
construction of Proposition~\ref{prop:threshold} is the special case in which
the sub-threshold parts are drawn independently of the excesses. Consequently,
every result of Sections~\ref{sec:preliminaries} and~\ref{sec:main_results}
applies to the excesses $E_i$ that per-risk excess-of-loss layers pay.
Figure~\ref{fig:tail} shows, for $n=3$, the admissible regions
of $2$-exclusivity and of JE above $l_\N$, together with that of the tail mutual
exclusivity discussed below: they thicken the axes and the coordinate planes of
Figure~\ref{fig:MEvsJE}, to which they reduce when $l_\N=0_\N$.

\begin{figure}[H]
\centering
\includegraphics[width=0.93\textwidth]{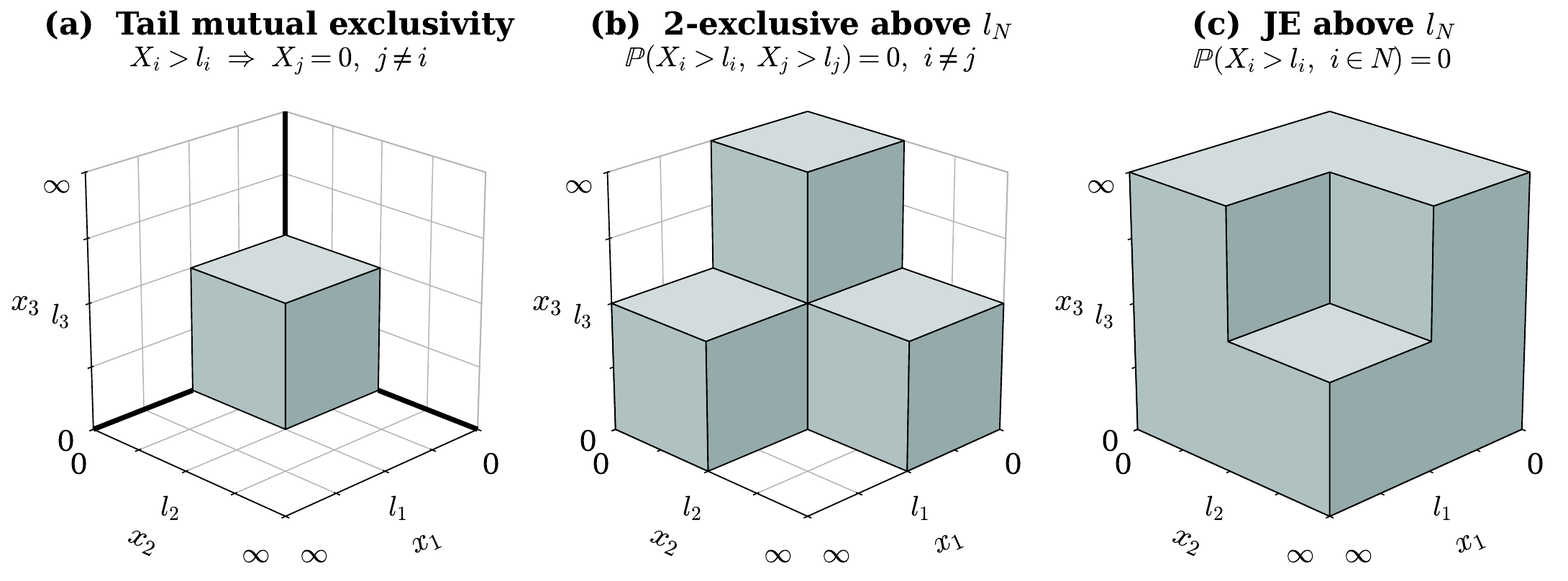}
\caption{Admissible regions above levels $l_\N$ for $n=3$, nested from left to
right: (a) tail mutual exclusivity \citep{Cheung2017}, the box
$\prod_{i\in\N}[0,l_i]$ and the three semi-axes beyond the levels; (b)
$2$-exclusivity above $l_\N$, the box and three arms along the axes; (c) JE
above $l_\N$, the orthant without the box $\prod_{i\in\N}(l_i,\infty)$.}
\label{fig:tail}
\end{figure}

At VaR levels $l_i=F_{Z_i}^{-1}(\alpha_i)$, $\alpha_i\in(0,1)$, with continuous
marginals, \eqref{eq:m-exclusive_threshold} reads
$\sum_{i\in\N}(1-\alpha_i)\le m-1$. For $m=2$, the exceedances are pairwise
exclusive, which requires $\sum_{i\in\N}(1-\alpha_i)\le1$. For $m=n$, the
condition becomes $\sum_{i\in\N}\alpha_i\ge1$: a simultaneous breach of all $n$
VaRs can be ruled out by a suitable dependence structure unless the levels are
very low, and, as in Lemma~\ref{lem:JE_elementary}, the joint breach
probability $\P(Z_i>F_{Z_i}^{-1}(\alpha_i),\,i\in\N)$ has the lower bound
$(1-\sum_{i\in\N}\alpha_i)_+$, which is attained. At the boundary
$\sum_{i\in\N}(1-\alpha_i)=m-1$, exactly $m-1$ components exceed their VaRs in
every period, as in Remark~\ref{rm:m_exclusive_canonical}(2). At quantile
levels, exclusivity is a property of the copula alone.

\begin{proposition}[Tail exclusivity as a copula property]
\label{prop:tail_copula}
Let $Z_\N$ have continuous marginals $F_{Z_1},\dots,F_{Z_n}$ and copula $C$,
let $\alpha_\N\in(0,1)^n$, $l_i=F_{Z_i}^{-1}(\alpha_i)$, $i\in\N$, and
$2\le m\le n$. Then $Z_\N$ is $m$-exclusive above $l_\N$ if and only if $C$
assigns no mass to the boxes $\{u_\N\in[0,1]^n:u_i>\alpha_i,\ i\in\I\}$ with
$|\I|=m$. Consequently, $m$-exclusivity above quantile levels is preserved by
continuous, strictly increasing transformations of the components, and a copula
with this property exists if and only if $\sum_{i\in\N}(1-\alpha_i)\le m-1$, in
which case it makes every vector with continuous marginals $m$-exclusive above
the corresponding quantiles.
\end{proposition}

\begin{proof}
With $U_i=F_{Z_i}(Z_i)$, uniform on $(0,1)$ with copula $C$, continuity gives
$F_{Z_i}(l_i)=\alpha_i$, so $\{U_i>\alpha_i\}\subseteq\{Z_i>l_i\}\subseteq
\{U_i\ge\alpha_i\}$, and $\P(U_i=\alpha_i)=0$; this proves the
characterisation. A continuous, strictly increasing $h_i$ leaves the copula
unchanged and maps $l_i$ to the $\alpha_i$-quantile of $h_i(Z_i)$. Existence
follows from Proposition~\ref{prop:threshold} with uniform marginals and levels
$\alpha_i$, and, given such a copula and $U_\N\sim C$, the vector
$(F_i^{-1}(U_i),\,i\in\N)$ has marginals $F_1,\dots,F_n$ and is $m$-exclusive
above their $\alpha_i$-quantiles by the characterisation.
\end{proof}

This contrasts sharply with exclusivity itself: by
Remark~\ref{rm:JECDFrm}(2), JE depends on the atoms at zero and is impossible
when all marginals are continuous, and so is ME. The closest notion in the
literature, depicted in Figure~\ref{fig:tail}(a), is the tail mutual
exclusivity of \citet{Cheung2017}.

\begin{remark}[Tail mutual exclusivity]
\label{rm:TME}
Let $X_\N$ be non-negative with marginals $F_1,\dots,F_n$, let
$\alpha_i\in[F_i(0),1)$ and $l_i=F_i^{-1}(\alpha_i)$ with $F_i(l_i)=\alpha_i$,
$i\in\N$. Following \citet{Cheung2017}, $X_\N$ is \emph{tail mutually
exclusive} (TME) at levels $\alpha_\N$ if $\P(X_i>l_i,\,X_j>0)=0$ for all
$i\neq j$: a loss above its VaR rules out any other positive loss, while
losses below the levels are unrestricted.
\begin{itemize}
\item[(1)] \emph{Inclusions.} As $l_j\ge0$, TME at levels $\alpha_\N$ implies
$2$-exclusivity above $l_\N$, which implies $m$-exclusivity above $l_\N$ for
every $m\ge2$, as the nested panels of Figure~\ref{fig:tail} show. With
$\alpha_i=F_i(0)>0$, so that $l_i=0$, TME and $2$-exclusivity above $l_\N$ both
reduce to ME.
\item[(2)] \emph{Admissibility.} TME is admissible if and only if
$\sum_{i\in\N}(1-\alpha_i)+\max_{i\in\N}(\alpha_i-F_i(0))\le1$
\citep{Cheung2017}, whereas $m$-exclusivity above $l_\N$ requires only
$\sum_{i\in\N}(1-\alpha_i)\le m-1$ (Proposition~\ref{prop:threshold}). The
additional term is the price of the vanishing requirement: for each $j$, the
event $\{X_j>0\}$, of probability $1-F_j(0)$, must avoid the exceedances of all
other lines, of total probability $\sum_{i\neq j}(1-\alpha_i)$. In particular,
for $n\ge2$, TME forces $F_i(0)>0$ for every $i$: it depends on the atoms at
zero, is impossible for losses that are continuous on $(0,\infty)$, such as the
annual losses of business units, and is not a property of the copula, in
contrast with Proposition~\ref{prop:tail_copula}.
\item[(3)] \emph{Dependence below the levels.} TME forces all other components
to vanish in every period in which one component exceeds its level, so that a
large loss always occurs alone; $m$-exclusivity above $l_\N$ leaves the
sub-threshold values entirely free (Proposition~\ref{prop:threshold}), which is
what permits positive dependence in ordinary periods
(Example~\ref{ex:tail_units} below).
\item[(4)] \emph{Roles.} TME is an extremal structure: if $F_i(0)>0$ for every
$i$ and some marginal is unbounded, the TVaR of a TME aggregate attains, at all
sufficiently high levels, its smallest value over the Fr\'{e}chet class
\citep{Cheung2017}. Exclusivity above levels is instead a class of structures,
encoding which large losses can occur together, and its worst case is
quantified by Theorem~\ref{thm:TVaR_threshold}. As TME vectors are $2$-exclusive
above $l_\N$, the TVaR of the aggregate over the vectors that are $2$-exclusive
above $l_\N$ then ranges from the Fr\'{e}chet lower bound, attained under TME,
to the worst-case value $\mathcal{W}_1(\alpha)$ of Theorem~\ref{thm:TVaR_threshold}.
\end{itemize}
Tail mutually exclusive JE vectors arise from the linear truncation of
Example~\ref{ex:G-JE_example}, at the price of the first inequality there.
\end{remark}

\begin{example}
\label{ex:TME}
Three lines claim with probability $\frac12$ each, with unit-mean exponential
severities, and $\alpha_i=0.9$, so that $l_i=\log5$. Both TME, as
$0.3+0.4\le1$, and $2$-exclusivity above $l_\N$, as $0.3\le1$, are admissible.
For every dependence structure, $S\ge\sum_{i\in\N}X_i\mathds{1}\{X_i>t\}$ with
$t=\log150$, for which $\sum_{i\in\N}\P(X_i>t)=0.01$, so
$\mathrm{TVaR}_{0.99}(S)\ge t+1=6.01$. A TME vector attains this bound: outside
the exceedances of the levels, $S\le3\log5<t$, and in each exceedance a single
line is positive. Independence gives $7.73$; the worst case under
$2$-exclusivity above $l_\N$ (Theorem~\ref{thm:TVaR_threshold}) is $9.16$, that
under JE above $l_\N$ is $12.23$, and co-monotonicity gives $14.74$. Knowing
that no two lines exceed their 1-in-10 levels together thus confines
$\mathrm{TVaR}_{0.99}(S)$ to $[6.01,9.16]$, whereas TME singles out the safest
structure in that range.
\end{example}

Tail $m$-exclusivity thus brings exclusivity to continuous losses, while
leaving their dependence below the thresholds free, as the following example
shows.

\begin{example}
\label{ex:tail_units}
Three business units have unit-mean exponential annual losses, and a unit has a
\emph{severe year} when its loss exceeds its 1-in-10-year level $l=\log10$.
Apply Proposition~\ref{prop:threshold} with $m=2$, mutually exclusive excesses
and co-monotonic sub-threshold parts: in $70\%$ of years no unit is severe and
the three losses coincide, and in each of the remaining three slices of
probability $0.1$ one unit is severe while the other two share a common loss
below $\log10$. Any two units then have Spearman's rho $0.618$ and Pearson
correlation $0.21$, yet two units are never severe together, whereas
independence would make this happen in $2.8\%$ of years. The aggregate has
$\mathrm{TVaR}_{0.99}(S)=9.01$, below its value $9.64$ under independence and
far below $16.82$ under co-monotonicity; by Theorem~\ref{thm:TVaR_threshold}
below, no dependence under which at most one unit is severe at a time can raise
it above $11.24$. In the tail, exclusivity matters more than correlation.
\end{example}

\subsection{Aggregate losses and stop-loss premiums}
\label{subsec:aggregate}

For a random vector $X_\N$ with fixed non-negative marginals, the aggregate
$S=\sum_{i\in\N}X_i$ drives stop-loss and excess-of-loss premiums as well as
solvency capital \citep{Denuit2005, Albrecher2017}. When $\tau\le1$, ME yields
the smallest aggregate in convex order and hence the smallest stop-loss
premiums \citep{Dhaene1999}, a minimality property that characterises ME
\citep{Cheung2014}. The following elementary bound shows what is lost when ME
is unavailable.

\begin{proposition}
\label{prop:SL_lower}
Let $X_\N$ be non-negative with marginals $F_1,\dots,F_n$ having finite means.
Then
\[
\Pi_S(d)\ge\sum_{i\in\N}\Pi_{X_i}(d),\qquad d\ge0,
\]
with equality for all $d\ge0$ if and only if $X_\N$ is ME. If, in addition, the
marginals have finite second moments, then
\begin{equation}
\label{eq:JE_var_bound}
\int_0^\infty\Bigl(\Pi_S(d)-\sum_{i\in\N}\Pi_{X_i}(d)\Bigr)\,\d d
=\sum_{\substack{i,j\in\N\\i<j}}\E[X_iX_j]
=\frac12\bigl(\Var(S)-\sigma^2_{\mathrm{ME}}\bigr),
\end{equation}
where $\sigma^2_{\mathrm{ME}}=\sum_{i\in\N}\Var(X_i)-2\sum_{i<j}\E[X_i]\E[X_j]$
is the variance that the aggregate would have under ME. In particular, if
$\tau>1$, every $X_\N$ with these marginals, JE or not, satisfies
$\Pi_S(d)>\sum_{i\in\N}\Pi_{X_i}(d)$ for some $d\ge0$.
\end{proposition}

Identity~\eqref{eq:JE_var_bound} shows that the stop-loss gap, integrated over
all retentions, coincides with the sum of the cross moments and hence with half
the variance excess over the ME benchmark. When $\tau>1$, the benchmark is unattainable,
and two tools identify what comes closest. The first is the
inclusion--exclusion principle of Proposition~\ref{prop:K_IE}, which is
explicit at every level of the hierarchy: for
$\mathcal{K}=\{\I\subseteq\N:|\I|\le m-1\}$, the identity
$\sum_{j=0}^{r}(-1)^j\binom{s}{j}=(-1)^r\binom{s-1}{r}$, $s\ge1$, $r\ge0$,
gives $\chi_{\mathcal{K}}(K)=(-1)^{m-1-|K|}\binom{n-1-|K|}{m-1-|K|}$, so that
$X_\N$ is $m$-exclusive if and only if
\begin{equation}
\label{eq:m_IE}
\E\bigl[g(X_\N)\bigr]=\sum_{\substack{K\subseteq\N\\|K|\le m-1}}
(-1)^{m-1-|K|}\binom{n-1-|K|}{m-1-|K|}\,\E\bigl[g\bigl(X^{(K)}\bigr)\bigr]
\end{equation}
for every bounded measurable $g:[0,\infty)^n\to\RR$. For $m=n$, all binomial
coefficients equal one and Theorem~\ref{thm:JE_IE} is recovered; for $m=2$,
\eqref{eq:m_IE} reads
$\E[g(X_\N)]=\sum_{i\in\N}\E[g(X^{(\{i\})})]-(n-1)g(0_\N)$, the classical
representation of ME \citep{Dhaene1999}. The law of an $m$-exclusive vector is
thus determined by its $(m-1)$-dimensional margins, and $\Pi_S(d)$ is the
corresponding signed sum of the $\Pi_{S_K}(d)$, $1\le|K|\le m-1$, $d\ge0$,
which reduces to $\Pi_S=\sum_{i\in\N}\Pi_{X_i}$ under ME. Likewise, every mixed
moment $\E[\prod_{i\in\I}X_i^{k_i}]$ with positive integers $k_i$ and $|\I|\ge m$
vanishes, so that the moments of $S$ involve at most $m-1$ distinct lines;
under $3$-exclusivity, for instance,
$\E[S^3]=\sum_{i\in\N}\E[X_i^3]+3\sum_{i\neq j}\E[X_i^2X_j]$.

The second tool identifies the safest and the most dangerous structures within
the frequency--severity framework that underlies the individual risk model
\citep{Klugman2012a, Denuit2005}: the claim indicators may be arbitrarily
dependent, whereas the severities are exchangeable and independent of the
indicators. Exchangeability covers independent and identically distributed
severities, common-shock models in which the severities are conditionally
independent given a common environment such as claims inflation, and the
extreme case of a single common severity.

\begin{theorem}
\label{thm:safest_FS}
Let $q_\N\in[0,1]^n$, and let $Y_\N=(Y_1,\dots,Y_n)$ be an exchangeable
random vector with $\P(Y_1>0)=1$ and $\E[Y_1]<\infty$. Let
$\mathcal{C}(q_\N,Y_\N)$ be the class of random vectors with components
$X_i=B_iY_i$, $i\in\N$, where $B_\N$ is any vector of indicators with
$\P(B_i=1)=q_i$, independent of $Y_\N$; thus $\F_i(0)=q_i$ and
$\tau=\sum_{i\in\N}q_i$. Let $T_0=0$ and $T_k=Y_1+\dots+Y_k$,
$k\in\N$; when $Y_1,\dots,Y_n$ are independent, $T_k$ has the $k$-fold
convolution of the law of $Y_1$.
\begin{itemize}
\item[(a)] Let $X^*_\N\in\mathcal{C}(q_\N,Y_\N)$ use the indicators
$B^*_i=\mathds{1}\{\Theta\in R_i\}$ of Theorem~\ref{thm:JE_Characterisation_1}.
Then $S^*=\sum_{i\in\N}X^*_i\le_{\mathrm{cx}}S$ for every
$X_\N\in\mathcal{C}(q_\N,Y_\N)$, and
\begin{equation}
\label{eq:SL_star}
\Pi_{S^*}(d)=\bigl(1-\tau+\lfloor\tau\rfloor\bigr)\,
\Pi_{T_{\lfloor\tau\rfloor}}(d)
+\bigl(\tau-\lfloor\tau\rfloor\bigr)\,\Pi_{T_{\lfloor\tau\rfloor+1}}(d),
\qquad d\ge0,
\end{equation}
where the second term is absent when $\tau$ is an integer. Moreover,
$X^*_\N$ is ME if $\tau\le1$, JE if $\tau\le n-1$, and $m$-exclusive for every
$m\in\{2,\dots,n\}$ with $m\ge\lceil\tau\rceil+1$.
\item[(b)] Let $2\le m\le n$ with $\tau\le m-1$. Every $m$-exclusive
$X_\N\in\mathcal{C}(q_\N,Y_\N)$ satisfies
\begin{equation}
\label{eq:SL_range}
\Pi_{S^*}(d)\le\Pi_S(d)\le\frac{\tau}{m-1}\,\Pi_{T_{m-1}}(d),\qquad d\ge0.
\end{equation}
Both bounds hold with equality at $d=0$. Some $m$-exclusive
$X_\N\in\mathcal{C}(q_\N,Y_\N)$ attains the upper bound for every $d\ge0$ if
and only if $\max_{i\in\N}q_i\le\tau/(m-1)$, and the upper bound is
non-decreasing in $m$.
\end{itemize}
\end{theorem}

Theorem~\ref{thm:safest_FS}(a) extends, within the frequency--severity class,
the safest-dependence result of \citet{Dhaene1999} beyond the ME regime: for
$\tau\le1$ the minimiser is ME, and for $1<\tau\le n-1$ it is JE and attains
the strongest admissible level of exclusivity. The minimising claim pattern is
the same for every exchangeable $Y_\N$; only the minimal premiums depend on the
law of $Y_\N$. By part~(b), the lower bound
in~\eqref{eq:SL_range} does not depend on $m$, whereas the upper bound is
non-decreasing in $m$: exclusivity constraints shrink the range of stop-loss
premiums from above. At the boundary $\tau=m-1$, the two bounds coincide:
every $m$-exclusive $X_\N\in\mathcal{C}(q_\N,Y_\N)$ then has
$|\mathcal{A}|=m-1$ almost surely and hence $S\stackrel{d}{=}T_{m-1}$, so that
dependence uncertainty about the aggregate disappears altogether. When the
severities are degenerate at some $v>0$, as for fixed-benefit covers such as
term life or parametric insurance, every random vector with marginals
$(1-q_i)\delta_0+q_i\delta_v$ belongs to $\mathcal{C}(q_\N,Y_\N)$, so
Theorem~\ref{thm:safest_FS} then identifies the safest structure over the
entire Fr\'{e}chet class and, when $\max_{i\in\N}q_i\le\tau/(m-1)$, the most
dangerous $m$-exclusive one; Example~\ref{ex:SL_table} below quantifies these
bounds.

Both hypotheses of Theorem~\ref{thm:safest_FS} matter. Exchangeability is a
property of the whole severity vector and, for $n\ge3$, rules out
counter-monotonic pairs of continuous severities, which would force all pairs
to be counter-monotonic (Section~\ref{sec:intro}). If the severities may depend
on the claim pattern, the safest JE structure may leave the frequency--severity
class (Proposition~\ref{prop:CM_atom} below); if they are independent of the
pattern but not exchangeable, a safest structure need not exist
(Example~\ref{ex:no_extremes}). The most dangerous structure,
by contrast, can be identified over the whole Fr\'{e}chet class under a much
weaker condition, by means of the following elementary bound.

\begin{proposition}
\label{prop:SL_upper}
Let $2\le m\le n$, let $X_\N$ be $m$-exclusive with marginals having finite
means, and, for $d\ge0$, let $\mathcal{D}_m(d)$ be the set of
$d_\N\in[0,\infty)^n$ such that $\sum_{i\in\I}d_i\le d$ for every
$\I\subseteq\N$ with $|\I|=m-1$, that is, such that the sum of the $m-1$
largest $d_i$ does not exceed $d$. Then
\begin{equation}
\label{eq:SL_upper}
\Pi_S(d)\le\sum_{i\in\N}\Pi_{X_i}(d_i),\qquad d_\N\in\mathcal{D}_m(d);
\qquad\text{in particular,}\qquad
\Pi_S(d)\le\sum_{i\in\N}\Pi_{X_i}\Bigl(\frac{d}{m-1}\Bigr).
\end{equation}
\end{proposition}

\begin{proof}
Let $d_\N\in\mathcal{D}_m(d)$. On $\{|\mathcal{A}|\le m-1\}$, choose
$\I\supseteq\mathcal{A}$ with $|\I|=m-1$; then
$(S-d)_+=(\sum_{i\in\I}X_i-d)_+\le(\sum_{i\in\I}(X_i-d_i))_+
\le\sum_{i\in\I}(X_i-d_i)_+\le\sum_{i\in\N}(X_i-d_i)_+$. Taking expectations
gives~\eqref{eq:SL_upper}, and $d_i\equiv d/(m-1)$ lies in
$\mathcal{D}_m(d)$.
\end{proof}

Without exclusivity, the same argument yields the classical co-monotonic bound
$\Pi_S(d)\le\sum_{i\in\N}\Pi_{X_i}(d_i)$ for $\sum_{i\in\N}d_i=d$
\citep{Dhaene2002}; $m$-exclusivity relaxes the budget constraint on the
retentions $d_i$ from all $n$ lines to the $m-1$ largest. When all lines share
the same zero-truncated severity law, the relaxed bound is sharp and is attained
simultaneously for all retentions.

\begin{theorem}
\label{thm:cx_max}
Let $Y$ be a positive random variable with distribution function $F_Y$ and finite mean, let
$q_\N\in[0,1]^n$ with $q_1\ge q_2\ge\dots\ge q_n$, and let $F_1,\dots,F_n$ be
the marginals with $\F_i(x)=q_i\F_Y(x)$, $x\ge0$, so that every line
has claim probability $q_i$ and severity law $F_Y$. Let $2\le m\le n$ with
$\tau\le m-1$, write $c=m-1$, and let
\[
r=\min\{k\in\{0,\dots,c-1\}:\ q_{k+1}\le t_k\},\qquad
t_k=\frac{1}{c-k}\sum_{j=k+1}^nq_j,\qquad t=t_r.
\]
For $U$ uniform on $(0,1)$, let
\begin{equation}
\label{eq:S_sharp}
S^\sharp=\sum_{i=1}^rF_i^{-1}(U)
+(c-r)\,F_Y^{-1}\Bigl(\frac{U-1+t}{t}\Bigr)\mathds{1}\{U>1-t\},
\end{equation}
the last term being zero if $t=0$. Then $S\le_{\mathrm{cx}}S^\sharp$ for every
$m$-exclusive $X_\N$ with marginals $F_1,\dots,F_n$, and $S^\sharp$ is the
aggregate of such a vector. In particular, if $\max_{i\in\N}q_i\le\tau/(m-1)$,
then $r=0$ and $S^\sharp$ equals $(m-1)Y$ with probability $\tau/(m-1)$ and
zero otherwise, in law.
\end{theorem}

The most dangerous structure has a transparent anatomy. The $r$ lines with the
largest claim probabilities are co-monotonic, and the remaining lines are pooled
into a single \emph{super-risk}: blocks of $c-r$ lines that claim together with
a common severity, arranged around a circle as in
Lemma~\ref{lem:wraparound}, with total claim probability $t$. The super-risk is
in turn co-monotonic with the $r$ large lines, and the exclusivity constraint
binds exactly when it is active, since $c=m-1$ lines then claim together. For
identical marginals, the large lines disappear and the whole portfolio
collapses into one super-risk. Without an exclusivity constraint, the most
dangerous structure is the co-monotonic one, under which all lines claim
together whenever the line with the smallest claim probability does
\citep{Dhaene2002}. Since the Tail Value-at-Risk
$\mathrm{TVaR}_\alpha(Z)=(1-\alpha)^{-1}\int_\alpha^1F_Z^{-1}(u)\,\d u
=\min_{s\in\RR}\{s+\Pi_Z(s)/(1-\alpha)\}$, $\alpha\in(0,1)$, is
non-decreasing in convex order \citep{Denuit2005, McNeil2015},
Theorems~\ref{thm:safest_FS} and~\ref{thm:cx_max} also yield sharp TVaR bounds.

\begin{corollary}
\label{cor:TVaR}
Let $\alpha\in(0,1)$ and $2\le m\le n$ with $\tau\le m-1$.
\begin{itemize}
\item[(a)] Under the assumptions of Theorem~\ref{thm:safest_FS}, every
$m$-exclusive $X_\N\in\mathcal{C}(q_\N,Y_\N)$ satisfies
$\mathrm{TVaR}_\alpha(S^*)\le\mathrm{TVaR}_\alpha(S)\le
\mathrm{TVaR}_\alpha(T_{m-1}\mathds{1}\{U\le\tau/(m-1)\})$, with $U$ uniform
on $(0,1)$ and independent of $Y_\N$; the upper bound is attained when
$\max_{i\in\N}q_i\le\tau/(m-1)$.
\item[(b)] Under the assumptions of Theorem~\ref{thm:cx_max}, the largest TVaR
of the aggregate of an $m$-exclusive vector is $\mathrm{TVaR}_\alpha(S^\sharp)$.
If $\max_{i\in\N}q_i\le\tau/(m-1)$, it equals
$(m-1)\,\mathrm{TVaR}_{\alpha'}(Y)$ with $\alpha'=1-(1-\alpha)(m-1)/\tau$ for
$\alpha\ge1-\tau/(m-1)$, and $\tau\,\E[Y]/(1-\alpha)$ otherwise.
\end{itemize}
\end{corollary}

\begin{example}
\label{ex:SL_table}
Return to $n=10$ lines with $q_i=0.15$ and unit-mean exponential severities.
Table~\ref{tab:bounds} reports the variance and $\mathrm{TVaR}_{0.99}$ of the
aggregate under the structures identified above. ME is impossible, but the
wrap-around structure $S^*$ improves on independent claim occurrences by
roughly $17\%$ in TVaR. Over the whole Fr\'{e}chet class, where severities may
also be dependent, JE on its own barely tames the worst case, as the most
dangerous JE structure lets nine lines claim together with a common severity,
whereas $3$-exclusivity cuts the worst-case TVaR by more than two thirds. Exclusivity information is thus a powerful
structural constraint in the sense of the dependence-uncertainty literature
\citep{Embrechts2013a, Bernardetal2017}, and its value increases sharply as the
admissible number of simultaneous claims decreases.

\begin{table}[H]
\centering
\caption{Variance and $\mathrm{TVaR}_{0.99}$ of $S$ for $n=10$ lines with
claim probability $0.15$ and unit-mean exponential severities.}
\label{tab:bounds}
\small
\begin{tabular}{llrr}
\hline
Class & Structure & $\Var(S)$ & $\mathrm{TVaR}_{0.99}(S)$\\
\hline
Frequency--severity & safest $S^*$ (Theorem~\ref{thm:safest_FS}) & 1.75 & 7.12\\
with independent & independent claim occurrences & 2.78 & 8.55\\
severities & most dangerous $3$-exclusive (Theorem~\ref{thm:safest_FS}) & 2.25 & 7.44\\
 & most dangerous JE (Theorem~\ref{thm:safest_FS}) & 12.75 & 15.93\\
 & co-monotonic claim occurrences & 14.25 & 17.05\\
\hline
Fr\'{e}chet class & most dangerous $3$-exclusive (Theorem~\ref{thm:cx_max}) & 3.75 & 10.64\\
 & most dangerous JE (Theorem~\ref{thm:cx_max}) & 24.75 & 34.32\\
 & co-monotonic & 27.75 & 37.08\\
\hline
\end{tabular}
\end{table}
\end{example}

Both theorems rely on homogeneity: Example~\ref{ex:no_extremes} in
Appendix~\ref{app:example} exhibits three lines with heterogeneous severities
for which the frequency--severity class has no minimum in convex order, and
three fixed-benefit covers with unequal benefits whose JE class has no maximum
in convex order. In both cases the extremal premium at a given retention is a
linear programme over the pattern law whose solution changes with the
retention.

Convex-order extremes thus require homogeneity; worst-case TVaR does not.
Under a cap on simultaneous exceedances (Section~\ref{subsec:threshold}), it is
explicit for arbitrary marginals and arbitrary thresholds, which may bear on
large losses rather than on all claims. For levels $l_\N\in\RR^n$, let
$p_i=\F_i(l_i)$ and, for $\alpha\in(0,1)$ and
$\lambda\in\Lambda_i=[(p_i-\alpha)_+,\min\{p_i,1-\alpha\}]$,
\begin{equation}
\label{eq:g_i}
g_i(\lambda)=\int_{1-\lambda}^1F_i^{-1}(u)\,\d u
+\int_{1-p_i-(1-\alpha)+\lambda}^{1-p_i}F_i^{-1}(u)\,\d u ,
\end{equation}
the largest expected losses that line $i$ can contribute from events of
probabilities $\lambda$ and $1-\alpha-\lambda$ on which it does and does not
exceed $l_i$.

\begin{theorem}
\label{thm:TVaR_threshold}
Let $F_1,\dots,F_n$ be distribution functions on $\RR$ with finite means, let
$l_\N\in\RR^n$, $\alpha\in(0,1)$ and $2\le m\le n$ with
$\sum_{i\in\N}p_i\le m-1$, and write $c=m-1$. Let $\mathcal{W}_c(\alpha)$ be the supremum
of $\mathrm{TVaR}_\alpha(S)$ over all random vectors with marginals
$F_1,\dots,F_n$ that are $m$-exclusive above $l_\N$, and let
$\mathcal{W}_n(\alpha)=(1-\alpha)^{-1}\sum_{i\in\N}\int_\alpha^1F_i^{-1}(u)\,\d u$ be
the unconstrained, co-monotonic value.
\begin{itemize}
\item[(a)] Each $g_i$ is concave and non-decreasing on $\Lambda_i$, and
\begin{equation}
\label{eq:W_program}
\mathcal{W}_c(\alpha)=\frac{1}{1-\alpha}\max\Bigl\{\sum_{i\in\N}g_i(\lambda_i):\
\lambda_i\in\Lambda_i,\ \sum_{i\in\N}\lambda_i\le c\,(1-\alpha)\Bigr\},
\end{equation}
where the supremum defining $\mathcal{W}_c(\alpha)$ is attained; with $c=n$, \eqref{eq:W_program} returns
$\mathcal{W}_n(\alpha)$.
\item[(b)] $\mathcal{W}_c(\alpha)$ is non-decreasing and concave in $c\in\{c_0,\dots,n\}$,
where $c_0=\max\{1,\lceil\sum_ip_i\rceil\}$.
\item[(c)] If $F_i=F$ and $l_i=l$ for all $i\in\N$, so that $p_i=p$, the maximum
in~\eqref{eq:W_program} is attained at $\lambda_i=\min\{p,\,c(1-\alpha)/n\}$.
\end{itemize}
\end{theorem}

The worst case has a transparent anatomy. In the worst $100(1-\alpha)\%$ of
periods, line $i$ exceeds its threshold with probability $\lambda_i$,
contributing its largest losses, and otherwise sits just below its threshold;
in the attaining structure, exactly $c$ lines exceed in each of these periods
when the cap binds, and the optimal $\lambda_i$ equalise the marginal gains
$g_i'(\lambda_i)=F_i^{-1}(1-\lambda_i)-F_i^{-1}(1-p_i-(1-\alpha)+\lambda_i)$
across the lines not at the bounds of $\Lambda_i$, a water-filling rule.
Part~(b) is a law of increasing returns to exclusivity: each step down the
hierarchy lowers the worst-case TVaR by at least as much as the previous one.
With $l_\N=0_\N$ and non-negative marginals, the second integral
in~\eqref{eq:g_i} vanishes, and~\eqref{eq:W_program} is the worst-case TVaR of
an $m$-exclusive aggregate for arbitrary marginals; if these are continuous on
$(0,\infty)$ and $\sum_{i\in\N}\min\{1-\alpha,p_i\}\ge c(1-\alpha)$, a
maximiser is $\lambda_i=\min\{1-\alpha,\F_i(\eta)\}$, with $\eta\ge0$ chosen so
that $\sum_{i\in\N}\lambda_i=c(1-\alpha)$: in the worst periods, every line
contributes its losses above a common level $\eta$. This recovers
Corollary~\ref{cor:TVaR}(b) for a common severity law and the values $11.00$
and $15.05$ of Example~\ref{ex:cx_max_r1}, and it applies where no convex-order
maximum exists (Example~\ref{ex:no_extremes}). Applied to the excesses
$Z^+_\N$ of Section~\ref{subsec:threshold}, which are $m$-exclusive with zero
thresholds, it gives the worst-case TVaR of the total paid by per-risk
excess-of-loss layers. Finally, as
$\Pi_S(d)=\sup_{0<\beta<1}\beta\{\mathrm{TVaR}_{1-\beta}(S)-d\}$, the
worst-case stop-loss premium of the aggregate is
$\sup_{0<\beta<1}\beta\{\mathcal{W}_c(1-\beta)-d\}$.

\begin{example}
\label{ex:units}
Return to the unit-mean exponential losses and 1-in-10-year levels
$l=\log10$ of Example~\ref{ex:tail_units}, now for five business units, so that
$\sum_{i\in\N}p_i=0.5$ and every cap $c\ge1$ is admissible.
Table~\ref{tab:units} reports the worst-case $\mathrm{TVaR}_{0.99}$ of the total
loss, from Theorem~\ref{thm:TVaR_threshold}(c), and of the total excess
$\sum_{i\in\N}(X_i-l)_+$, from Corollary~\ref{cor:TVaR}(b), which applies
because, by memorylessness, the excesses are unit-mean exponential with claim
probability $0.1$.

\begin{table}[H]
\centering
\caption{Five units with unit-mean exponential losses and levels $l=\log10$:
worst-case $\mathrm{TVaR}_{0.99}$ of the total loss $S$ and of the total excess
$\sum_{i\in\N}(X_i-l)_+$ when at most $c$ units exceed $l$ together, with
independence as a benchmark; the second column describes the worst $1\%$ of
years in the worst case for $S$.}
\label{tab:units}
\small
\begin{tabular}{llrr}
\hline
Cap & Worst $1\%$ of years (total loss) & $\mathrm{TVaR}_{0.99}(S)$ & $\mathrm{TVaR}_{0.99}$(excess)\\
\hline
independence & -- & 13.00 & 5.62\\
$c=5$ (none) & all five units above $l$ & 28.03 & 16.51\\
$c=4$ & four above $l$, the fifth just below & 25.61 & 14.10\\
$c=3$ & three above $l$, the others just below & 22.91 & 11.44\\
$c=2$ & two above $l$, the others just below & 19.86 & 8.44\\
$c=1$ & one above $l$, the others just below & 16.27 & 4.91\\
\hline
\end{tabular}
\end{table}

\noindent For the total loss, the savings $2.42$, $2.69$, $3.05$ and
$3.59$ increase as the cap tightens, as Theorem~\ref{thm:TVaR_threshold}(b)
predicts, and $c=1$ closes $78\%$ of the gap between the co-monotonic value and
independence. Under $c=1$, the excesses are ME, so their total has the same law
for every admissible structure, and its TVaR of $4.91$ lies below the value
$5.62$ under independence. With $n=3$, the same computation gives the bound
$11.24$ quoted in Example~\ref{ex:tail_units}.
\end{example}

\subsection{Generalised JE and distortion risk measures}
\label{subsec:distortion}

Each $G_i$ in Definition~\ref{def:g-JE} is non-decreasing with $G_i(0)=0$ and
$G_i(1)=1$, hence a distortion function in the sense of~\cite{Wang1996}, and
$G_i\circ\F_i$ is a distorted survival function of $X_i$. The generalised JE
thus replaces the zero-truncated severity law on the intermediate faces by a
distorted law.
The next result identifies the capacity $G_i^*$ and the expected face losses in
these terms. Let $\mu_i^+$ denote the restriction of the law of $X_i$ to
$(0,\infty)$, of total mass $\F_i(0)$, let $Q_i^G$ be the probability measure on
$(0,\infty)$ with survival function $G_i\circ\F_i$ (the common marginal
$Q_{\I,i}$, $\I\in\M_i$, of the intermediate face laws in
Proposition~\ref{prop:JE_face}), and write
$\rho_G(X)=\int_0^\infty G(\F_X(x))\,\d x$ for the distortion risk measure of a
non-negative risk $X$ with distortion function $G$ \citep{Wang1996}.

\begin{proposition}
\label{prop:Gstar_LR}
For each $i\in\N$,
\[
G_i^*=\sup\bigl\{\eta\ge0:\ \eta\,Q_i^G(\mathcal{E})\le\mu_i^+(\mathcal{E})\ \text{for all Borel sets } \mathcal{E}\subseteq(0,\infty)\bigr\}
=\Bigl(\operatorname*{ess\,sup}_{\mu_i^+}\frac{\d Q_i^G}{\d\mu_i^+}\Bigr)^{-1},
\]
where the right-hand side is read as zero if $Q_i^G$ is not absolutely continuous
with respect to $\mu_i^+$ or its density is unbounded. Moreover, every G-JE
random vector satisfies
\[
\E\Bigl[X_i\,\mathds{1}\Bigl\{X_\N\in\bigcup_{\I\in\M}E_\I\Bigr\}\Bigr]
=\pi_i\,\rho_{G_i}(X_i),\qquad i\in\N.
\]
\end{proposition}

Proposition~\ref{prop:Gstar_LR} gives the capacity a transparent meaning: the
total face weight of line $i$ can be at most the reciprocal of the largest
likelihood ratio between the face severity law and the marginal, and the
expected loss of line $i$ in joint-claim scenarios is $\pi_i$ times a distortion
premium of $X_i$. In canonical JE, $Q_i^G=\mu_i^+/\F_i(0)$, $G_i^*=\F_i(0)$ and
$\rho_{G_i}(X_i)=\E[X_i\mid X_i>0]$, the TVaR of $X_i$ at level $F_i(0)$. Under
the linear truncation of Example~\ref{ex:G-JE_example} with
$a_i=1-\alpha_i$, $Q_i^G$ is the law of $X_i$ restricted to the quantile levels
$(F_i(0),\alpha_i]$ and renormalised, so that
$\d Q_i^G/\d\mu_i^+\le(\alpha_i-F_i(0))^{-1}$, with equality on the support of
$Q_i^G$ when $F_i$ is continuous at $F_i^{-1}(\alpha_i)$, and
$\rho_{G_i}(X_i)=(\alpha_i-F_i(0))^{-1}\int_{F_i(0)}^{\alpha_i}F_i^{-1}(u)\,\d u$
is a range value-at-risk. The concave power distortions of
Example~\ref{ex:G-JE_power}, $\gamma\in(0,1)$ with $a_i=0$ and
$b_i=\F_i(0)$, are proportional hazards transforms of the zero-truncated
severity \citep{Wang1996}, which load the upper tail;
for continuous $\F_i$, their likelihood ratio with respect to $\mu_i^+$ is
unbounded near the upper end of the face range, whence $G_i^*=0$ and the faces
must vanish. Convex power distortions, $\gamma>1$, tilt the face severities
towards small losses, with likelihood ratio bounded by $\gamma/\F_i(0)$ and,
under the conditions of Example~\ref{ex:G-JE_power}, capacity
$G_i^*=\F_i(0)/\gamma$. Distortions that make joint losses lighter or truncated are thus admissible,
but heavier-tailed ones are not, a constraint relevant to distortion-based
pricing and capital allocation \citep{TsanakasBarnett2003, Dhaene2012}.

\subsection{Connection with joint mixability}
\label{subsec:JM}

A random vector $\xi_\N$ is jointly mixable (JM) if $\sum_{i\in\N}\xi_i$ is almost
surely constant \citep{Wang2016}; after centring, the constant may be taken to
be zero. JE and JM laws are both concentrated on $(n-1)$-dimensional sets,
namely
\[
W_\mathrm{JE}=\Bigl\{x_\N\in[0,\infty)^n:\min_{i\in\N}x_i=0\Bigr\}
\qquad\text{and}\qquad
W_\mathrm{JM}=\Bigl\{y_\N\in\RR^n:\textstyle\sum_{i\in\N}y_i=0\Bigr\},
\]
the boundary of the non-negative orthant and the zero-sum hyperplane, depicted
in Figure~\ref{fig:JEvsJM} for $n=3$.

\begin{figure}[H]
\centering
\includegraphics[width=0.52\textwidth]{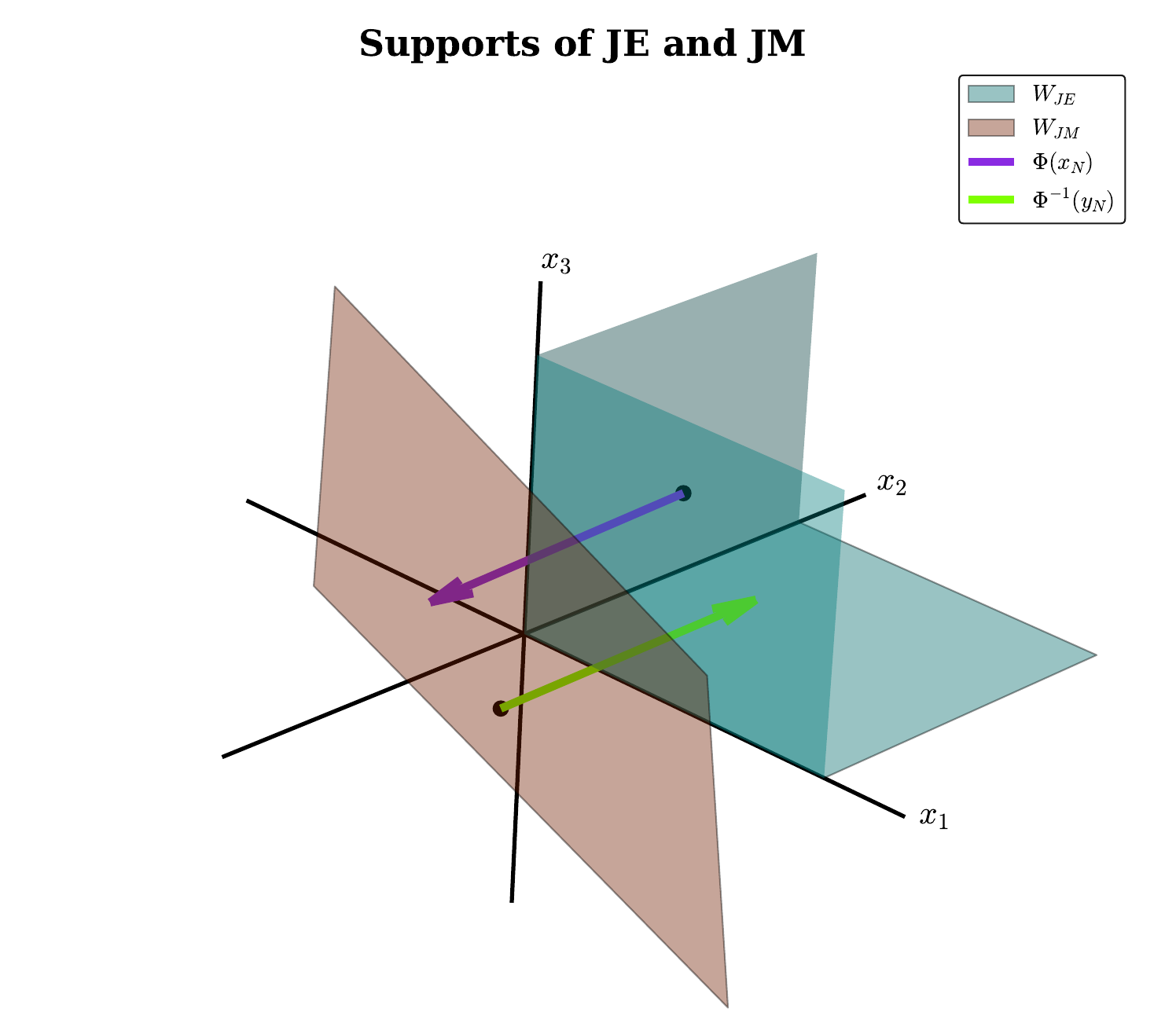}
\caption{The support sets $W_\mathrm{JE}$ and $W_\mathrm{JM}$ for JE and
JM random vectors ($n=3$), and the action of the homeomorphism $\Phi$.}
\label{fig:JEvsJM}
\end{figure}

\noindent The centring map $\Phi(x_\N)=(x_i-\xb,\,i\in\N)$, with
$\xb=\frac1n\sum_{j\in\N}x_j$, is a homeomorphism from $W_\mathrm{JE}$ onto
$W_\mathrm{JM}$ with inverse $\Phi^{-1}(y_\N)=(y_i-\min_{j\in\N}y_j,\,i\in\N)$:
both maps are continuous and map into the stated sets, and direct substitution,
using $\min_kx_k=0$ on $W_\mathrm{JE}$ and $\sum_ky_k=0$ on $W_\mathrm{JM}$,
shows that they are mutually inverse. Push-forward under $\Phi$ is therefore a
bijection between the JE laws and the JM laws with centre zero, as the
arrows in Figure~\ref{fig:JEvsJM} indicate. The map
does not, however, preserve marginals, as the law of $X_i-\Xb$ depends on the
joint law of $X_\N$. This is why the two existence problems differ so sharply
in difficulty: JE with prescribed marginals is governed by the single
condition~\eqref{eq:JECondition}, whereas JM with prescribed marginals admits
explicit criteria only for special families \citep{Wang2016, Puccetti2015}.

A second, more substantive connection arises at the level of the claim
indicators. For $q_\N\in[0,1]^n$ with $\tau=\sum_{i\in\N}q_i$,
Lemma~\ref{lem:wraparound} with $w_i=q_i$ and $\omega=1$ shows that the
indicators $B^*_i=\mathds{1}\{\Theta\in R_i\}$, with $\P(B^*_i=1)=q_i$, satisfy
$\sum_{i\in\N}B^*_i=\tau$ almost surely whenever $\tau$ is an integer. Since a
sum of indicators is integer-valued with mean $\tau$, this recovers the
elementary fact that Bernoulli marginals with parameters $q_1,\dots,q_n$ are
jointly mixable if and only if $\tau$ is an integer; for non-integer $\tau$,
the count $\kappa^*=\sum_{i\in\N}B^*_i$ is, by
Theorem~\ref{thm:JE_Characterisation_1}, as close to constant as possible in
convex order. At the JE boundary $\tau=n-1$, the two notions meet: by
Remark~\ref{rm:JE_Canonical_Remark}(4), the claim indicators of every JE vector
satisfy $\sum_{i\in\N}\mathds{1}\{X_i>0\}=n-1$ almost surely and thus form a
joint mix, and, conversely, any non-negative vector whose claim indicators sum to
$n-1$ is JE. At its boundary, JE is exactly joint mixability of the claim
pattern.

The connection extends from the claim indicators to the losses themselves when
the severities are completely mixable, that is, when copies of the severity can
be coupled so that their sum is almost surely constant \citep{Wang2011,
Puccetti2012}. JE then produces explicit joint mixes of distributions with an
atom at zero, and these are the safest JE structures of all.

\begin{proposition}
\label{prop:CM_atom}
Let $n\ge3$, and let $Y$ be a positive random variable with mean
$\mu_Y<\infty$ whose law $F_Y$ is $(n-1)$-completely mixable, that is, there are
$Z_1,\dots,Z_{n-1}\sim F_Y$ with $\sum_{k=1}^{n-1}Z_k=(n-1)\mu_Y$ almost surely.
Then $F=n^{-1}\delta_0+(1-n^{-1})F_Y$ is $n$-completely mixable: the canonical JE
vector with marginals $F_1=\dots=F_n=F$, face weights $p_{\N\setminus\{k\}}=1/n$,
$k\in\N$, and face copulas equal to a survival copula of $(Z_1,\dots,Z_{n-1})$
satisfies $S=(n-1)\mu_Y$ almost surely.
\end{proposition}

For instance, the uniform distribution on $(0,1)$ is $k$-completely mixable for
every $k\ge2$ \citep{Wang2011}, so the claim distribution
$F=\frac13\delta_0+\frac23\,\mathrm{U}(0,1)$ is $3$-completely mixable: in the
canonical JE vector with counter-monotonic face copulas, exactly two lines claim
in every period and their losses add up to one, as on the face $E_{\{2,3\}}$
of Example~\ref{ex:JE_canonical_example}. The aggregate is then constant,
whereas in the frequency--severity class of Theorem~\ref{thm:safest_FS} the
aggregate $T_2=Y_1+Y_2$ is degenerate only when the severities are, since by
exchangeability a constant $Y_1+Y_2$ would force $Y_1=Y_2=Y_3$ to be constant.
Indeed, this vector admits no frequency--severity representation at all
(Proposition~\ref{prop:FS_rep}), as no three continuous severities can be
pairwise counter-monotonic. Outside the frequency--severity class, the safest
JE structure is therefore governed by the complete mixability of the
zero-truncated severities, for which no general characterisation is available
\citep{Wang2016}.

%+++++++++++++++++++++++++++++++++++++++++++++++++++
\section{Conclusions}
\label{sec:conclusions}

This paper has introduced joint exclusivity as a tractable notion of extremal
negative dependence in arbitrary dimensions. By replacing the pairwise
constraint of mutual exclusivity with the requirement that not all $n$
components are simultaneously positive, JE leaves all lower-order interactions
free; the resulting class is far richer than ME, yet remarkably tractable. An
inclusion--exclusion identity for expectations determines JE laws by their
$(n-1)$-dimensional margins, represents their distribution functions and
transforms, characterises the extension of lower-dimensional margins, and
decomposes aggregate tails and stop-loss premiums into those of sub-portfolios;
JE is preserved by per-risk coverage modifications, and the JE laws with fixed
marginals form a convex, weakly compact face of the Fr\'{e}chet class.

A JE random vector with marginals $F_1,\dots,F_n$ exists if and only if
$\sum_{i\in\N}\F_i(0)\le n-1$, and a single wrap-around construction attains
this criterion and minimises the number of claims in convex order. Every
exclusive vector draws a random admissible pattern and face severities; the
canonical and generalised families separate the allocation of mass across faces
from the choice of face copulas and distortions, and the classical
frequency--severity form is the special case in which all face laws are margins
of one law. Between ME and JE lie the $m$-exclusivity hierarchy and, more
generally, any family of admissible claim patterns, which imposes JE on each
group of lines that may never claim together: admissibility is membership of
the claim probabilities in a polytope, explicit for hierarchical caps, $n+1$
claim scenarios suffice, and the inclusion--exclusion principle persists. Above
thresholds, exclusivity becomes a copula property that suits continuous losses
and contains the tail mutual exclusivity of \citet{Cheung2017} as an extremal
special case. For aggregate losses, the
wrap-around structure is the safest in the frequency--severity model with
exchangeable severities, and, for a common severity law, every exclusivity
level admits an explicit most dangerous structure over the whole Fr\'{e}chet
class; together they yield sharp two-sided stop-loss and TVaR bounds. Beyond
homogeneity, the worst-case TVaR under a cap on simultaneous exceedances of
arbitrary thresholds is the value of an explicit concave programme, concave in
the cap: knowing which losses can occur together has increasing returns
(Example~\ref{ex:units}) and constrains the tail more than correlation does
(Example~\ref{ex:tail_units}). At the JE boundary, the construction produces
explicit joint mixes of distributions with an atom at zero.

Several directions remain open. As VaR is not monotone in convex order, sharp
VaR bounds under exclusivity constraints call for the rearrangement and
mixability techniques of the dependence-uncertainty literature
\citep{Embrechts2013a, Wang2016}. For heterogeneous severities, extremal
premiums within the frequency--severity class are linear programmes over the
pattern polytope~\eqref{eq:pattern_polytope}, whose efficient solution for
large portfolios deserves study, as does the worst-case TVaR under constraints
other than a cap. Finally, statistical inference for JE models calls
for estimators of the face weights and face copulas from claim data with excess
zeros, for tests of exclusivity based on Proposition~\ref{prop:JECDF}(c), and
for estimates of the thresholds above which a cap on simultaneous exceedances
is supported by the data.

%+++++++++++++++++++++++++++++++++++++++++++++++++++
\begin{appendices}
\section{Proofs}
\label{app:proofs}

\subsection{Proofs for Section~\ref{sec:preliminaries}}

\begin{proof}[Proof of Proposition~\ref{prop:JECDF}]
Apply~\eqref{eq:IE_master} to $g(z_\N)=\prod_{i\in\N}g_i(z_i)$ with,
respectively, $g_i(z)=\mathds{1}\{z\le x_i\}$, $g_i(z)=\mathrm{e}^{\mathrm{i}t_iz}$
and $g_i(z)=\mathrm{e}^{-s_iz}$; in each case $g_i(0)=1$, and
$\E[g(X^{(\J)})]$ equals $F_\J(x_\J)$, $\varphi_\J(t_\J)$ and $L_\J(s_\J)$,
respectively. By Lemma~\ref{lem:increments}(c),
$\Delta_{X_\N}g=\prod_{i\in\N}(g_i(X_i)-1)$, which vanishes when
$\min_{i\in\N}X_i=0$, so the indicator in~\eqref{eq:IE_master} may be dropped.
This gives the three identities, since
$\prod_{i\in\N}(\mathds{1}\{X_i\le x_i\}-1)=(-1)^n\mathds{1}\{X_i>x_i,\,i\in\N\}$.

(a) By Lemma~\ref{lem:JE_elementary}(d), the right-hand side
$(-1)^n\F_\N(x_\N)$ vanishes for all $x_\N\in[0,\infty)^n$ if and only if it
vanishes at $0_\N$, if and only if $X_\N$ is JE. Moreover, expanding
$\mathds{1}\{X_\N\le x_\N\}=\prod_{i\in\N}(1-\mathds{1}\{X_i>x_i\})$ gives
$F_\N(x_\N)=\sum_{\J\subseteq\N}(-1)^{|\J|}\F_\J(x_\J)$, so the second sum
in~\eqref{eq:JECDF} also equals $F_\N(x_\N)-(-1)^n\F_\N(x_\N)$.

(c) For $s_\N\in(0,\infty)^n$, the integrand in~\eqref{eq:LST_defect} is
non-negative and vanishes if and only if $\min_{i\in\N}X_i=0$; hence the
right-hand side vanishes, at one or at every such $s_\N$, if and only if
$X_\N$ is JE.

(b) Necessity follows from Theorem~\ref{thm:JE_IE}. Conversely, let $\mu^\circ$
denote the restriction of the law of $X_\N$ to $E_\N$ and, for
$\J\subseteq\N$, let $\mu^\circ_{(\J)}$ be its image under $x_\N\mapsto
x^{(\J)}$. Since the integrand $\prod_{i\in\N}(\mathrm{e}^{\mathrm{i}t_iX_i}-1)$
vanishes off $E_\N$, expanding the product shows that the right-hand side of the
identity in~(b) is the Fourier--Stieltjes transform of the finite signed
measure $\sum_{\J\subseteq\N}(-1)^{n-|\J|}\mu^\circ_{(\J)}$. If it vanishes for
all $t_\N$, this signed measure is zero by the uniqueness theorem. Evaluating it
at $E_\N$, where $\mu^\circ_{(\J)}(E_\N)=0$ for $\J\neq\N$, gives
$\P(X_\N\in E_\N)=\mu^\circ(E_\N)=0$.
\end{proof}

\begin{proof}[Proof of Corollary~\ref{cor:JE_consequences}]
(a) Apply~\eqref{eq:IE_JE} to indicators of Borel sets; each $\mu^{(\J)}$,
$\J\subsetneq\N$, depends only on the law of $X_\J$, a margin of
$X_{\N\setminus\{k\}}$ for any $k\notin\J$. (b) Necessity and uniqueness follow
from (a). Conversely, suppose that $\nu\ge0$. Its total mass is
$\sum_{\J\subsetneq\N}(-1)^{n-1-|\J|}=1-\sum_{\J\subseteq\N}(-1)^{n-|\J|}=1$,
and $\nu(E_\N)=0$ because each $\mu^{(\J)}$, $\J\subsetneq\N$, is concentrated
on points with a vanishing coordinate; thus $\nu$ is the law of a JE vector. Fix
$k\in\N$. For $\J'\subsetneq\N\setminus\{k\}$, the measures $\mu^{(\J')}$ and
$\mu^{(\J'\cup\{k\})}$ have the same image under the deletion of the $k$-th
coordinate, by consistency, and enter $\nu$ with opposite signs; the only
unpaired term, $\J=\N\setminus\{k\}$, has sign $+1$, so the corresponding
$(n-1)$-dimensional margin of $\nu$ is $\mu_{\N\setminus\{k\}}$. (c) Apply
Theorem~\ref{thm:JE_IE} to $g(x_\N)=\mathds{1}\{\sum_ix_i>z\}$,
$\exp(-s\sum_ix_i)$ and $(\sum_ix_i-d)_+$, noting that
$\sum_iX^{(\J)}_i=S_\J$, that the terms with $\J=\emptyset$ vanish in the first
and third cases because $z,d\ge0$, and that $(S_\J-d)_+\le S$.
\end{proof}

\begin{proof}[Proof of Proposition~\ref{prop:JE_closure}]
(a) If $X_k=0$, then $h_k(X_k)=0$. (b) $\min_{i\in\N}X_i\le
\min_{i\in\I}X_i=0$. (c) The map $\Lambda\mapsto\Lambda(E_\N)$ is affine and
non-negative on the Fr\'{e}chet class, and $\mathcal{J}(F_1,\dots,F_n)$ is its
zero set; hence $\mathcal{J}(F_1,\dots,F_n)$ is convex, and if
$\vartheta\Lambda_1+(1-\vartheta)\Lambda_2\in\mathcal{J}(F_1,\dots,F_n)$ with $\vartheta\in(0,1)$, then
$\Lambda_1(E_\N)=\Lambda_2(E_\N)=0$. The Fr\'{e}chet class is tight and weakly closed,
hence weakly compact by Prokhorov's theorem, and $\mathcal{J}(F_1,\dots,F_n)$
is a closed subset: if $\Lambda_k\to\Lambda$ weakly and $\Lambda_k(E_\N)=0$, the
portmanteau theorem for the open set $E_\N$ gives
$\Lambda(E_\N)\le\liminf_k\Lambda_k(E_\N)=0$. Finally,
$S\,\mathds{1}\{S>y\}\le n\sum_{i\in\N}X_i\mathds{1}\{X_i>y/n\}$ for $y>0$ shows
that $S$ is uniformly integrable over the Fr\'{e}chet class, so the continuous
payoff $(S-d)_+\le S$ has weakly continuous expectation there.
\end{proof}

\subsection{Proofs for Section~\ref{sec:main_results}}

\begin{proof}[Proof of Proposition~\ref{prop:canonical_valid}]
(a) For each $\I\in\M$, the right-hand side of~\eqref{eq:JEProbFaces} is $p_\I$
times the joint survival function of a law on $(0,\infty)^{|\I|}$ with
marginal survival functions $\F_i/\F_i(0)$ and survival copula $C_\I$. By
\eqref{eq:JE_parameters_1}, \eqref{eq:JEProbAxes} is a non-negative multiple
of the survival function $\F_i/\F_i(0)$, with total mass $\F_i(0)-\pi_i$. The
faces and semi-axes thus carry total mass
$\sum_{\I\in\M}p_\I+\sum_{i\in\N}(\F_i(0)-\pi_i)=\tau-\sum_{\I\in\M}(|\I|-1)p_\I$,
because $\sum_{i\in\N}\pi_i=\sum_{\I\in\M}|\I|\,p_\I$. The origin therefore
receives~\eqref{eq:JEProbOrigin}, which is non-negative
by~\eqref{eq:JE_parameters_2}. Finally, setting $x_k=0$ for
$k\in\I\setminus\{i\}$ in~\eqref{eq:JEProbFaces} and using $C(u,1,\dots,1)=u$,
\[
\P(X_i>x_i)=\sum_{\I\in\M_i}p_\I\,\frac{\F_i(x_i)}{\F_i(0)}
+\F_i(x_i)\Bigl(1-\frac{\pi_i}{\F_i(0)}\Bigr)=\F_i(x_i),\qquad x_i\ge0.
\]
(b) Necessity follows from (a) and Theorem~\ref{thm:JE_Characterisation_1}.
Conversely, under~\eqref{eq:JECondition}, the vector $X^*_\N$ in the proof of
Theorem~\ref{thm:JE_Characterisation_1} is canonical: with
$p_\I=\P(\{i:\Theta\in R_i\}=\I)$ and $C_\I$ the survival copula of $Y_\I$, it
satisfies~\eqref{eq:JEProbFaces}--\eqref{eq:JEProbOrigin} by independence of
$\Theta$ and $Y_\N$, and $(p_\I)_{\I\in\M}\in\mathcal{P}(F_1,\dots,F_n)$
because its axis masses $\F_i(0)-\pi_i$ and its origin
mass~\eqref{eq:JEProbOrigin} are probabilities, hence non-negative.
(c) For a canonical vector, setting $x_k=0$ for $k\in\I\setminus\{i\}$
in~\eqref{eq:JEProbFaces}, and using~\eqref{eq:JEProbAxes} for $\I=\{i\}$,
gives~\eqref{eq:canonical_char} with $\P(\mathcal{A}=\I)=p_\I$ and
$\P(\mathcal{A}=\{i\})=\F_i(0)-\pi_i$, respectively. Conversely,
under~\eqref{eq:canonical_char}, the face laws $Q_\I$ of
Proposition~\ref{prop:JE_face} have marginal survival functions $\F_i/\F_i(0)$,
so, by Sklar's theorem, \eqref{eq:JEProbFaces} holds with
$p_\I=\P(\mathcal{A}=\I)$ and a survival copula $C_\I$ of $Q_\I$;
\eqref{eq:JEProbAxes} and~\eqref{eq:JEProbOrigin} follow
from~\eqref{eq:face_balance} and~\eqref{eq:canonical_char}, and the weights lie
in $\mathcal{P}(F_1,\dots,F_n)$ as in part~(b).
\end{proof}

\begin{proof}[Proof of Corollary~\ref{cor:JE_Canonical_CDF_phi}]
In the second sum of~\eqref{eq:JECDF}, the terms with $|\J|\le1$ contribute
$1-\sum_{i\in\N}\F_i(x_i)$. For $\I\in\M$, the event
$\{X_i>x_i,\,i\in\I\}$ with $x_\I\ge0$ can occur only on faces $E_\J$ with
$\J\supseteq\I$, $\J\in\M$, and setting $x_j=0$ for $j\in\J\setminus\I$
in~\eqref{eq:JEProbFaces} gives
$\F_\I(x_\I)=\sum_{\J\in\M,\,\J\supseteq\I}p_\J\,C_\J(\F_i(x_i)/\F_i(0),\,i\in\I;
\,1,\,j\in\J\setminus\I)$. The converse holds because a CDF determines the law.
\end{proof}

\begin{proof}[Proof of Proposition~\ref{prop:face_allocation}]
\textit{Upper bound.} Let $p$ be feasible and let $r\in[0,\infty)^n$ satisfy
$\sum_{i\in\I}r_i\ge|\I|-1$ for all $\I\in\M^{(m)}$. Then
\[
\sum_{\I\in\M^{(m)}}(|\I|-1)p_\I\le\sum_{\I\in\M^{(m)}}p_\I\sum_{i\in\I}r_i
=\sum_{i\in\N}r_i\sum_{\I\in\M^{(m)},\,\I\ni i}p_\I\le\sum_{i\in\N}r_ic_i.
\]
The choice $r_i\equiv(m-2)/(m-1)$ is admissible, because
$|\I|(m-2)/(m-1)\ge|\I|-1$ if and only if $|\I|\le m-1$, and yields the first
bound; the choice $r=\one-e_j$, where $e_j$ is the $j$-th unit vector, is
admissible and yields $\|c\|_1-c_j$ for every $j$.

\textit{Attainment.} The case $c=0$ is trivial. Otherwise, let
$\omega=\max\{\|c\|_\infty,\,\|c\|_1/(m-1)\}$. Lemma~\ref{lem:wraparound} with
$w_i=c_i$ provides arcs $R_i$ of lengths $c_i$ on a circle of circumference
$\omega$ with covering number $\kappa\le\lceil \|c\|_1/\omega\rceil\le m-1$.
For $\I\in\M^{(m)}$, let $p_\I$ be the Lebesgue measure of
$\{\theta\in[0,\omega):\{i:\theta\in R_i\}=\I\}$. Then
$\sum_{\I\ni i}p_\I\le|R_i|=c_i$, and
\[
\sum_{\I\in\M^{(m)}}(|\I|-1)p_\I=\int_0^\omega(\kappa(\theta)-1)_+\,\d\theta
\ge\int_0^\omega(\kappa(\theta)-1)\,\d\theta=\|c\|_1-\omega,
\]
which equals the right-hand side of~\eqref{eq:face_allocation_value}.
\end{proof}

\begin{proof}[Proof of Proposition~\ref{prop:JE_trivariate}]
Write $\varpi(\lambda)=\sum_{\I\in\M}p_\I(\lambda)$ for the total face mass, which
is affine and non-decreasing in $\lambda$. Any two capacities satisfy
$c_i+c_k\ge\|c\|_1-\|c\|_\infty\ge\tau-1$, which gives
$\underline{p}_\I\le\overline{p}_\I$, hence
$\underline{p}_\I\le p_\I(\lambda)\le\overline{p}_\I$. At $\lambda=1$,
$\varpi(1)=\sum_{\I\in\M}\overline{p}_\I\ge\|c\|_1-\|c\|_\infty\ge\tau-1$, since
the sum of the three pairwise minima is at least the sum of the two smallest
capacities; so the total-mass constraint holds exactly on $[\lambda_-,1]$. It
remains to verify the capacity constraints at $\lambda=\lambda_-$; since they
are affine in $\lambda$, they then hold on an interval $[\lambda_-,\lambda_+]$.
Fix $i$ and let $\{j,k\}=\N\setminus\{i\}$. If $\lambda_->0$, then
$\varpi(\lambda_-)=\tau-1$ and
\[
\sum_{\I\in\M_i}p_\I(\lambda_-)=\tau-1-p_{\N\setminus\{i\}}(\lambda_-)
\le\tau-1-\underline{p}_{\N\setminus\{i\}}\le c_i.
\]
If $\lambda_-=0$, then
$\sum_{\I\in\M_i}p_\I(0)=(\tau-1-c_k)_++(\tau-1-c_j)_+$, which is at most $c_i$
because $\tau-1\le c_i+c_k$, $\tau-1\le c_i+c_j$ and $2(\tau-1)\le\|c\|_1$.
\end{proof}

\begin{proof}[Proof of Theorem~\ref{thm:G_JE_characterization}]
Fix face weights $p_\I\ge0$. The axis function $\overline{H}_i$ is
right-continuous, vanishes at infinity and equals $\F_i(0)-\pi_i$ at zero. It
is non-increasing if and only if
$\F_i(x)-\F_i(y)\ge\pi_i\bigl(G_i(\F_i(x))-G_i(\F_i(y))\bigr)$ for all $x<y$,
that is, if and only if~\eqref{eq:G-JE_parameters_1} holds; it is then also
non-negative. As in the proof of Proposition~\ref{prop:canonical_valid}, the
origin mass is~\eqref{eq:JEProbOrigin}, which is non-negative if and only
if~\eqref{eq:JE_parameters_2} holds, and
$\P(X_i>x)=\overline{H}_i(x)+\pi_iG_i(\F_i(x))=\F_i(x)$. Hence a G-JE vector
exists if and only if some $p\ge0$ satisfies~\eqref{eq:G-JE_parameters_1}
and~\eqref{eq:JE_parameters_2}, that is, if and only if
$V_n(G_1^*,\dots,G_n^*)\ge\tau-1$. By Proposition~\ref{prop:face_allocation},
this is~\eqref{eq:G-JECondition}.
\end{proof}

\begin{proof}[Proof of Proposition~\ref{prop:count_range}]
The lower bound is Theorem~\ref{thm:JE_Characterisation_1}. For the upper
bound, $\kappa=|\mathcal{A}(X_\N)|$ takes values in $\{0,\dots,m-1\}$ and has
mean $\tau$; for convex $f$, $f(k)\le(1-\frac{k}{m-1})f(0)+\frac{k}{m-1}f(m-1)$
for $k\in[0,m-1]$, and taking expectations gives
$\E f(\kappa)\le\E f(\kappa^\sharp_m)$; the same chord argument on
$[0,m'-1]$ gives the final claim. If $\kappa\stackrel{d}{=}\kappa^\sharp_m$,
then $\{X_i>0\}\subseteq\{\kappa=m-1\}$ forces
$\F_i(0)\le\P(\kappa=m-1)=\tau/(m-1)$. Conversely, suppose that
$\max_{i\in\N}\F_i(0)\le\tau/(m-1)$; the case $\tau=0$, in which
$X_\N=0_\N$ almost surely, is trivial, so let $\tau>0$ and apply
Lemma~\ref{lem:wraparound} with $w_i=\F_i(0)$ and $\omega=\tau/(m-1)\le1$: the
arcs $R_i\subseteq[0,\omega)$ cover every point exactly $m-1$ times. With
$\Theta$ and $Y_\N$ as in the proof of Theorem~\ref{thm:JE_Characterisation_1},
the vector $X^\sharp_i=Y_i\,\mathds{1}\{\Theta\in R_i\}$ has marginals
$F_1,\dots,F_n$, and $|\mathcal{A}(X^\sharp_\N)|$ equals $m-1$ on
$\{\Theta<\omega\}$ and $0$ otherwise.
\end{proof}

\begin{proof}[Proof of Theorem~\ref{thm:families}]
(a)$\Rightarrow$(b). If $X_\N$ is $\mathcal{K}$-exclusive, its pattern law
$p_\I=\P(\mathcal{A}=\I)$ vanishes off $\mathcal{K}$, and~\eqref{eq:face_balance}
gives $q_\N=\sum_{\I\in\mathcal{K}}p_\I\mathbf{1}_\I\in\mathcal{Q}(\mathcal{K})$.
(b)$\Rightarrow$(a), together with the final statement. By the same
argument, the pattern law of every $\mathcal{K}$-exclusive vector with
marginals $F_1,\dots,F_n$ lies in $\mathcal{P}_{\mathcal{K}}(q_\N)$; conversely,
let $(p_\I)_{\I\in\mathcal{K}}\in\mathcal{P}_{\mathcal{K}}(q_\N)$, so that
$q_\N=\sum_{\I\in\mathcal{K}}p_\I\mathbf{1}_\I$. Draw a random set
$\mathcal{A}'$ with law $(p_\I)$ and, independently, any random vector $Y_\N$
with $\P(Y_i>y)=\F_i(y)/\F_i(0)$, $y\ge0$, for each $i$ with $\F_i(0)>0$ (and
$Y_i\equiv1$ otherwise). Then $X_i=Y_i\mathds{1}\{i\in\mathcal{A}'\}$ has
$\P(X_i>x)=\F_i(0)\,\F_i(x)/\F_i(0)=\F_i(x)$ for $x\ge0$, and
$\mathcal{A}(X_\N)=\mathcal{A}'\in\mathcal{K}$ almost surely, because a line
with $\F_i(0)=0$ belongs to no $\I$ with $p_\I>0$. The dependence between the
components of $Y_\I$ on $\{\mathcal{A}'=\I\}$ is arbitrary, so any survival
copulas may be used on the faces, as in Proposition~\ref{prop:JE_face}, and a
single $Y_\N$ for all faces gives the frequency--severity form. If
$p\in\mathcal{P}_{\mathcal{K}}(q_\N)$ charges more than $n+1$ patterns, the
vectors $(1,\mathbf{1}_\I)\in\RR^{n+1}$ of the charged patterns are linearly
dependent, and moving $p$ along a dependence in either direction keeps it in
$\mathcal{P}_{\mathcal{K}}(q_\N)$ for small steps; hence $p$ is not a vertex.
(b)$\Leftrightarrow$(c). If $q_\N=\sum_{\I}p_\I\mathbf{1}_\I$, then
$\sum_iw_i\F_i(0)=\sum_\I p_\I\sum_{i\in\I}w_i\le h_{\mathcal{K}}(w)$. For the
converse, note first that $\mathcal{Q}(\mathcal{K})$ is \emph{down-monotone}:
if $x\in\mathcal{Q}(\mathcal{K})$ and $0\le y\le x$ componentwise, then
$y\in\mathcal{Q}(\mathcal{K})$. Indeed, it suffices to lower one coordinate,
and if $x=\sum_\I p_\I\mathbf{1}_\I$ and $\vartheta\in[0,1]$, then
$x-\vartheta x_ie_i=\sum_\I p_\I\bigl((1-\vartheta)\mathbf{1}_\I+\vartheta\mathbf{1}_{\I\setminus\{i\}}\bigr)$,
where $\I\setminus\{i\}\in\mathcal{K}$ because $\mathcal{K}$ is closed under
taking subsets. Now let $q_\N\notin\mathcal{Q}(\mathcal{K})$. As
$\mathcal{Q}(\mathcal{K})$ is a compact convex set, there is $w\in\RR^n$ with
$\sum_iw_i\F_i(0)>\max_{x\in\mathcal{Q}(\mathcal{K})}\sum_iw_ix_i$. Let
$w^+_i=\max\{w_i,0\}$. For $x\in\mathcal{Q}(\mathcal{K})$, the vector $x'$
obtained by setting $x_i=0$ whenever $w_i<0$ lies in
$\mathcal{Q}(\mathcal{K})$ by down-monotonicity, and
$\sum_iw^+_ix_i=\sum_iw_ix'_i$. Hence
\[
\max_{x\in\mathcal{Q}(\mathcal{K})}\sum_iw^+_ix_i
\le\max_{x\in\mathcal{Q}(\mathcal{K})}\sum_iw_ix_i
<\sum_iw_i\F_i(0)\le\sum_iw^+_i\F_i(0).
\]
Since a linear function attains its maximum over $\mathcal{Q}(\mathcal{K})$ at
a vertex $\mathbf{1}_\I$, $\I\in\mathcal{K}$, the left-hand side equals
$h_{\mathcal{K}}(w^+)$, so~(c) fails for $w^+\in[0,\infty)^n$.
\end{proof}

\begin{proof}[Proof of Corollary~\ref{cor:laminar}]
Necessity follows from $\E|\mathcal{A}\cap L|=\sum_{i\in L}\F_i(0)$. A
labelling of the required kind exists because the members of a laminar family
form a forest under inclusion: listing the lines by a depth-first traversal of
this forest makes every member a block of consecutive indices. For such a
labelling, the arcs of the lines in $L$ form a block in
Lemma~\ref{lem:wraparound} with $\omega=1$, so
$|\mathcal{A}(X^*_\N)\cap L|\in\{\lfloor x_L\rfloor,\lceil x_L\rceil\}$, where
$x_L=\sum_{i\in L}\F_i(0)$; as $c_L$ is an integer, \eqref{eq:laminar} gives
$\lceil x_L\rceil\le c_L$. Convex-order minimality follows from the Jensen
argument in the proof of Theorem~\ref{thm:JE_Characterisation_1}, applied to
the integer-valued variable $|\mathcal{A}\cap L|$ with mean $x_L$.
\end{proof}

\begin{proof}[Proof of Proposition~\ref{prop:K_rep}]
As admissible patterns are those that contain no forbidden group,
$\{\mathcal{A}\notin\mathcal{K}\}=\bigcup_{\J\in\mathcal{F}(\mathcal{K})}
\{X_i>0,\ i\in\J\}=\bigcup_{\I\notin\mathcal{K}}\{\mathcal{A}=\I\}$, which gives
(a)$\Leftrightarrow$(b)$\Leftrightarrow$(c). If $C$ is a copula of $X_\N$ and
$U_\N\sim C$, then $(F_i^{-1}(U_i),\,i\in\N)\stackrel{d}{=}X_\N$ by Sklar's
theorem, and $F_i^{-1}(u)>0$ if and only if $u>F_i(0)$; hence
$\P(X_i>0,\ i\in\J)=\P(U_i>F_i(0),\ i\in\J)$ for every $\J\subseteq\N$, which
gives (b)$\Leftrightarrow$(d). In a representation $(B_iY_i)$ with $Y_\N>0$,
the active set is $\{i:B_i=1\}$, so $B_\N$ has the pattern law of $X_\N$ and is
$\mathcal{K}$-exclusive by~(c); independence gives
$\F_i(y)=\P(B_i=1)\P(Y_i>y)$ for $y\ge0$, and the remaining claim is
Proposition~\ref{prop:FS_rep}. Finally, let $\mathcal{A}'$ have law $(p_\I)$
and, independently, let $Y_\N$ be any random vector with
$\P(Y_i>y)=\F_i(y)/\F_i(0)$, $y\ge0$, whenever $\F_i(0)>0$ (and $Y_i\equiv1$
otherwise). By~\eqref{eq:face_balance}, $\P(i\in\mathcal{A}')=\F_i(0)$, so
$(\mathds{1}\{i\in\mathcal{A}'\}Y_i,\,i\in\N)$ has marginals $F_1,\dots,F_n$ and
pattern law $(p_\I)$, and its face laws are margins of the law of $Y_\N$.
\end{proof}

\begin{proof}[Proof of Proposition~\ref{prop:K_IE}(a)]
Let $x_\N\in[0,\infty)^n$ have its set of positive coordinates in
$\mathcal{K}$. By Lemma~\ref{lem:increments}(b), $\Delta^\J_{x_\N}g=0$ unless
$\J$ is contained in this set, hence unless $\J\in\mathcal{K}$, so
Lemma~\ref{lem:increments}(a) and~\eqref{eq:rect_increment} give
\[
g(x_\N)=\sum_{\J\in\mathcal{K}}\Delta^\J_{x_\N}g
=\sum_{\J\in\mathcal{K}}\sum_{\I\subseteq\J}(-1)^{|\J|-|\I|}g\bigl(x^{(\I)}\bigr)
=\sum_{\I\in\mathcal{K}}\chi_{\mathcal{K}}(\I)\,g\bigl(x^{(\I)}\bigr),
\]
as every subset of a member of $\mathcal{K}$ lies in $\mathcal{K}$. Taking
expectations gives~\eqref{eq:K_IE} under $\mathcal{K}$-exclusivity, and each
$X^{(\I)}$, $\I\in\mathcal{K}$, is a function of the margin of $X_\N$ on a
maximal admissible pattern containing $\I$. Conversely, the indicator $g$ of
$\{x_\N:\{i:x_i>0\}\notin\mathcal{K}\}$ vanishes at every $X^{(\I)}$ with
$\I\in\mathcal{K}$, so~\eqref{eq:K_IE} gives $\P(\mathcal{A}\notin\mathcal{K})=0$.
\end{proof}

\begin{proof}[Proof of Proposition~\ref{prop:K_IE}(b)]
Every $\I\in\mathcal{K}$ lies in a maximal admissible pattern, and a set
contained in two of them lies in their intersection, so consistency makes each
$\mu^{(\I)}$ well defined. Necessity and uniqueness follow from part~(a)
applied to indicators of Borel sets. Conversely, suppose that $\nu\ge0$. Each
$\mu^{(\I)}$ is concentrated on points whose positive coordinates lie in
$\I\in\mathcal{K}$, and the identity in the proof of~(a), at $x_\N=0_\N$ with
$g\equiv1$, gives $\sum_{\I\in\mathcal{K}}\chi_{\mathcal{K}}(\I)=1$, the total
mass of $\nu$; thus $\nu$ is the law of a $\mathcal{K}$-exclusive vector. Let
$F\in\mathcal{K}$ be maximal and let $\pi_F$ denote the projection onto the
coordinates in $F$. By consistency, $\pi_F\mu^{(\I)}=\pi_F\mu^{(\I\cap F)}$, so
$\pi_F\nu=\sum_{L\subseteq F}a_L\,\pi_F\mu^{(L)}$, where, writing $\I=L\cup M$
with $M\subseteq\J\setminus F$,
\[
a_L=\sum_{\I\in\mathcal{K},\,\I\cap F=L}\chi_{\mathcal{K}}(\I)
=\sum_{\J\in\mathcal{K},\,\J\supseteq L}(-1)^{|\J|-|L|}
\sum_{M\subseteq\J\setminus F}(-1)^{|M|}
=\sum_{L\subseteq\J\subseteq F}(-1)^{|\J|-|L|},
\]
because the inner sum vanishes unless $\J\subseteq F$, and every such $\J$
lies in $\mathcal{K}$. Hence $a_L=1$ if $L=F$ and $a_L=0$ otherwise, and
$\pi_F\nu=\mu_F$.
\end{proof}

\subsection{Proofs for Section~\ref{sec:further_discussion}}

\begin{proof}[Proof of Proposition~\ref{prop:threshold}]
The number of exceedances $\sum_{i\in\N}\mathds{1}\{Z_i>l_i\}$ is at most
$m-1$ and has expectation $\sum_{i\in\N}\P(Z_i>l_i)$, which gives necessity.
Under~\eqref{eq:m-exclusive_threshold}, Theorem~\ref{thm:m-exclusivity}
provides an $m$-exclusive $Z^+_\N$ with the required marginals. Since
$Z^-_i\le l_i$, $\{\widetilde Z_i>l_i\}=\{Z^+_i>0\}$, so $\widetilde Z_\N$ is
$m$-exclusive above $l_\N$. For $z\ge l_i$,
$\P(\widetilde Z_i>z)=\P(Z^+_i>z-l_i)=\P(Z_i>z)$, because $Z^-_i\le l_i\le z$.
For $z<l_i$, independence gives
$\P(\widetilde Z_i>z)=\P(Z_i>l_i)+\P(Z_i\le l_i)\,\P(Z^-_i>z)
=\P(Z_i>l_i)+\P(z<Z_i\le l_i)=\P(Z_i>z)$, where the middle equality holds
trivially when $\P(Z_i\le l_i)=0$.
\end{proof}

\begin{proof}[Proof of Proposition~\ref{prop:SL_lower}]
Let $x_\N\in[0,\infty)^n$, $d\ge0$ and $\mathcal{A}_d=\{i:x_i>d\}$. If
$\mathcal{A}_d=\emptyset$, then $\sum_i(x_i-d)_+=0$; otherwise
$\sum_i(x_i-d)_+=\sum_{i\in\mathcal{A}_d}x_i-|\mathcal{A}_d|\,d\le\sum_ix_i-d$,
with strict inequality if $|\mathcal{A}_d|\ge2$ and $d>0$. Hence
$(\sum_ix_i-d)_+\ge\sum_i(x_i-d)_+$, which gives the bound; equality holds
pointwise when at most one component is positive, so ME yields equality. If
$X_\N$ is not ME, then $\P(X_i>1/k,\,X_j>1/k)>0$ for some $i\ne j$ and $k\ge1$,
and for $d=1/(2k)$ the inequality is strict on this event, so
$\Pi_S(d)>\sum_i\Pi_{X_i}(d)$. Finally,~\eqref{eq:JE_var_bound} follows from
Tonelli's theorem and $\int_0^\infty(z-d)_+\,\d d=z^2/2$ for $z\ge0$, which give
$\int_0^\infty\Pi_S(d)\,\d d=\frac12\E[S^2]$ and
$\int_0^\infty\Pi_{X_i}(d)\,\d d=\frac12\E[X_i^2]$.
\end{proof}

\begin{proof}[Proof of Theorem~\ref{thm:safest_FS}]
Let $\kappa=\sum_{i\in\N}B_i=|\mathcal{A}(X_\N)|$. Given $B_\N$ with active set
$\I$, the aggregate $S=\sum_{i\in\I}Y_i$ has, by exchangeability and
independence of $Y_\N$ and $B_\N$, the law of $T_{|\I|}$. Hence
$\Pi_S(d)=\E\psi(\kappa)$ with $\psi(k)=\Pi_{T_k}(d)$, $k=0,\dots,n$, for each
fixed $d\ge0$. For $y\ge0$, the map $z\mapsto(z+y-d)_+-(z-d)_+$ is
non-decreasing. Applying it with $z=T_{k+1}\ge T_k$ and $y=Y_{k+2}$, and using
that $(Y_1,\dots,Y_k,Y_{k+2})$ and $(Y_1,\dots,Y_{k+1})$ have the same law, we
obtain, for $0\le k\le n-2$,
\[
\begin{aligned}
\psi(k+2)-\psi(k+1)
&=\E\bigl[(T_{k+1}+Y_{k+2}-d)_+-(T_{k+1}-d)_+\bigr]\\
&\ge\E\bigl[(T_k+Y_{k+2}-d)_+-(T_k-d)_+\bigr]
=\psi(k+1)-\psi(k).
\end{aligned}
\]
Thus $\psi$ is convex on $\{0,1,\dots,n\}$, and so is its piecewise linear
extension. (a) By Theorem~\ref{thm:JE_Characterisation_1},
$\kappa^*\le_{\mathrm{cx}}\kappa$, so
$\Pi_{S^*}(d)=\E\psi(\kappa^*)\le\E\psi(\kappa)=\Pi_S(d)$; as
$\E[S^*]=\tau\,\E[Y_1]=\E[S]$, this is convex order.
Identity~\eqref{eq:SL_star} follows from
$\P(\kappa^*=\lfloor\tau\rfloor+1)=\tau-\lfloor\tau\rfloor=1-\P(\kappa^*=\lfloor\tau\rfloor)$,
and the exclusivity statements are Theorems~\ref{thm:JE_Characterisation_1}
and~\ref{thm:m-exclusivity}. (b) At $d=0$, all three members
of~\eqref{eq:SL_range} equal $\tau\,\E[Y_1]$. By
Proposition~\ref{prop:count_range}, $\kappa\le_{\mathrm{cx}}\kappa^\sharp_m$,
so $\Pi_S(d)\le\E\psi(\kappa^\sharp_m)=\frac{\tau}{m-1}\psi(m-1)$, as
$\psi(0)=0$; monotonicity in $m$ follows from
$\kappa^\sharp_m\le_{\mathrm{cx}}\kappa^\sharp_{m'}$. If
$\max_{i\in\N}q_i\le\tau/(m-1)$, the vector $X^\sharp_\N$ of
Proposition~\ref{prop:count_range}, built with the severities $Y_\N$
independent of $\Theta$, lies in $\mathcal{C}(q_\N,Y_\N)$ and attains the
bound. Conversely, if equality holds for all $d\ge0$, then $S$ has the law of
$T_{\kappa^\sharp_m}$ with $\kappa^\sharp_m$ independent of $Y_\N$, since
stop-loss transforms determine the law of a non-negative random variable. As
$T_k>0$ for $k\ge1$, $\P(\kappa=0)=\P(S=0)=1-\tau/(m-1)$, so $\kappa$ takes the
value $m-1$ whenever it is positive, and Proposition~\ref{prop:count_range}
gives $\max_{i\in\N}q_i\le\tau/(m-1)$.
\end{proof}

\begin{proof}[Proof of Theorem~\ref{thm:cx_max}]
\emph{Step 1: the parameters.} Since $t_{c-1}=\sum_{j\ge c}q_j\ge q_c$, the
set defining $r$ is non-empty and $r\le c-1$. If $r\ge1$, then
$q_r>t_{r-1}=(q_r+(c-r)t)/(c-r+1)$, that is, $q_r>t$; hence $q_i>t$ for all
$i\le r$. Moreover, $q_j\le q_{r+1}\le t$ for $j>r$, and $t\le1$, because
$t=\tau/c\le1$ if $r=0$ and $t<q_r\le1$ if $r\ge1$. Recall that
$F_i^{-1}(u)=F_Y^{-1}(1-(1-u)/q_i)$ for $u>1-q_i$ and $F_i^{-1}(u)=0$ otherwise.

\emph{Step 2: attainment.} Let $U$ and $\Theta$ be independent and uniform on
$(0,1)$ and $[0,1)$, and let $V=(U-1+t)/t$. If $t>0$, apply
Lemma~\ref{lem:wraparound} with $\omega=1$ to the lengths $w_j=q_j/t\le1$,
$j=r+1,\dots,n$, whose total is $c-r$; the resulting arcs $R_j$ cover every
point exactly $c-r$ times. Set $X^\sharp_i=F_i^{-1}(U)$ for $i\le r$ and
$X^\sharp_j=F_Y^{-1}(V)\mathds{1}\{U>1-t\}\mathds{1}\{\Theta\in R_j\}$ for $j>r$
(and $X^\sharp_j=0$ if $t=0$). Since $V$ is uniform on $(0,1)$ given
$\{U>1-t\}$ and independent of $\Theta$,
$\P(X^\sharp_j>x)=t\cdot(q_j/t)\cdot\F_Y(x)=\F_j(x)$, $x\ge0$. On
$\{U>1-t\}$, all lines $i\le r$ are active, because $U>1-t>1-q_i$, and exactly
$c-r$ lines $j>r$ are active; on $\{U\le1-t\}$, at most $r\le c-1$ lines are
active. Hence $X^\sharp_\N$ is $m$-exclusive with marginals $F_1,\dots,F_n$ and
aggregate $S^\sharp$.

\emph{Step 3: a co-monotonic split.} For $u\in(0,1)$, let
$z_i(u)=F_i^{-1}(u)$, $i=1,\dots,r$, and
$z_0(u)=(c-r)F_Y^{-1}((u-1+t)/t)\mathds{1}\{u>1-t\}$. These functions are
non-negative and non-decreasing, and $S^\sharp=\sum_{i=0}^rz_i(U)$. Let
$X_\N$ be $m$-exclusive with marginals $F_1,\dots,F_n$ and fix $d\ge0$. Let
$z_\Sigma=\sum_{i=0}^rz_i$ and let $g^*=F_Y^{-1}(1-)\le\infty$ be the right endpoint
of the support of $F_Y$. If $z_\Sigma\le d$ on $(0,1)$, then
$\sup_uz_\Sigma(u)=k^*g^*\le d$, where $k^*=\min\{c,|\{i:q_i>0\}|\}$, since
$t>0$ forces at least $c-r$ lines $j>r$ with $q_j>0$; as at most $k^*$
components of $X_\N$ are positive and each is at most $g^*$, both $\Pi_S(d)$ and
$\Pi_{S^\sharp}(d)$ vanish. Otherwise, let
$u_d=\sup\{u\in(0,1):z_\Sigma(u)\le d\}<1$ (with $\sup\emptyset=0$), and write
$z_i(u_d-)$ and $z_i(u_d+)$ for one-sided limits, with $z_i(0-)=0$. Then
$\sum_iz_i(u_d-)\le d\le\sum_iz_i(u_d+)$, so there are
$\varepsilon_i\in[z_i(u_d-),z_i(u_d+)]$, $i=0,\dots,r$, with $\sum_i\varepsilon_i=d$. For
$u<u_d$, $z_i(u)\le \varepsilon_i$ for all $i$, and for $u>u_d$, $z_i(u)\ge \varepsilon_i$ for all
$i$; hence $(S^\sharp-d)_+=\sum_{i=0}^r(z_i(U)-\varepsilon_i)_+$ almost surely and
\begin{equation}
\label{eq:split}
\Pi_{S^\sharp}(d)=\sum_{i=0}^r\E\bigl[(z_i(U)-\varepsilon_i)_+\bigr].
\end{equation}

\emph{Step 4: ordering of the split.} Let $\delta=\varepsilon_0/(c-r)$. We claim that
$\delta\le \varepsilon_i$ for $i\le r$. This is trivial if $\varepsilon_0=0$. Otherwise
$z_0(u_d+)\ge \varepsilon_0>0$, which forces $u_d\ge1-t$, as $z_0$ vanishes on
$(0,1-t]$. Since $u_d<1$ and $q_i>t$,
$a=1-(1-u_d)/t<b=1-(1-u_d)/q_i$, with $a\ge0$. As $F_Y^{-1}(a+)\le F_Y^{-1}(b)$,
we obtain $\delta\le z_0(u_d+)/(c-r)=F_Y^{-1}(a+)\le F_Y^{-1}(b)=z_i(u_d)
=z_i(u_d-)\le \varepsilon_i$, by left-continuity of $z_i$.

\emph{Step 5: conclusion.} Let $d_i=\varepsilon_i$ for $i\le r$ and $d_j=\delta$ for
$j>r$. By Step~4 and $n-r\ge c-r$, the sum of the $c$ largest $d_i$ is
$\sum_{i\le r}\varepsilon_i+(c-r)\delta=d$, so $d_\N\in\mathcal{D}_m(d)$, and
Proposition~\ref{prop:SL_upper} gives
$\Pi_S(d)\le\sum_{i\le r}\Pi_{X_i}(\varepsilon_i)+\sum_{j>r}\Pi_{X_j}(\delta)$. Here
$\Pi_{X_i}(\varepsilon_i)=\E[(z_i(U)-\varepsilon_i)_+]$, as $z_i(U)\sim F_i$, and
\[
\sum_{j>r}\Pi_{X_j}(\delta)=\sum_{j>r}q_j\,\E[(Y-\delta)_+]
=t\,\E\bigl[((c-r)Y-\varepsilon_0)_+\bigr]=\E\bigl[(z_0(U)-\varepsilon_0)_+\bigr],
\]
By~\eqref{eq:split}, $\Pi_S(d)\le\Pi_{S^\sharp}(d)$ for every
$d\ge0$. Since $\E[S^\sharp]=(\sum_{i\le r}q_i+(c-r)t)\E[Y]=\tau\E[Y]=\E[S]$,
it follows that $S\le_{\mathrm{cx}}S^\sharp$. If $q_1\le\tau/c$, then $r=0$,
$t=\tau/c$ and $S^\sharp=cF_Y^{-1}(V)\mathds{1}\{U>1-t\}$, where $F_Y^{-1}(V)\sim F_Y$
is independent of $\{U>1-t\}$.
\end{proof}

\begin{proof}[Proof of Corollary~\ref{cor:TVaR}]
Both parts follow because TVaR is non-decreasing in convex order. In
part~(a),
$Z=T_{m-1}\mathds{1}\{U\le\tau/(m-1)\}$ has $\Pi_Z=\frac{\tau}{m-1}\Pi_{T_{m-1}}$
on $[0,\infty)$ and mean $\tau\E[Y_1]=\E[S]$. For the closed form in part~(b), $S^\sharp$ vanishes on a set of
probability $1-\tau/(m-1)$ and equals $(m-1)Y$ otherwise, so its quantile
function vanishes on $(0,1-\tau/(m-1)]$ and equals
$(m-1)F_Y^{-1}(1-(1-u)(m-1)/\tau)$ above.
\end{proof}

\begin{proof}[Proof of Theorem~\ref{thm:TVaR_threshold}]
Write $\bar\alpha=1-\alpha$. \emph{A bathtub bound.} Let $Z$ have distribution
function $F_Z$ and finite mean, let $U$ be uniform on $(0,1)$ with
$Z=F_Z^{-1}(U)$ almost surely, and let $\Omega_0\subseteq\{U\in(j_0,j_1]\}$ be
an event. Then $\E[Z\mathds{1}_{\Omega_0}]=\int_{j_0}^{j_1}F_Z^{-1}(u)h(u)\,\d u$
with $h(u)=\P(\Omega_0\mid U=u)\in[0,1]$ and $\int_{j_0}^{j_1}h=\P(\Omega_0)$; as
$F_Z^{-1}$ is non-decreasing,
$\E[Z\mathds{1}_{\Omega_0}]\le\int_{j_1-\P(\Omega_0)}^{j_1}F_Z^{-1}(u)\,\d u$. With
$(j_0,j_1]=(0,1]$, this gives $\E[S\mathds{1}_A]\le\bar\alpha\,
\mathrm{TVaR}_\alpha(S)$ whenever $\P(A)=\bar\alpha$, with equality for
$A=\{V>\alpha\}$ if $V$ is uniform with $S=F_S^{-1}(V)$ almost surely.

\emph{Upper bound.} Enlarging the probability space if necessary, let $V$ and
$U_i$ be distributional transforms of $S$ and $X_i$ \citep{Ruschendorf2013},
so that $S=F_S^{-1}(V)$ and $X_i=F_i^{-1}(U_i)$ almost surely, and let
$A=\{V>\alpha\}$ and $H_i=\{X_i>l_i\}=\{U_i>1-p_i\}$. The probabilities
$\lambda_i=\P(A\cap H_i)$ satisfy $\lambda_i\le\min\{p_i,\bar\alpha\}$ and
$\bar\alpha-\lambda_i=\P(A\setminus H_i)\le1-p_i$, so $\lambda_i\in\Lambda_i$,
and $\sum_i\lambda_i=\E[\mathds{1}_A\sum_i\mathds{1}_{H_i}]\le c\bar\alpha$ by
exclusivity above $l_\N$. The bathtub bound with $(j_0,j_1]=(1-p_i,1]$ and
$(0,1-p_i]$ gives $\E[X_i\mathds{1}_{A\cap H_i}]+\E[X_i\mathds{1}_{A\setminus
H_i}]\le g_i(\lambda_i)$, so $\bar\alpha\,\mathrm{TVaR}_\alpha(S)=
\E[S\mathds{1}_A]\le\sum_ig_i(\lambda_i)$.

\emph{The programme.} Each $g_i$ is continuous; it is concave, as
$\lambda\mapsto\int_{1-\lambda}^1F_i^{-1}$ has the non-increasing derivative
$F_i^{-1}(1-\lambda)$ and $\lambda\mapsto\int_{\underline{u}_i+\lambda}^{1-p_i}F_i^{-1}$, with
$\underline{u}_i=1-p_i-\bar\alpha$, the non-increasing derivative
$-F_i^{-1}(\underline{u}_i+\lambda)$
(almost everywhere); and it is non-decreasing, because
$1-\lambda\ge\underline{u}_i+\lambda$ for $\lambda\le\min\{p_i,\bar\alpha\}$. The
feasible set is compact and non-empty: if $k$ lines have $p_i>\alpha$, then
$\sum_i(p_i-\alpha)_+\le k\bar\alpha\le c\bar\alpha$ when $k\le c$, and
$\sum_i(p_i-\alpha)_+\le c-k\alpha\le c\bar\alpha$ when $k>c$. A maximiser
$\lambda^*$ therefore exists, and, since the $g_i$ are non-decreasing, it may be
chosen with $\sum_i\lambda^*_i=\min\{c\bar\alpha,\sum_i\min\{p_i,\bar\alpha\}\}$.

\emph{Attainment.} Let $\Theta,\Theta_1,\dots,\Theta_n$ be independent and uniform on
$[0,1)$, and $A=\{\Theta<\bar\alpha\}$. Lemma~\ref{lem:wraparound} with
$\omega=\bar\alpha$ and $w_i=\lambda^*_i$ yields arcs $R^A_i\subseteq[0,\bar\alpha)$
covering each point at most $\lceil\sum_i\lambda^*_i/\bar\alpha\rceil\le c$
times; with $\omega=\alpha$ and $w_i=p_i-\lambda^*_i\le\alpha$, shifted by
$\bar\alpha$, it yields arcs $R^B_i\subseteq[\bar\alpha,1)$ covering each point
at most $c$ times, because either $\sum_i\lambda^*_i=c\bar\alpha$ and
$\sum_i(p_i-\lambda^*_i)\le c-c\bar\alpha$, or $\lambda^*_i=\min\{p_i,\bar\alpha\}$
for all $i$, in which case fewer than $c$ lines have $p_i>\bar\alpha$ and each
contributes at most $\alpha$. Let $U_i$ equal $1-\lambda^*_i\Theta_i$ on
$\{\Theta\in R^A_i\}$, $1-p_i-(\bar\alpha-\lambda^*_i)\Theta_i$ on
$A\setminus\{\Theta\in R^A_i\}$, $1-\lambda^*_i-(p_i-\lambda^*_i)\Theta_i$ on
$\{\Theta\in R^B_i\}$, and $(\alpha-p_i+\lambda^*_i)\Theta_i$ otherwise. Each of the
four events has the probability of the length of the corresponding interval of
levels, on which $U_i$ is then uniform, so $U_i$ is uniform on $(0,1)$ and
$X_i=F_i^{-1}(U_i)$ has law $F_i$. Moreover, $\{X_i>l_i\}=\{U_i>1-p_i\}=
\{\Theta\in R^A_i\cup R^B_i\}$, so at most $c$ lines exceed their levels at
once, and $\E[X_i\mathds{1}_A]=g_i(\lambda^*_i)$. By the bathtub bound,
$\bar\alpha\,\mathrm{TVaR}_\alpha(S)\ge\E[S\mathds{1}_A]=\sum_ig_i(\lambda^*_i)$,
which proves~\eqref{eq:W_program} and attainment. For $c=n$ the cap is void,
$\lambda^*_i=\min\{p_i,\bar\alpha\}$, and $g_i(\lambda^*_i)=\int_\alpha^1
F_i^{-1}(u)\,\d u$ in both cases $p_i\ge\bar\alpha$ and $p_i<\bar\alpha$.

(b) Let $\mathcal{V}(\rho)$ be the maximum in~\eqref{eq:W_program} with
$c\bar\alpha$ replaced by $\rho\ge\sum_i(p_i-\alpha)_+$. It is non-decreasing,
and concave: a convex combination of maximisers for $\rho'$ and $\rho''$ is
feasible for the corresponding combination of $\rho'$ and $\rho''$, and
$\sum_ig_i$ is concave. As $\mathcal{W}_c(\alpha)=\mathcal{V}(c\bar\alpha)/\bar\alpha$,
the claim follows.

(c) If all $g_i=g$ and $\Lambda_i=\Lambda$ coincide, replacing a maximiser by the
vector whose coordinates all equal its mean preserves feasibility and, by
concavity, does not decrease $\sum_ig(\lambda_i)$; raising the common value to
$\min\{p,\bar\alpha,c\bar\alpha/n\}=\min\{p,c\bar\alpha/n\}$, which lies in
$\Lambda$ because $np\le c$, does not decrease it either, as $g$ is
non-decreasing.
\end{proof}

\begin{proof}[Proof of Proposition~\ref{prop:Gstar_LR}]
Every interval $(x,y]$ with $0\le x<y$ satisfies
$\mu_i^+((x,y])=\F_i(x)-\F_i(y)$ and
$Q_i^G((x,y])=G_i(\F_i(x))-G_i(\F_i(y))$, so~\eqref{eq:G-JE_G_star} is the
infimum of $\mu_i^+((x,y])/Q_i^G((x,y])$ over the intervals with
$Q_i^G((x,y])>0$. Hence $\eta\le G_i^*$ if and only if $\eta Q_i^G\le\mu_i^+$ on all
such intervals, and then, by monotone limits, on all open subsets of
$(0,\infty)$, as these are countable disjoint unions of open intervals. Outer
regularity of finite Borel measures extends the inequality to all Borel sets,
which proves the first identity; the second follows from the Radon--Nikodym
theorem. Finally, on each intermediate face $E_\I$ with $\I\in\M_i$, the
conditional survival function of $X_i$ is $G_i\circ\F_i$
by~\eqref{eq:G-JEProbFaces}, so
$\E[X_i\mathds{1}\{X_\N\in E_\I\}]=p_\I\int_0^\infty G_i(\F_i(x))\,\d x$;
summing over $\I\in\M_i$ gives the claim.
\end{proof}

\begin{proof}[Proof of Proposition~\ref{prop:CM_atom}]
Here $\tau=n-1$, and the weights satisfy~\eqref{eq:JE_parameters_1}
and~\eqref{eq:JE_parameters_2} with equality, so the canonical vector exists
and has exactly $n-1$ active lines (Remark~\ref{rm:JE_Canonical_Remark}(4)). On
each face $E_{\N\setminus\{k\}}$, the active severities have marginal law $F_Y$
and the survival copula of $(Z_1,\dots,Z_{n-1})$, hence the joint law of
$(Z_1,\dots,Z_{n-1})$, so their sum is $(n-1)\mu_Y$.
\end{proof}

\section{Further examples}
\label{app:example}

\begin{example}
\label{ex:no_extremes}
Theorems~\ref{thm:safest_FS} and~\ref{thm:cx_max} cannot dispense with their
homogeneity assumptions. Let $n=3$ and let $p_{ij}=\P(\mathcal{A}=\{i,j\})$ for
$i<j$. For any JE vector of the form $X_i=B_iY_i$ with $Y_1,Y_2,Y_3$
independent and independent of $B_\N$, eliminating the axis weights
via~\eqref{eq:face_balance} gives
\[
\Pi_S(d)=\sum_{i\in\N}\F_i(0)\,\Pi_{Y_i}(d)+\sum_{i<j}p_{ij}\,\Gamma_{ij}(d),
\qquad
\Gamma_{ij}(d)=\Pi_{Y_i+Y_j}(d)-\Pi_{Y_i}(d)-\Pi_{Y_j}(d)\ge0,
\]
where, for $\F_i(0)=\frac12$, $i\in\N$, the pair weights range over
$p_{ij}\ge0$ with $p_{12}+p_{13},\,p_{12}+p_{23},\,p_{13}+p_{23}\le\frac12$
and $p_{12}+p_{13}+p_{23}\ge\frac12$.
\begin{itemize}
\item[(1)] \emph{No safest structure.} Let $Y_1=Y_2=1$, and let $Y_3$ equal
$\frac12$ or $3$ with probability $\frac12$ each. At $d=1$,
$\Gamma_{12}=1>\Gamma_{13}=\Gamma_{23}=\frac34$, so $\Pi_S(1)$ is minimal if and only if
$p_{12}=0$ and $p_{13}+p_{23}=\frac12$; at $d=\frac74$,
$\Gamma_{12}=\frac14<\Gamma_{13}=\Gamma_{23}=\frac12$, so $\Pi_S(\frac74)$ is minimal if and
only if $p_{12}=\frac12$. No such vector minimises both premiums, so this class
has no minimum in convex order.
\item[(2)] \emph{No most dangerous structure.} Let $Y_1=1$ and $Y_2=Y_3=2$
deterministically, so that every random vector with these marginals is of the
above form. At $d=1$, $\Gamma_{ij}=1$ for all pairs, and $\Pi_S(1)$ is maximal if
and only if the total pair weight attains its maximum $\frac34$, that is,
$p_{12}=p_{13}=p_{23}=\frac14$; at $d=\frac52$,
$\Gamma_{12}=\Gamma_{13}=\frac12<\Gamma_{23}=\frac32$, and $\Pi_S(\frac52)$ is maximal if and
only if $p_{23}=\frac12$ and $p_{12}=p_{13}=0$. Hence the JE class with these
marginals has no maximum in convex order.
\end{itemize}
In both cases, the extremal premium for a fixed retention is a linear programme
over the pair weights, but its solution changes with the retention.
\end{example}

\begin{example}[Power transformations]
\label{ex:G-JE_power}
In the setting of Definition~\ref{def:g-JE}, let
$G_i(u)=\bigl((u-a_i)/(b_i-a_i)\bigr)^{\gamma}$ on $[a_i,b_i]$, $\gamma>0$,
extended by $0$ on $[0,a_i]$ and by $1$ on $[b_i,1]$. For values $v<u$ of $\F_i$ in $[a_i,b_i]$, the mean value
theorem gives
\[
\frac{u-v}{G_i(u)-G_i(v)}=\frac{(b_i-a_i)^\gamma(u-v)}{(u-a_i)^\gamma-(v-a_i)^\gamma}
=\frac{(b_i-a_i)^\gamma}{\gamma\,\xi^{\gamma-1}}
\quad\text{for some }\xi\in(v-a_i,u-a_i).
\]
If $\gamma\in(0,1)$ and the range of $\F_i$ accumulates at $a_i$ from above (as
it does when $\F_i$ is continuous), letting $u,v\downarrow a_i$ drives the ratio
to zero, so $G_i^*=0$ and all faces involving $i$ vanish; if this holds for
every $i$, \eqref{eq:G-JECondition} collapses to $\tau\le1$ and the G-JE
vector is ME. If $\gamma>1$, $a_i=0$, $b_i=\F_i(0)$ and the range of $\F_i$
accumulates at $\F_i(0)$ from below, the ratio is minimised as
$u,v\uparrow\F_i(0)$, giving $G_i^*=\F_i(0)/\gamma$.
Condition~\eqref{eq:G-JECondition} then becomes
\[
\tau\le\min\left\{
\frac{\gamma(n-1)}{\gamma(n-1)-(n-2)},\;
\frac{\gamma-\max_{i\in\N}\F_i(0)}{\gamma-1}
\right\},
\]
whose right-hand side decreases from $n-1$ as $\gamma\downarrow1$ (when
$\max_{i\in\N}\F_i(0)<1$) to $1$ as $\gamma\to\infty$: the more strongly the
face severities are tilted towards small losses, the closer the admissible
marginals come to those of ME.
\end{example}

\begin{example}
\label{ex:cx_max_r1}

\emph{A most dangerous structure with a large line.} When the claim
probabilities are heterogeneous, the most dangerous structure of
Theorem~\ref{thm:cx_max} can contain co-monotonic large lines. Consider $n=4$
lines with claim probabilities $q_\N=(0.95,0.35,0.30,0.20)$, so that
$\tau=1.8$, and unit-mean exponential severities. For $3$-exclusivity ($c=2$),
$t_0=0.9<q_1$ and $t_1=0.85\ge q_2$, so $r=1$ and $t=0.85$: in the most
dangerous structure, line~1 claims whenever any other line does, and lines
$2$, $3$ and $4$ partition the event $\{U>0.15\}$ of probability $0.85$ into
arcs of lengths $0.35/0.85$, $0.30/0.85$ and $0.20/0.85$, so that exactly two
lines claim on this event. By~\eqref{eq:S_sharp},
\[
S^\sharp=\log\frac{0.95}{1-U}\,\mathds{1}\{U>0.05\}
+\log\frac{0.85}{1-U}\,\mathds{1}\{U>0.15\},
\]
and $\mathrm{TVaR}_{0.99}(S^\sharp)=\log(0.95\cdot0.85)+2(1+\log100)\approx11.00$.
For JE ($c=3$), $t_0=0.6<q_1$ and $t_1=0.425\ge q_2$, so again $r=1$, now with
$t=0.425$ and pairs of the remaining lines claiming together, which gives
$\mathrm{TVaR}_{0.99}(S^\sharp)=\log(0.95\cdot0.425^2)+3(1+\log100)\approx15.05$.
Without exclusivity, the co-monotonic aggregate has
$\mathrm{TVaR}_{0.99}=\log(0.95\cdot0.35\cdot0.30\cdot0.20)+4(1+\log100)\approx18.51$.
\end{example}
\end{appendices}

% Bibliography embedded (generated with the apalike style and natbib).

%%%%%%%%%%%%%%%%%%%%%%%%%%%%%%%%%%%%%%%%%%%%%%%%%%%%%%%%%%%%%%%%%%%%%%%%%%%%%%%%%%%%%%%%%%%%%%%%%%%%%
\end{document}